\documentclass{amsart}
\usepackage{amssymb}
\usepackage{amsmath}
\usepackage{amsthm}
\usepackage{biblatex}
\usepackage{bm}
\usepackage{enumerate}
\usepackage{geometry}
\usepackage{graphicx}
\usepackage[pdftex,colorlinks,citecolor=black,linkcolor=black,urlcolor=black,bookmarks=false]{hyperref}
\usepackage{mathabx}
\usepackage{mathrsfs}
\usepackage{microtype}
\usepackage{stackrel}
\usepackage{tikz-cd}
\usepackage{adjustbox}

\newtheorem{theorem}{Theorem}[section]
\newtheorem{proposition}[theorem]{Proposition}
\newtheorem{definition}[theorem]{Definition}
\newtheorem{corollary}[theorem]{Corollary}
\newtheorem{conjecture}[theorem]{Conjecture}
\newtheorem{construction}[theorem]{Construction}
\newtheorem{lemma}[theorem]{Lemma}
\newtheorem{example}[theorem]{Example}
\theoremstyle{remark}
\newtheorem{remark}[theorem]{Remark}

\numberwithin{equation}{section}
\numberwithin{figure}{section}
\numberwithin{table}{section}

\DeclareMathOperator{\Ad}{\mathrm{Ad}}
\DeclareMathOperator{\Ann}{\mathrm{Ann}}
\DeclareMathOperator{\Bun}{\mathrm{Bun}}
\DeclareMathOperator{\Cht}{\mathrm{Cht}}

\DeclareMathOperator{\FinS}{\mathrm{FinS}}
\DeclareMathOperator{\Fl}{{\mathcal{F}\ell}}
\DeclareMathOperator{\Frob}{\mathrm{Frob}}
\DeclareMathOperator{\Gal}{\mathrm{Gal}}
\DeclareMathOperator{\GL}{\mathop{GL}}
\DeclareMathOperator{\Gr}{\mathrm{Gr}}

\DeclareMathOperator{\Hom}{\mathrm{Hom}}

\DeclareMathOperator{\rec}{\mathrm{rec}}
\DeclareMathOperator{\Reg}{\mathrm{Reg}}
\DeclareMathOperator{\Rep}{\mathrm{Rep}}
\DeclareMathOperator{\Res}{\mathrm{Res}}

\DeclareMathOperator{\Spec}{\mathrm{Spec}}

\DeclareMathOperator{\triv}{\mathrm{triv}}
\DeclareMathOperator{\Weil}{\mathrm{Weil}}

\newcommand{\atimes}{\stackbin{\leftarrow}{\times}}

\newcommand{\qlbar}{\overline{\mathbb{Q}}_l}
\newcommand{\F}{\mathbb{F}}
\newcommand{\fqbar}{\overline{\mathbb{F}}_q}
\newcommand{\btilde}{\widetilde{\mathcal{B}}}
\newcommand{\fltilde}{\widetilde{\Fl}}
\newcommand{\wtilde}{\widetilde{W}}

\title{Restricted shtukas and $\Psi$-factorizable sheaves}
\author{Andrew Salmon}
\email{asalmon@mit.edu}
\address{Department of Mathematics\\
Massachusetts Institute of Technology\\
Cambridge, MA 02139\\ USA}

\begin{document}

\begin{abstract}
    We show that Lusztig's theories of two-sided cells and non-unipotent representations of a reductive group over a finite field are compatible with the V. Lafforgue's automorphic-to-galois direction of the Langlands correspondence.  To do this, we extend cases where nearby cycles commutes with pushforward from sheaves on the moduli space of shtukas to a product of curves to include certain depth $0$ cases.  More generally, we introduce the notion of $\Psi$-factorizability to study nearby cycles over general bases, whereby a sheaf is $\Psi$-factorizable if its nearby cycles are the same as iterated nearby cycles with respect to arbitrary compositions of specializations on the base.  The Satake sheaves on Beilinson-Drinfeld grassmannians and their cohomology sheaves on curves are nontrivial examples of $\Psi$-factorizable sheaves.  This notion allows us to adapt arguments from Xue \cite{xue2020smoothness}.  As an application, for automorphic forms in depth zero attached to a Langlands parameter, we characterize the image of the tame generator of this parameter in terms of semisimple orbits and two-sided cells attached to representations, extending ideas of Lusztig-Yun and Bezrukavnikov-Finkelberg-Ostrik.
\end{abstract}

\maketitle

\section{Introduction}

For a reductive group $G$, a $G$-shtuka is a principal $G$-bundle on a curve together with a Hecke modification relating the $G$-bundle to its Frobenius pullback.  While the early study of shtukas applied to the Langlands correspondence for $\GL_n$ used the Arthur trace formula applied to simple modification types, such as in the work of L.~Lafforgue \cite{lafforgue2002chtoucas} and Drinfeld, arbitrary modification types have been used since the work of Varshavsky \cite{varshavsky2004moduli}.  Notably, V.~Lafforgue introduced an automorphic-to-Galois direction of the Langlands correspondence for any split reductive group $G$.  For $G$ a split reductive group over a function field $K$, the reciprocity map gives
\[
\begin{tikzcd}
\rec \colon \left\{ \begin{gathered}\mbox{cuspidal automorphic} \\ \mbox{representations} \\
\pi \subset L^2(G(K) \backslash G(\mathbb{A})) \\ \mbox{with finite central character} \end{gathered} \right\} \arrow[r] & \left\{ \begin{gathered} \mbox{semisimple} \\ \widehat{G}(\qlbar)\mbox{-valued representations} \\ \text{of} \Weil(\overline{K} / K). \end{gathered} \right\}
\end{tikzcd}
\]
Automorphic representations include representations of reductive groups over local fields at each place of the function field $K$, and $\Weil$ group representations restrict to local Galois groups at each place.  It is expected that there should be a compatibility between these representations.  For example, reductive groups over local fields have a Moy-Prasad filtration at any parahoric subgroup, while local Galois groups have an upper numbering filtration, and it is expected that these filtrations are compatible with each other.  In depth 0, the local Galois groups are expected to factor over the tame fundamental group, and the image of the tame generator should be related to two-sided cells and semisimple parameters.  In fact, in some elliptic depth zero cases, we will give a complete description of the image of the tame generator in Corollary~\ref{cor:Langlands param}.

We need the following preliminaries.  Let $G$ be a split reductive group and $\pi$ be a compactly supported cuspidal automorphic form for $G(\mathbb{A}_K)$ that is attached to the semisimple Langlands parameter $\rho$ under reciprocity.  Let $C$ be the smooth curve over $\mathbb{F}_q$ with function field $K$, and let $|C|$ denote its places.  $G(\mathbb{A}_K)$ contains an open compact subgroup $\prod_{x \in |C|} G_x \subset G(\mathbb{A}_K)$ so that the group $G_x$ encodes some ramification at the closed point $x$.  Let $y \in |C|$ be a fixed place corresponding to an $\mathbb{F}_q$-point that is unramified, let $G_y$ be the standard maximal parahoric at $y$, and let $G^+_y$ be the first congruence subgroup with Levi quotient $G_y / G^+_y \cong G(\mathbb{F}_{q})$.

For a fixed open compact subgroup $U^+ = \prod_x H_x \subset \prod_x G_x$ with level structure $H_y = G^+_y$ at $y$, and $U = \prod_{x \ne y} H_x \times G_y$, we get a map:
\[ G(K) \backslash G(\mathbb{A}_K) / U^+ \to G(K) \backslash G(\mathbb{A}_K) / U \]
that is a Galois $G(\mathbb{F}_{q})$-cover.  Therefore, the space of compactly supported functions has an isotypic decomposition according to the regular representation $\Reg(G(\mathbb{F}_q))$ of this finite group
\[ H_{1} \cong C_c(G(K) \backslash G(\mathbb{A}_K) / U^+) \cong C_c(G(K) \backslash G(\mathbb{A}_K) / U, \Reg_{G(\mathbb{F}_q)}) \cong \bigoplus_{\zeta \in \Rep(G(\mathbb{F}_q))} H_{1,\zeta} \otimes \underline{\zeta^*}. \]
The spaces of functions $H_{1,\zeta}$ for $\zeta \in \Rep(G(\mathbb{F}_{q}))$ irreducible can be realized as a space of automorphic forms, given a vector $0 \ne v \in \underline{\zeta^*}$.  To irreducible $\zeta$ Lusztig has attached a $W$-orbit of semisimple parameters $\lambda \colon T \to \qlbar^{\times}$ and a two-sided cell $\textbf{c}$ of the affine Weyl group $\wtilde_{\lambda}$ fixing $\lambda$.  We will see in Section~\ref{sec:horocycle} that once we make a choice of tame generator, a pair of semisimple parameter and two-sided cell determines a conjugacy class in $\widehat{G}$, which we denote $\mathfrak{o}_{\zeta}$.

On the Galois side we have local Galois group $\Gal(\overline{K_x} / K_x) \subset \Gal(\overline{K} / K)$ at each place.  Each local Galois group has a ramification filtration whose first steps are given by
\[ P \subset I \subset \Gal(\overline{K_y} / K_y), \]
where $P$ is wild inertia, a Sylow-p subgroup, and $I$ is inertia.  The quotient $I / P$ is topologically generated by a single element $\gamma$, which we fix and call the tame generator.
\begin{theorem}\label{thm:depth0}
Let our curve $C$, reductive group $G$ be chosen as above, and suppose furthermore that $G$ has connected center.  Let $y \in C(\mathbb{F}_q)$ be a place and let $\pi$ be an automorphic form in an isotypic component $H_{1}[\zeta] \subset H_{1}$, where the irreducible representation $\zeta$ is attached to the conjugacy class $\mathfrak{o}_{\zeta}$.  Furthermore, suppose $\pi$ is attached to the Langlands parameter $\rho$ under the reciprocity map.  Let $\gamma$ be a choice of tame generator inside local inertia mod wild inertia $I / P$ at the point $y$.  Then $\rho(P) = 1$ and thus $\rho(\gamma)$ is well-defined.  Furthermore, $\rho(\gamma) \in \overline{\mathfrak{o}_{\zeta}}$.

In the case that $\rho$ is an elliptic Langlands parameter, we have $\rho(\gamma) \in \mathfrak{o}_{\zeta}$.
\end{theorem}

At a coarse level, the method of proof is to connect the theorem above to statements about the action of the excursion algebra and its framed variant acting on the cohomology of moduli spaces of shtukas.  To see the connection, we note that automorphic forms as we have described above are a subspace of $\qlbar$-valued functions on the discrete groupoid $\Bun_{G,U}(\mathbb{F}_q)$ of $G$-bundles with level structure given by $U$ such that the center acts via a finite order central character.  Functions on a discrete groupoid are the same as cohomology of the underlying 0-dimensional stack.  This stack can be taken to be a stack of shtukas with trivial modifications and the Galois covering given by level structure is in fact a Galois covering that appears for stacks of shtukas with modifications as well.  The cohomology of shtukas are related to each other by excursion operators, and the action of all excursion operators gives V. Lafforgue's automorphic to Galois reciprocity map described above.

To understand the behavior of excursion operators for the local Galois group, we will reinterpret the cohomology locally around a point $y$ as the cohomology of a nearby cycles sheaf and the action of monodromy on this nearby cycles sheaf.  In the depth 0 case, such nearby cycles sheaves will turn out to be related to monodromic variations of Gaitsgory's nearby cycles sheaves, and the action of monodromy will be arise from a horocycle correspondence which can be viewed as an affine analogue of the the horocycle correspondence that Lusztig uses to study character sheaves \cite{lusztig2015non} and non-unipotent representations of $G(\F_q)$ \cite{lusztig2016non}, thus making the connection to two-sided cells.

One obstacle to the study of shtukas is the lack of a good compactification in the general case.  One complication to the study of compactifications is that the moduli space of shtukas is not of finite type in general unless shtukas are truncated using the Harder-Narasimhan filtration.  But even in finite type cases (such as shtukas for an inner form, like a division algebra), the moduli space is known to be noncompact if the Hecke modification type is sufficiently large \cite{lau2007degenerations}.  Moreover, for deeper level structures, the singularities of proper compactifications studied by L.~Lafforgue seem challenging to control, whereas it is desirable to construct compactifications whose singularities are no worse than the singularities of the Beilinson-Drinfeld grassmannian corresponding to the Hecke modification used.

Instead, we will focus on techniques that bypass the need for constructing a compactification, following Xue \cite{xue2020smoothness}, who proved the smoothness of the cohomology of shtukas as a sheaf on the base curve without constructing a compactification.  The key technique, following \cite{salmon2023unipotent} which handled the parahoric case, is to establish the appropriate fusion structures on the nearby cycles over the relevant complexes of sheaves.  Then, using Zorro's lemma, we can construct an inverse to the natural map
\begin{equation}\label{eqn: canonical nearby cycle map} \mathrm{can} \colon R\mathfrak{p}_! \Psi_J \mathscr{F}_{I \cup J, W \boxtimes V, \zeta} \to \Psi_J R\mathfrak{p}_! \mathscr{F}_{I \cup J, V \boxtimes W, \zeta}. \end{equation}
Here the map $\mathfrak{p}$ sends an $S$-point of the stack of shtukas to a tuple of points on a power of a curve $C^{I \cup J}$.  While isomorphisms of this form are immediate if $\mathfrak{p}$ is proper, $\mathfrak{p}$ will fail to be proper in almost all cases of interest.  Here, $\Psi_J$ can be understood as nearby cycles over general bases from a geometric generic point $\overline{\eta_J}$ of $C^J$ to an $\fqbar$-point $\overline{s_J}$.  If $J = \{ j_1, \dots, j_k \}$, we may also consider an iterated nearby cycle $\Psi_{j_1} \cdots \Psi_{j_k}$ where $\Psi_{j_m}$ specializes from a geometric generic point $\overline{\eta}$ of $C$ to $\overline{s_m}$, the projection of $\overline{s_J}$ to the $j_m$ factor of the curve.

The sheaves on shtukas that we work with $\mathscr{F}_{I, V, \zeta}$ will live in the bounded derived category of constructible sheaves on shtukas with legs indexed by $I$ at the \emph{parahoric} level, but they encode information about a piece of the cohomology of shtukas at deeper level.  To explain this, let us consider a special case of congruence subgroups at a single point $y \in C(\mathbb{F}_q)$ of the curve.  Consider shtukas with level structure given by a divisor $N = d^{-} \cdot y$, where the $-$ indicates that our convention will differ slightly from what we use later in Definition~\ref{def: shtukas N level}.  The stack $\Cht_{G,N,I,V}$ classifies shtukas with Hecke modification bounded by $V$ and a level structure over the divisor $d \cdot y$.  The IC sheaf on this stack pushes forward along the \'{e}tale Galois map $c_N \colon \Cht_{G,N,I,V} \to \Cht_{G,I,V}$, and such a pushforward $(c_N)_! \mathrm{IC}_V$ breaks up into isotypic components corresponding to the regular representation of the finite group $H(\mathbb{F}_q) = (L^+G / (t^d))(\mathbb{F}_q)$ where $L^+G$ is the positive loop group formed by Weil restriction along the formal disk around $y$.  The isotypic summands will correspond to sheaves on the classifying stack $BH(\mathbb{F}_q)$.  It turns out the shtukas at the parahoric level admit a natural smooth map to restricted shtukas.  In our special case, an isotypic summand of the cohomology with level structure $N$ is in fact the cohomology of a sheaf $\mathscr{F}_{I,V,\zeta}$ that lives over shtukas the spherical level.  It is to this sheaf on $\Cht_{G,I,V}$ that we consider the nearby cycles.  The following theorem can be seen by combining the results of Theorem~\ref{thm:xuezorro} and Theorem~\ref{thm: psi-factorizable GFq}.
\begin{theorem}
    Let $N = \sum_{i} a_i y_i$ be a level structure with $a_1 = 1$ and $y_1 = y \in C(\F_q)$ and let $\zeta \in \Rep(H(\F_q))$.  Write $N = y + N'$ with $N'$ not containing $y$ in its support and let $U = C \setminus N'$ be a Zariski-open subset of $C$.  Let $\overline{\eta_I} \to \overline{s}$ be a specialization in the \'etale topos of $U^I \subset C^I$.  Then the map \eqref{eqn: canonical nearby cycle map} is an isomorphism.  Moreover, $\Psi_J$ can be taken to be sliced nearby cycles over general bases or iterated nearby cycles with respect to the legs $J$, and all such choices of nearby cycles are independent of the permutation $J = \{ j_1, \dots, j_k \}$.  That is, all arrows in the diagram
    \begin{equation}
        \begin{tikzcd}
            R\mathfrak{p}_! \Psi_J \mathscr{F}_{I \cup J, W \boxtimes V, \zeta} \arrow{r} \arrow{d}& \Psi_J R\mathfrak{p}_! \mathscr{F}_{I \cup J, V \boxtimes W, \zeta} \arrow{d} \\
            R\mathfrak{p}_! \Psi_{j_1} \cdots \Psi_{j_k} \mathscr{F}_{I \cup J, W \boxtimes V, \zeta} \arrow{r}& \Psi_{j_1} \cdots \Psi_{j_k} R\mathfrak{p}_! \mathscr{F}_{I \cup J, V \boxtimes W, \zeta}
        \end{tikzcd}
    \end{equation}
    are isomorphisms.
\end{theorem}

In Section~\ref{sec:nearby cycles}, we review the nearby cycles formalism and explain how it applies to our setting.  We also introduce the notion of $\Psi$-factorizability, which is related to the K\"unneth formula studied by Gabber, and then by Illusie and Zheng for nearby cycles over general bases \cite{illusie2017around}.  $\Psi$-factorizability is the key technical feature that the sheaves on shtukas and their pushforwards to curves satisfy that allows Xue's ``Zorro lemma'' argument \cite{xue2020smoothness} to run.

In Section~\ref{sec:shtukas}, we survey properties of reductive groups, Beilinson-Drinfeld grassmannians, and global shtukas.  Our perspective combines that of Xue \cite{xue2020smoothness} and Hemo-Richarz-Scholbach \cite{hemo2020constructible} to prove $\Psi$-factorizability for the direct limit of certain cocartesian functors.  We will later see that examples are given by cohomology sheaves of shtukas, as well as cohomology sheaves of their nearby cycles, at fixed level structure given by parahoric and Moy-Prasad subgroups at a set of closed points.  In this setting, $\Psi$-factorizability here is contained in Drinfeld's lemma.

In Section~\ref{sec:fusion}, we investigate $\Psi$-factorizability on the sheaves $\mathscr{F}_{I,V,\zeta}$, before pushing forward.  This property is closely related to the centrality of nearby cycles functors into Hecke categories generalizing Gaitsgory's geometric construction of central elements of the affine Hecke algebra \cite{gaitsgory1999construction}.  As an application, in cases where the $\Psi$-factorizability of $\mathscr{F}$ is known, we can run Xue's ``Zorro lemma'' argument using nearby cycles instead of specializations to prove that nearby cycles commutes with pushforward to the power of the curve.

In Section~\ref{sec:horocycle}, we restrict to the setting of the principal level subgroup.  We construct a horocycle correspondence for restricted shtukas for principal level subgroups at an $\mathbb{F}_q$-rational point that relates these restricted shtukas to the affine flag variety on the special fiber and recovers the usual horocycle correspondence between $G$ and $(G / U)^2$ on the generic fiber.  At this level, we show that $\Psi$-factorizability of the sheaves $\mathscr{F}$ can be proven using the horocycle correspondence.  We also extend some results of Bezrukavnikov \cite{bezrukavnikov2004tensor} and Bezrukavnikov-Finkelberg-Ostrik \cite{bezrukavnikov2009tensor}.

Finally in Section~\ref{sec:depth 0}, we give applications to automorphic forms by computing the monodromy action of nearby cycles over restricted shtukas, proving Theorem~\ref{thm:depth0}.  Using the results of the previous section, we identify framed excursion operators with the action of specific elements in $\widehat{G}$.  We develop the machinery of framed excursion operators introduced in \cite{lafforgue2018d} to relate the semisimple and unipotent parts of the image of the tame generator of a Langlands parameter to the isotypic component of a residual $G(\mathbb{F}_q)$-representation in which a certain automorphic form lives, obtaining particularly complete results in the case of elliptic Langlands parameters.

\subsection*{Acknowledgements}

Many of the ideas in this document arose in discussions with Jianqiao Xia as well as my advisor, Zhiwei Yun, and they offered many clarifying explanations and suggestions as this document was being prepared.

In the preparation of this paper, we learned that Arnaud Eteve has independently studied the semisimple part of the image of the tame generator as in Theorem~\ref{thm:depth0} for depth $0$ representations.  We would like to thank him sharing a copy of the preprint \cite{eteve}.

I would like to thank Alain Genestier and Vincent Lafforgue for identifying a serious gap in a preliminary version of this document and for suggesting Definition~\ref{def:grgtrivmax}.
I would like to thank Patrick Bieker for explaining an error in an earlier version of this document.

\tableofcontents

\section{Nearby cycles formalism}\label{sec:nearby cycles}

At parahoric level structures, the Langlands correspondence produces local Galois representations that factor over the tame quotient and such that the tame generator acts unipotently.  At deeper level structures, neither of these facts are expected to hold, so to adapt the nearby cycles techniques of \cite[Section~4]{salmon2023unipotent}, we will need to use nearby cycles over general bases to get nontrivial constructions.  The theory of nearby cycles over general bases dates back to Laumon and Deligne \cite{laumon1983vanishing} and was developed by Orgogozo \cite{orgogozo2006modifications}.  For a general exposition of oriented toposes, see also \cite[Expos\'{e}~XI]{illusie2014travaux}.

\subsection{Review of nearby cycles over general bases}

\begin{definition}
Let $f \colon X \to S$ be a map of toposes.  The oriented product of toposes $X \atimes_S S$ is the universal topos which admits projections $p_1 \colon X \atimes_S S \to X$, $p_2 \colon X \atimes_S S \to S$, and a $2$-morphism $\tau \colon f p_1 \to p_2$.  The oriented topos is naturally equipped with a morphism $\Psi_f \colon X \to X \atimes_S S$.  If $K$ is a sheaf on $X$, then the direct image $R\Psi_{f} K = R(\Psi_f)_* K$ is called the nearby cycles sheaf.
\end{definition}

Now let $X$ be a scheme and let $\Lambda$ be a coefficient ring.  If $\Lambda$ is torsion, we may consider the derived category $D(X, \Lambda)$ of \'{e}tale sheaves, $D^+(X, \Lambda)$ of bounded below complexes of \'{e}tale sheaves, and $D^b(X, \Lambda)$ of bounded complexes.  Similarly, if $\Lambda$ is an finite extension of $\mathbb{Q}_\ell$ or $\overline{\mathbb{Q}}_{\ell}$, we may consider $D(X, \Lambda)$ as the derived category of pro-\'{e}tale sheaves in the pro-\'{e}tale topology.  There is also the subcategory $D_c(X, \Lambda)$ of constructible sheaves.  As a convention, all notions in this section will be derived, so when we say that $K$ is constructible, we really mean it is a constructible complex of sheaves, and so on.

We suppose $f \colon X \to S$ is a morphism of schemes and $K$ is a sheaf of complex of sheaves in $D(X, \Lambda)$ and we identify the notation $X$ and $S$ with their \'{e}tale, resp. pro-\'{e}tale toposes.  For the applications to shtukas, we may take these schemes to live over $\overline{\mathbb{F}_q}$, although none of our constructions we give will depend on this.  We say that $(f,K)$ is $\Psi$-good if $R\Psi_f K$ commutes with base change along any cartesian square given by maps $h' \colon X' \to X$ over $h \colon S' \to S$.  For this notion, see \cite[Definition~1.5]{illusie2017around}.  The property of being $\Psi$-good is not guaranteed for sheaves, and only guaranteed up to modification of the base, as the main result of \cite{orgogozo2006modifications} shows:

\begin{theorem}
Let $f \colon X \to S$ and $K$ be a constructible $\Lambda$-sheaf with $S$ noetherian and $f$ of finite type.  Then there exists some modification $h \colon S' \to S$ such that if $f' \colon X \times_S S' \to S'$, $p_1 \colon X \times_S S' \to X$, and $K' = p_1^* K$, then $R \Psi_{f'} K'$ is constructible and $(f',K')$ is $\Psi$-good.
\end{theorem}

In particular, over a regular $1$-dimensional scheme $S$, any pair $(f,K)$ as above will be $\Psi$-good.  This construction recovers the usual nearby cycles functor.

We will want to verify that certain Satake sheaves on spaces similar to the Beilinson-Drinfeld grassmannian are $\Psi$-good.  Our constructions will not use Orgogozo's theorem except in the $1$-dimensional case, and rather reduce everything to the convolution product.  However, we note that the property of being $\Psi$-good is an important hypothesis in the following K\"unneth formula of \cite[Theorem~2.3 and Corollary~2.4, also Theorem~A.3]{illusie2017around}.

\begin{proposition}\label{kunneth}
Assume $f_1 \colon X_1 \to Y_1$ and $f_2 \colon X_2 \to Y_2$ are maps over $S$ and $K_i$ are constructible and finite tor-dimension in $D(X_i, \Lambda)$.  Suppose $(f_i, K_i)$ are $\Psi$-good.  Then $(f,K = K_1 \boxtimes^L K_2)$ is $\Psi$-good and there is a natural isomorphism $R\Psi_{f_1} K_1 \boxtimes^L R\Psi_{f_2} K_2 \cong R\Psi_f K$.
\end{proposition}

For a geometric point $s$ of $S$, we recall that the neighborhood $S_{(s)}$ is the localization of $s$ in $S$.  We write $X_{(s)} = X \times_S S_{(s)}$.  Any specialization of points $t \to s$ gives a map of neighborhoods $S_{(t)} \to S_{(s)}$.

The sliced and shredded nearby cycles functors recover the familiar functor from the general fiber to the special fiber of $X$ \cite[Section~1.3]{illusie2017around}.

\begin{definition}
Let $s,t \to S$ be two geometric points of $S$ and let $c \colon t \to s$ be a specialization map.  Let $j_t^s \colon X_{(t)} \to X_{(s)}$ and $i^s \colon X_s \to X_{(s)}$.  Then the shredded nearby cycles functor (called sliced nearby cycles in \cite{illusie2017around} and \cite{orgogozo2006modifications}) is
\begin{equation}
    R (\Psi_f)_t^s = (i^s)^* R(j_t^s)_* \colon D^+(X_{(t)},\Lambda) \to D^+(X_s,\Lambda).
\end{equation}
When it does not create any ambiguity, we will also use this notation to denote the functor $R(\Psi_f)_t^s (i^t)_* \colon D^+(X_t, \Lambda) \to D^+(X_s, \Lambda)$.
\end{definition}

In what follows, we will assume that $\Lambda$ is a torsion $\mathbb{F}_\ell$-algebra or $\ell$-adic with $\ell$ prime to the characteristic.  The nearby cycles functors satisfy the following base change properties.

\begin{proposition}\label{prop:basic functorialities}
Let $\pi \colon X' \to X$ be a separated, finite type morphism between coherent schemes over a base $S$.  Let $s,t \to S$ be two geometric points of $S$ and let $c \colon t \to s$ be a specialization map.  Then there is a natural transformation
\begin{equation}
    R\pi_! R\Psi_t^s \rightarrow R\Psi_t^s R\pi_!
\end{equation}
that is an isomorphism if $\pi$ is proper.

Let $g \colon X' \to X$ be a morphism between schemes over $S$ under the same assumptions as above.  Then there is a natural transformation
\begin{equation}
    g^* R\Psi_t^s \rightarrow R\Psi_t^s g^*
\end{equation}
that is an isomorphism if $g$ is locally of finite type and smooth.
\end{proposition}

\begin{proof}
The statement about pullback follows by \cite[Equation~4.6]{lu2019duality}, noting that smooth morphisms are universally locally acyclic and the isomorphism holds for locally acyclic morphisms locally of finite type.  The statement about pushforward follows by \cite[Equation~4.8]{lu2019duality}.  The statements in \cite{lu2019duality} only give these results for torsion coefficients, but the results extend without difficulty to apply to the pro-\'{e}tale topology as well to handle coefficients in $\mathbb{Z}_\ell$, $\mathbb{Q}_\ell$, etc.
\end{proof}

\subsection{$\Psi$-factorizability}

\subsubsection{A construction involving iterated nearby cycles}

The following basic construction will establish a correspondence between various iterated nearby cycles functors.

\begin{construction}\label{map to iterated}
Let $f \colon X \to S$ be a map, and let $s$, $t$, and $u$ be three points of $S$ with specializations $u \to t \to s$.  Let $j_u^t \colon X_{(u)} \to X_{(t)}$, similarly for the other neighborhoods, and $i^t \colon X_t \to X_{(t)}$, $i^s \colon X_s \to X_{(s)}$ the inclusions of fibers.  This gives an adjunction
\begin{equation}
    (i^s)^*R(j_t^s)_*R(j_u^t)_* \to (i^s)^*R(j_t^s)_*R(i^t)_*(i^t)^*R(j_u^t)_*.
\end{equation}
Recall that the shredded nearby cycles gives a functor
\[ (\Psi_f)_{u}^{s} = (i_s)^* R(j_u^s)_* \colon D^+(X_{u}, \Lambda) \to D^+(X_{s}, \Lambda) \]
so the above adjunction can be read as a natural transformation of shredded nearby cycle functors
\begin{equation}
R(\Psi_f)_u^s \to R(\Psi_f)_t^s R(i_t)_* R(\Psi_f)_u^t \colon D^+(X_{(u)}, \Lambda) \to D^+(X_s, \Lambda).
\end{equation}
\end{construction}

By iterating this construction we get for any two choices of path of specializations $t \to s$ a pair of maps relating the iterated shredded nearby cycles functors corresponding to the two choices of paths.

\subsubsection{$\Psi$-factorizability for curves}

\begin{definition}
Let $f \colon X \to S$ be a structure map.  We say that $(f, K)$ of $f$ and a constructible sheaf $K$ is $\Psi$-factorizable if $(f, K)$ is $\Psi$-good and for any composition $u \to t \to s$ of specializations, the natural map
\[ R(\Psi_f)_u^s K \to R(\Psi_f)_t^s R(i_t)_* R(\Psi_f)_u^t K \]
is an isomorphism.
\end{definition}

We now focus on building examples of $\Psi$-factorizable pairs.  The first basic example is that of the $1$-dimensional case, where all constructible sheaves are already $\Psi$-good.  Recall that a trait is the spectrum of a Henselian local ring.
\begin{proposition}
    If $S$ is a trait and $K$ is constructible, $(f, K)$ is $\Psi$-factorizable.
\end{proposition}

\begin{proof}
    Over a trait, we say that a geometric point is generic if it lies over the geometric generic point.  For a generic point $\eta$, the map $i_{\eta}$ is the identity.  We say that a geometric point is special if it lies over the closed point, and for a special point $j_{s}$ is the identity.  There are no specialization maps from a special point to a generic point.  Let $u \to t \to s$ be a composition of specializations.  We will do a case-by-case analysis of the type of each point.  If $u$ and $s$ are both generic, then let $a_1 \colon t \to s$ and $a_2 \colon u \to t$.  Since $i_s$ and $i_t$ are the identity, the natural map is identified with
    \[ R(\Psi_f)_{a_1 \circ a_2} \cong a_1^* a_2^* \cong R(\Psi_f)_{a_1} R(i_{t})_* R(\Psi_f)_{a_2}. \]
    Similarly, if $u$ and $s$ are both special, we get maps $b_1 \colon X_s \to X_t$ and $b_2 \colon X_t \to X_u$ of the special fiber, and
    \[ R(\Psi_f)_{b_i} \cong b_i^* i_s^*, \]
    so
    \[ R(\Psi_f)_{b_1 \circ b_2} \cong b_1^* b_2^* (i_s)_* \cong R(\Psi_f)_{b_1} R(i_{s})_* R(\Psi_f)_{b_2} \]
    is the base change map giving Construction~\ref{map to iterated}.  So we need to check the result when $u$ is generic and $s$ is special.  Write $c$ for the specialization from a generic point to a special point.  If $t$ is generic, write $a \colon u \to t$ the generic-to-generic specialization, so we have
    \[ R(\Psi_f)_{c \circ a} \cong (i_s)^* R(j_{\eta})_* a^* \cong R(\Psi_f)_{c} R(i_{\eta})_* R(\Psi_f)_{a}. \]
    If $t$ is special, write $b \colon X_s \to X_t$, so we have
    \[ R(\Psi_f)_{b \circ c} \cong b^* (i_s)^* R(j_{\eta})_* \cong R(\Psi_f)_{b} (i_s)_* R(\Psi_f)_{c}. \]
\end{proof}

Having handled the basic example of a trait, we can extend this to any smooth curve.

\begin{corollary}
    If $S$ is a smooth curve and $K$ is constructible, $(f, K)$ is $\Psi$-factorizable.
\end{corollary}

\begin{proof}
    Tautologically, a pair $(f, K)$ is $\Psi$-factorizable if it is $\Psi$-factorizable with respect to all localizations at all points.  In the case of a smooth curve, the localizations at any geometric point are a single geometric generic point or a trait.
\end{proof}

\subsubsection{The K\"unneth formula revisited}

The technique to produce higher-dimensional examples of $\Psi$-factorizable sheaves relies on the following reformulation of the K\"unneth formula.

\begin{proposition}\label{prop: kunneth factorizable}
    Let $f_i \colon X_i \to S_i$ be finite type morphisms between noetherian schemes, and let $(f_i, K_i)$ be $\Psi$-factorizable for $i \in \{ 1,2 \}$.  Then $(f_1 \times f_2, K_1 \boxtimes K_2)$ is $\Psi$-factorizable.
\end{proposition}

\begin{proof}
    Write $f = f_1 \times f_2$, $K = K_1 \boxtimes K_2$, $S = S_1 \times S_2$, and $X = X_1 \times X_2$.  The statement that $(f,K)$ is $\Psi$-good follows from \cite[Corollary~2.4]{illusie2017around}.
    
    Let $p_i \colon S \to S_i$ be projections and similarly let $q_i \colon X \to X_i$ be projections.  Then in the diagram,
    \begin{equation}
        \begin{tikzcd}
            X \arrow{r}{p_i} \arrow{d}{\Psi}& X_i \arrow{d}{\Psi} \\
            X \atimes_{S} S \arrow{r}{\overleftarrow{p_i}}& X_i \atimes_{S_i} S_i
        \end{tikzcd}
    \end{equation}
    there is a natural map $\overleftarrow{p_i}^* R\Psi_{f_i} \to R\Psi_f p_i^*$, which is an isomorphism when applied to a sheaf that is $\Psi$-good \cite[Equation~(2.1.4)]{illusie2017around}.  If $t \to s$ is a specialization in $S_1 \times S_2$, $p_i(t) \to p_i(s)$ is a specialization in $S_i$, and on the level of slices, this isomorphism gives
    \[ p_i^* R(\Psi_{f_i})_{p_i(t)}^{p_i(s)} K_i \cong R(\Psi_f)_t^s p_i^* K_i. \]
    Using the fact that the map
    \[ R\Psi_f p_1^* K_1 \otimes R\Psi_f p_2^* K_2 \to R\Psi_f K \]
    is an isomorphism and taking slices, we arrive at the statement in Proposition~\ref{kunneth}, that
    \[ R(\Psi_{f_1 \times f_2})_t^s(K_1 \boxtimes K_2) \cong R(\Psi_{f_1})_{p_1(t)}^{p_1(s)} K_1 \boxtimes R(\Psi_{f_2})_{p_2(t)}^{p_2(s)} K_2 \]
    is an isomorphism.
    
    Let $u \to t \to s$ be a composition of specializations of geometric points in $S_1 \times S_2$.  Then $p_i(u) \to p_i(t) \to p_i(s)$ is a composition of specializations in $S_i$.  Consider the diagram in Figure~\ref{fig:kunneth diagram}
    \begin{figure}
        \begin{tikzcd}
            R(\Psi_f)_u^s K \arrow{r} \arrow{d}& R(\Psi_{f_1})_{p_1(u)}^{p_1(s)} K_1 \boxtimes R(\Psi_{f_2})_{p_2(u)}^{p_2(s)} K_2 \arrow{d} \\
            R(\Psi_f)_t^s R(\Psi_f)_u^t K \arrow{r}& R(\Psi_{f_1})_{p_1(t)}^{p_1(s)} R(\Psi_{f_1})_{p_1(u)}^{p_1(t)} K_1 \boxtimes R(\Psi_{f_2})_{p_2(t)}^{p_2(s)} R(\Psi_{f_2})_{p_2(u)}^{p_2(t)} K_2.
        \end{tikzcd}
    \caption{Diagram giving $\Psi$-factorizability of $K_1 \boxtimes K_2$.  Horizontal arrows come from base change diagrams using the $\Psi$-goodness of $K_i$, while the vertical arrows express the $\Psi$-factorizability property on the left in terms of two $\Psi$-factorizability propertes on the right.}\label{fig:kunneth diagram}
    \end{figure}
    Since the horizontal arrows are isomorphisms, and the right vertical arrow is an isomorphism by the hypothesis of $\Psi$-factorizability of $K_1$ and $K_2$, we conclude that the left vertical arrow is an isomorphism.  Since the composition $u \to t \to s$ was arbitrary, we conclude that $(f, K)$ is $\Psi$-factorizable.
\end{proof}

In the examples of interest, we will have external product sheaves defined on the fiber product $X \times_S U$ for an open subset $U = U_1 \times \dots \times U_n$ with $U_i \subset S_i$.  We note that $!$-pushforwards will be themselves external product sheaves to which Proposition~\ref{prop: kunneth factorizable} applies.

\begin{lemma}
    Let $f_i \colon X_i \to S_i$ be finite type morphisms between noetherian schemes, and let $j_i \colon U_i \to S_i$ also be finite type, and let all schemes be separated.  Let $K_i$ be constructible sheaves on $X_i \times_{S_i} U_i$.  Define $X$, $U$, and $S$ as the products $\times_i X_i$, $\times_i U_i$, and $\times_i S_i$, and define $j \colon X \times_S U \to X$ as the pullback of the map $U \to S$.  Then
    \[ j_!(K_1 \boxtimes \dots \boxtimes K_n) \cong (j_1)_! K_1 \boxtimes \dots \boxtimes (j_n)_! K_n. \]
\end{lemma}

\begin{proof}
    By induction, it suffices to prove the $n = 2$ case.  But this follows by proper base change,
    \[ j_!p_i^* K_i \cong p_i^*(j_1)_! K_i \]
    and applying tensor product.
\end{proof}

The following corollary now follows immediately.
\begin{corollary}\label{corr: Kunneth factorizable}
    In the setting above, if $S_i$ are all curves, then $(f, j_!(K_1 \boxtimes \dots \boxtimes K_n))$ is $\Psi$-factorizable.
\end{corollary}

\subsubsection{Compatibilities under pushforward and pullback}

\begin{proposition}\label{prop: psi factorizability under pushforward}
    If $K$ is $\pi_! L$, $L$ is $\Psi$-factorizable, and $\pi$ is proper, then $K$ is $\Psi$-factorizable.
\end{proposition}

\begin{proof}
First, we check that the $\Psi$-good property is preserved.  Let $\pi \colon Y \to X$ be proper, and let $g \colon S' \to S$ be a base change.  Define $X' = X \times_S S'$ and $Y' = Y \times_S S'$.  By properness, $\pi_! \cong \pi_*$.  Write $\overleftarrow{\pi}_*$ for the pushforward along $Y \atimes_S S \to X \atimes_S S$ and similarly $\overleftarrow{\pi'} \colon Y' \atimes_{S'} S' \to X' \atimes_{S'} {S'}$.  In the commutative diagram
\begin{equation}
    \begin{tikzcd}
        Y' \atimes_{S'} S' \arrow{r}{\overleftarrow{\pi'}} \arrow{d}{\overleftarrow{g}} & X' \atimes_{S'} S' \arrow{d}{\overleftarrow{g}} \\
        Y \atimes_{S} S \arrow{r}{\overleftarrow{\pi}}& X \atimes_{S} S
    \end{tikzcd}
\end{equation}
the property that $\overleftarrow{\pi'}_* \overleftarrow{g}^* \simeq \overleftarrow{g}^* \overleftarrow{\pi}_*$ follows by \cite[Lemma~2.3]{kato2020etale}.  This base change property allows us to form the diagram
\begin{equation}
    \begin{tikzcd}
        \overleftarrow{g}^* \Psi_f \pi_! L \arrow{r} \arrow{d}{\simeq}& \Psi_{f'} g^* \pi_! L \arrow{d}{\simeq} \\
        \overleftarrow{\pi'}_* \overleftarrow{g}^* \Psi_{f \circ \pi} L \arrow{r}{\simeq}& \overleftarrow{\pi'}_* \Psi_{f' \circ \pi'} g^* L,
    \end{tikzcd}
\end{equation}
where the bottom row is an isomorphism by the $\Psi$-good assumption on $L$ and the vertical arrows are base change isomorphisms using the identification of $\pi_!$ with $\pi_*$.  Therefore, the top row is an isomorphism, and we conclude that $\pi_! L = K$ is $\Psi$-good.

Now we check that the composition property is preserved.  Let $u \to t \to s$ be a sequence of specializations in the base $S$.  Then the diagram
\begin{equation}
\begin{tikzcd}
    \pi_! R(\Psi_{f \circ \pi})_u^s L \arrow{r}{\mathrm{can}} \arrow{d}&  R(\Psi_f)_u^s \pi_! L \arrow{d} \\
    \pi_! R(\Psi_{f \circ \pi})_t^s R(\Psi_{f \circ \pi})_u^t L \arrow{r}{\mathrm{can}}& R(\Psi_f)_t^s R(\Psi_f)_u^t \pi_! L
\end{tikzcd}
\end{equation}
commutes.  By assumption the left vertical arrow and both horizontal arrows are isomorphisms, so the right vertical arrow is as well.  Since the specializations were arbitrary, we conclude that $\pi_! L = K$ is $\Psi$-factorizable.
\end{proof}

\begin{proposition}\label{prop: psi factorizability under pullback}
    If $K$ is $h^* L$, $L$ is $\Psi$-factorizable, and $h$ is smooth, then $K$ is $\Psi$-factorizable.
\end{proposition}

\begin{proof}
    The proof is similar to the above except now we use a diagram
    \begin{equation}
    \begin{tikzcd}
        X' \arrow{r}{g} \arrow{d}{h'} & X \arrow{d}{h} \\
        Y' \arrow{r}{g} \arrow{d}{f'} & Y \arrow{d}{f} \\
        S' \arrow{r}{g} & S.
    \end{tikzcd}
    \end{equation}
    To check the $\Psi$-good property, we instead form the diagram
    \begin{equation}
    \begin{tikzcd}
        \overleftarrow{g}^* \Psi_{f \circ h} h^* L \arrow{r} \arrow{d}{\simeq}& \Psi_{f' \circ h'} g^* h^* L \arrow{d}{\simeq} \\
        \overleftarrow{h'}^* \overleftarrow{g}^* \Psi_{f} L \arrow{r}{\simeq}& \overleftarrow{h'}^* \Psi_{f'} g^* L,
    \end{tikzcd}
    \end{equation}
    where the relevant base change property for nearby cycles over general bases is now \cite[Equation~(4.4)]{lu2019duality}.

    To check the composition property, let $u \to t \to s$ be a sequence of specializations in the base $S$.  Then the diagram
\begin{equation}
\begin{tikzcd}
    h^* R(\Psi_f)_u^s L \arrow{r}{\mathrm{can}} \arrow{d}&  R(\Psi_{f \circ h})_u^s h^* L \arrow{d} \\
    h^* R(\Psi_f)_t^s R(\Psi_f)_u^t L \arrow{r}{\mathrm{can}}& R(\Psi_{f \circ h})_t^s R(\Psi_{f \circ h})_u^t h^* L
\end{tikzcd}
\end{equation}
commutes.  By assumption the left vertical arrow and both horizontal arrows are isomorphisms, so the right vertical arrow is as well.  Since the specializations were arbitrary, we conclude that $h^* L = K$ is $\Psi$-factorizable.
\end{proof}

\subsection{Nearby cycles over a product of a curve}

We make two extensions of our notions of the previous section without significant justification.  First, we will assert that all our notions will extend to the setting of $f \colon X \to S$ where $X$ is a Deligne-Mumford stack locally of finite type that admits an exhaustive filtration $X^{\le \mu}$ by open substacks over a countable poset.  Second, we will assert that the notion of $\Psi$-factorizability also makes sense, with appropriate modifications, for the case when the sheaf $K$ lives in the derived category of ind-constructible sheaves.

Let $J$ be a finite set and let $C^J$ be $J$-fold power of a curve.  Let $\eta$ be the generic point of $C$ in the Zariski topology, so $\eta^J$ is the generic point of $C^J$.  Choose a geometric generic point $\overline{\eta_J}$ over $\eta^J$.  Similarly, let $s_J$ be a closed point of $C^J$ and let $\overline{s_J}$ be an $\fqbar$-point over $s_J$ that is equipped with a specialization map $\overline{\eta_J} \to \overline{s_J}$.

Let $J_1$ and $J_2$ be two identical copies of $J$ and let our base be $S = C^{J_1 \cup J_2}$.  Under the diagonal copy $\Delta(C^J) \subset C^{J_1 \cup J_2}$, write the image of $\overline{s_J}$ as $\overline{s_{J_1}} \times \overline{s_{J_2}}$.  The choice of specialization map above gives a commutative diagram
\begin{equation}
    \begin{tikzcd}
        S_{\overline{\eta_{J_1}} \times \overline{\eta_{J_2}}} \arrow{r} \arrow{d}& S_{(\overline{\eta_{J_1}} \times \overline{s_{J_2}})} \arrow{d} \\
        S_{(\Delta(\overline{\eta_{J}}))} \arrow{r}& S_{(\overline{s_{J_1}} \times \overline{s_{J_2}})}.
    \end{tikzcd}
\end{equation}
Choosing a geometric generic point $\overline{\eta_{J_1 \cup J_2}}$ of $\overline{\eta_{J_1}} \times \overline{\eta_{J_2}}$ then gives a diagram of specializations
\begin{equation}
    \begin{tikzcd}
        \overline{\eta_{J_1 \cup J_2}} \arrow{r} \arrow{d}& \overline{\eta_{J_1}} \times \overline{s_{J_2}} \arrow{d} \\
        \Delta(\overline{\eta_{J}}) \arrow{r}& \Delta(\overline{s_J}) = \overline{s_{J_1}} \times \overline{s_{J_2}}.
    \end{tikzcd}
\end{equation}
\begin{definition}
For a map $f \colon X \to C^{J_1 \cup J_2}$, let $f_i \colon X \to C^{J_i}$ by projection.  We make the following definitions.
\begin{enumerate}
\item Let $\Psi_{J_1 \cup J_2}$ denote sliced nearby cycles $R(\Psi_f)_{\overline{\eta_{J_1 \cup J_2}}}^{\Delta(\overline{s_J})}$.
\item Let $\Psi_{\Delta}$ denote sliced nearby cycles $R(\Psi_f)_{\Delta(\overline{\eta_J})}^{\Delta(\overline{s_J})}$.
\item Let $\Psi^{\Delta}$ denote sliced nearby cycles $R(\Psi_f)_{\overline{\eta_{J_1 \cup J_2}}}^{\Delta(\overline{\eta_J})}$.
\item Let $\Psi_{J_i}$ denote the nearby cycles $R(\Psi_{f_i})_{\overline{\eta_J}}^{\overline{s_J}}$.
\end{enumerate}
\end{definition}
If $(f,K)$ is $\Psi$-good, construction~\ref{map to iterated} now gives maps
\begin{equation}
    \begin{tikzcd}
        \Psi_{\Delta} \Psi^{\Delta} & \Psi_{J_1 \cup J_2} \arrow{r} \arrow{l} & \Psi_{J_1} \Psi_{J_2},
    \end{tikzcd}
\end{equation}
and these maps are isomorphisms if $(f,K)$ is $\Psi$-factorizable.  If in addition $K$ is locally acyclic over the generic point, $\Psi^{\Delta}$ is specialization to the diagonal.

Our usage of nearby cycles over general bases is to prove various fusion properties of nearby cycles over a power of a curve.  These fusion properties will justify the terminology of $\Psi$-factorizability that we used in the previous section.  One fusion property is contained in Theorem~\ref{thm:smoothness}, which implies that cohomology sheaves of shtukas $\mathcal{H}_{I \cup J_1 \cup J_2,V \boxtimes W_1 \boxtimes W_2}$ over a curve $C^{I \cup J_1 \cup J_2}$ projecting to $C^{J_1 \cup J_2}$ are $\Psi$-good with respect to the projection and satisfy
\[
\begin{aligned}
    \Psi_{J_1 \cup J_2} \mathcal{H}_{I \cup J_1 \cup J_2, V \boxtimes W_1 \boxtimes W_2} &\cong \Psi_{J_1} \Psi_{J_2} \mathcal{H}_{I \cup J_1 \cup J_2, V \boxtimes W_1 \boxtimes W_2} \\
    &\cong \Psi_J \mathcal{H}_{I \cup J, V \boxtimes (W_1 \otimes W_2)}.
\end{aligned} \]
The other fusion property is Theorem~\ref{thm:nearbyfusion}, which implies corresponding properties for sheaves that we will denote $\mathscr{F}_{I \cup J, V \boxtimes W, \zeta}$.

\section{Smoothness of cohomology sheaves with Moy-Prasad level structures}\label{sec:shtukas}

\subsection{Parahoric group schemes and their Moy-Prasad filtrations}

Throughout this paper we deal with infinite-dimensional groups $G$ acting on ind-schemes or ind-stacks $X$.  These ind-schemes are equipped with a closed stratification $X_{\le \omega}$ and for each stratum there is a finite-dimensional quotient $G_n$ through which the infinite-dimensional group acts such that the subgroup acting trivially is a pro-unipotent group.  For such cases, $D(G_n \backslash X_{\le \omega}) \cong D(G_m \backslash X_{\le \omega})$ for all $m \ge n$, and we can define $D(G \backslash X_{\le \omega})$ unambiguously.  Then, we may set $D(G \backslash X) = \varinjlim_{\omega} D(G \backslash X_{\le \omega})$.

Let $G$ be parahoric group scheme with reductive generic fiber, also denoted $G$, and let $G_x$ denote its reduction at $x$, assumed to be a parahoric group.  Considering the split form of $G$ over a finite field, we may form the loop group $LG$ and positive loop group $L^+ G$.  The classical affine Grassmannian $\Gr_G$ is the quotient $LG / L^+ G$.  More generally, for the parahoric group $G_x$, we let $\Fl_x$ denote the partial affine flag variety for this parahoric group.  Let $R$ be the divisor over which $G$ does not have hyperspecial reduction.

Let $\Gr_{G,I}$ be the Beilinson-Drinfeld Grassmannian over $C^I$.  Over a general closed point $(x_1, \dots, x_{|I|})$ away from diagonals, if the reduction of $G_{x_i}$ can be identified with $L^+G$, then the fiber of the map is $\Gr_G^I$, $I$ copies of the usual affine Grassmannian $LG / L^+ G$, while over $(x, \dots, x)$ the fiber is $\Fl_x$, the partial affine flag variety associated to the parahoric group at $x$.  Let $G_{I,\infty} = \varprojlim G_{I,(n_i)}$ where $G_{I,(n_i)}(S)$ classifies an $I$-tuple $(x_i)_{i \in I}$ of $S$-points of $C$ together with an automorphism of the trivial $G$-bundle on the graph of $\sum n_i x_i$.  Geometric Satake produces sheaves on the quotient $[G_{I,\infty} \backslash \Gr_{G,I}] |_{(C \setminus R)^I}$ of the form $\mathcal{S}_V$ for $V \in \Rep(\widehat{G}^I)$ (more generally, see \cite[Section~2.5]{gaitsgory2007jong}).

Following \cite{lafforgue2018chtoucas}, we can put various additional structures on the Beilinson-Drinfeld Grassmannian.  For a partition $(I_1, \dots, I_n)$ of $I$, we can consider an iterated Beilinson-Drinfeld Grassmannian $\Gr^{(I_1, \dots, I_n)}$ classifying a chain of Hecke modifications along each term $I_k$ in the partition instead of a single Hecke modification along all the legs at once.  Over the diagonals, the map $\Gr^{(1, 2)}_{G, \{ 1,2 \}} \to \Gr_{G, \{1,2\}}$ realizes the convolution product.  Additionally, the Beilinson-Drinfeld Grassmannian admits a stratification by $\Gr_{G,I,V}$ for $V$ an $I$-tuple of irreducible representations of $\widehat{G}$ giving modification types along the legs.  The closures of strata are the supports of the Satake sheaves $\mathcal{S}_{I,V}$.

Let $x$ be a closed point of $C$.  For a point in the building $x_0$ such that $G_{x_0}$ corresponds to the parahoric group attached to $G_x$, there is a Moy-Prasad filtration that gives normal subgroups $G_{x_0}^r$ and $G_{x_0}^{r+}$ of $G_x$ for each real number $r$.  We will assume that such a point in the building is fixed in advance and suppress the dependence on $x_0$ in our notation, writing simply $G_x^{r}$, resp. $G_x^{r+}$.  These groups can be defined a priori on the level of rational points but they also admit a group scheme structure.  Let $H^{r+} = G_x / G_x^{r+}$ and $V = G_x^r / G_x^{r+}$.  For $r > 0$, $V$ is a vector group.  In the course of this paper, we will often consider $N$ to be the data of a subset $N \subset C$ together with an effective $\mathbb{Q}$-divisor supported on $N$, a sum $\sum_{y_j \in N} a_j y_j$ of points $y_j \in |C|$ with $a_j \in \mathbb{Q}_{\ge 0}$ giving the data of Moy-Prasad subgroups $G_{y_j}^{a_j+}$ for each point in $N$.  We may write $N \subset C$ to denote the collection of points $\{ y_i \}$ when no confusion may result.  Let $N$ be fixed and define $H = \prod_j (G_{y_j} / G_{y_j}^{a_j+})$.

\begin{proposition}\label{prop:btconstruction}
By dilatation \cite{mayeux2020n}, we can form a smooth group scheme $G_N$ that modifies $G$ to have reduction $G_{y_j}^{a_j+}$ at each point $y_j \in N$.
\end{proposition}

\begin{proof}
Let $y_0$ be the basepoint in the building that defines the parahoric group $G_y$.  If $\pi$ is a uniformizer of the complete local ring $\mathcal{O}_y$ at the point $y$ in the curve $C$, the Moy-Prasad subgroup $G_y^n$ has $\pi^n \mathfrak{g}_y$ as its Lie algebra and can be identified with the kernel $G_y(\mathcal{O}_y) \to G_y(\mathcal{O}_y / \mathfrak{m}_y^n)$ on the level of points for $\mathfrak{m}_y$ the maximal ideal in $\mathcal{O}_y$.  We can define $G_y^n$ as the kernel of $G_y \to (G_y)_{n y}$ where the latter group classifies automorphisms of the trivial $G$-bundle on the divisor $n y$.  In the language of \cite{mayeux2020n}, we can form a group scheme $G_{ny}$ by a N\'eron blowup of $G$ in the trivial group along the divisor $n y$.

For a real number $a$ and $n > a$ for an integer $n$, $G_y^{a+} / G_y^n$ defines a closed subgroup scheme inside $G_y / G_y^n$.  We can define $G_{ay}$ by performing a N\'eron blowup along $G_y^{a+} / G_y^n$.  Finally, for a divisor $N = \sum a_j y_j$, we can perform N\'eron blowups as described above for each $y_j$ individually to form our smooth group scheme $G_N$.

We check that this N\'{e}ron blowup is independent of the choice of $n$ above.  We note that if $m > n$, then $\Spec(\mathcal{O}_y / \mathfrak{m}_y^n) \to \Spec(\mathcal{O}_y / \mathfrak{m}_y^m)$ is a closed embedding under which $H_n = G_y^{a+} / G_y^n$ is the preimage of $H_m = G_y^{a+} / G_y^m$ viewed as schemes over these respective divisors.  Then the result follows by \cite[Lemma~2.7]{mayeux2020n}.
\end{proof}

We have the following properties of the group constructed above.

\begin{proposition}\label{prop:Ginfty construction}
\begin{enumerate}
\item Consider the group $(G_N)_{I \cup J,\infty} \to X^{I \cup J}$.  The special fiber over a tuple of points $(x_i)_{i \in I \cup J}$ is
\[ \prod_{x_i \not \in N} G_{x_i} \times \prod_{y_j} G_{y_j}^{a_j+}, \]
where the product $y_j$ runs over distinct points in the set $\{ x_i : i \in I \cup J \} \cap N$ and the product $x_i$ runs over distinct points in the relative complement.  The map $G_N \to G$ induces an injective morphism $(G_N)_{I \cup J,\infty} \to G_{I \cup J,\infty}$ compatible with this decomposition.
\item Let $J$ be defined as the points in the divisor $N$.  Define $G_{I,\infty,\sum \infty y_j}$ as the fiber of $G_{I \cup J,\infty} \to X^{I \cup J} \to X^J$ over $(y_j)_{j\in J}$ where $y_j$ runs through points in $N$.  Similarly, we may define $(G_N)_{I,\infty,\sum \infty y_j}$ and by the previous proposition there is a map $i_N \colon (G_N)_{I,\infty,\sum \infty y_j} \to G_{I,\infty,\sum \infty y_j}$.  This is a normal subgroup scheme whose cokernel can be identified with the constant group $H = \prod_{y_j} G_{y_j} / G_{y_j}^{a_j+}$.
\end{enumerate}
\end{proposition}

\begin{proof}
\begin{enumerate}
\item Let $(n_i)_{i \in I \cup J}$ be such that $n_i > a_j$ for any $i,j$.  By the independence of $n$ property in the proof of \ref{prop:btconstruction}, the fiber of $(G_N)_{I \cup J,(n_i)_{i \in I \cup J}}$ over a tuple $(x_i)_{i \in I \cup J}$ is
\[ \prod_{x \not \in N} G_{x} / G_{x}^{\sum_i n_i} \times \prod_{x \in N} G_{x}^{a(x)+} / G_{x}^{\sum_i n_i}, \]
where the product $x$ runs over distinct points in $\{ x_i : i \in I \cup J \}$, $\sum_i n_i$ denotes the sum of $n_i$ such that $x_i = x$, and $a(x)$ denotes the coefficient of the divisor in the case that $x \in N$.  We note the inclusion $G_N \to G$ induces an inclusion $G_{x}^{a(x)+} \to G_{x}$ in the product decomposition.  Now the proposition follows by passing to the inverse limit as $(n_i) \to \infty$.
\item It suffices to prove the corresponding statement for $G_{I\cup J,(d_i)_{i \in I \cup J}}$ for $(d_i)$ sufficiently large.  Let $D = \sum d_j y_j$ denote the divisor along which we perform dilatation in \ref{prop:btconstruction}, let $j \colon D \to C$ denote the inclusion of this divisor, and let $G_{D}^{(a_j+)}$ denote the subgroup scheme on $D$ in $G_{D}$ that produces $G_N$ by dilatation.  By \cite[Lemma~3.7]{mayeux2020n} we have an exact sequence of sheaves of pointed sets on the syntomic site
\begin{equation}
    1 \to G_N \to G \to j_*(G|_{D} / G_{D}^{(a_j+)}) \to 1.
\end{equation}

Now consider the universal divisor $\mathcal{D} \subset C^{I} \times C$ that over a point $(x_i)_{i \in I}$ has fiber $\sum n_i x_i + \sum \infty y_j$.  We have a diagram
\begin{equation}
    \begin{tikzcd}
    C^{I} & \mathcal{D} \arrow{l}{p_1} \arrow{r}{p_2} & C
    \end{tikzcd}
\end{equation}
The group scheme $G_{I,(n_i),\sum \infty y_j}$ represents the sheaf $(p_1)_* (p_2)^* G$.  Since $p_2$ is fpqc, if $S \to \mathcal{D}$ is fpqc, then $G(S)$ using the composition is the same, set-theoretically, as $(p_2)^* G(S)$, and the injection $G_N(S) \to G(S)$ gives an injection $(p_2)^* G_N \to (p_2)^* G$.  Pushing forward along $p_1$, we get an exact sequence of sheaves of pointed sets
\begin{equation}
    1 \to (p_1)_* (p_2)^* G_N \to (p_1)_* (p_2)^* G \to (p_1)_* (p_2)^* j_*(G|_{D} / G_{D}^{(a_j+)}) \to 1,
\end{equation}
where the sheaf on the right is represented by the constant group $H$.
\end{enumerate}
\end{proof}

For the following definition, compare the discussion preceding Lemme~2.7 of \cite{genestier2017chtoucas}.

\begin{definition}\label{def:grgtriv}
Let $\Gr_{G,I}^{\triv(N)}$ classify as its $S$-points:
\begin{enumerate}
    \item $(x_i)_{i \in I}\in C^I(S)$,
    \item $\mathcal{E}$ a $G$-bundle on $\Gamma_{\sum_i \infty x_i + \sum_j \infty y_j}$,
    \item $\phi$ a Hecke modification
    \[ \phi \colon \mathcal{E} |_{\Gamma_{\sum_i \infty x_i + \sum_j \infty y_j} \setminus \Gamma_{\sum_i x_i}} \simeq \mathcal{G} |_{\Gamma_{\sum_i \infty x_i + \sum_j \infty y_j} \setminus \Gamma_{\sum_i x_i}} \]
    where $\mathcal{G}$ is the trivial $G$-bundle, and
    \item for each $y_j$ in the support of $N$, a reduction of structure group $\varepsilon$ of $\mathcal{E}$ from $G_{y_j}$ to $G_{y_j}^{a_j+}$.
\end{enumerate}

More generally, if $(I_1, \dots, I_n)$ is a partition of $I$, let $\Gr_{G,I}^{(I_1, \dots, I_n),\triv(N)}$ classify as its $S$-points:
\begin{enumerate}
    \item $(x_i)_{i \in I}\in C^I(S)$,
    \item $\mathcal{E}_0, \dots, \mathcal{E}_n$ a set of $n+1$ $G$-bundles on $\Gamma_{\sum_i \infty x_i + \sum_j \infty y_j}$,
    \item $\phi_1, \dots, \phi_{n}$ Hecke modifications
    \[ \phi_k \colon \mathcal{E}_k |_{\Gamma_{\sum_{i\in I_k} \infty x_i + \sum_j \infty y_j} \setminus \Gamma_{\sum_{i\in I_k} x_i}} \simeq \mathcal{E}_{k+1} |_{\Gamma_{\sum_{i\in I_k} \infty x_i + \sum_j \infty y_j} \setminus \Gamma_{\sum_{i\in I_k} x_i}} \]
    and $\theta$ an isomorphism of the last bundle to the trivial bundle
    \[ \theta \colon \mathcal{E}_n \simeq \mathcal{G}. \]
    \item for each $y_j$ in the support of $N$, a reduction of structure group $\varepsilon$ of $\mathcal{E}_0$ from $G_{y_j}$ to $G_{y_j}^{a_j+}$.
\end{enumerate}
\end{definition}

\begin{example}
When $I = \{ 0 \}$, $G$ has Iwahori reduction at a point $x$, and $N = \{ x \} \subset C$ corresponding to the divisor $0 \cdot x$, $\Gr_{G,I}^{\triv(N)}$ is closely related to the ind-scheme $\widetilde{\Fl_C}$ appearing in \cite{bezrukavnikov2009tensor}.  In fact, this paper was the main motivation for considering the central sheaf construction on restricted shtukas.
\end{example}

$\Gr_{G,I}^{(I_1, \dots, I_n),\triv(N)}$ admits a natural map to $C^I$ that forgets all the information except for the legs.  When the partition is trivial $(I)$, then the iterated version reduces to $\Gr^{\triv(N)}_{G,I}$.

\begin{proposition}
$\Gr_{G,I}^{(I_1, \dots, I_n),\triv(N)} \to \Gr^{(I_1, \dots, I_n)}_{G,I}$ is an $H$-torsor.  When restricted to $(C \setminus N)^I$, this $H$-torsor is trivial.
\end{proposition}

\begin{proof}
Given a tuple $((x_i)_{i \in I}, (\mathcal{E}_0, \dots, \mathcal{E}_n), (\phi_0,\dots,\phi_n,\theta), \varepsilon)$, forgetting $\varepsilon$ gives an $H$-torsor over its quotient.  If we further restrict $\mathcal{E}$ to the formal neighborhood of $x_i$ (restricting away from points $y_j$ not coinciding with some $x_i$), this gives a point of $\Gr^{(I_1, \dots, I_n)}_{G,I}$.  We want to show that this identifies $\Gr^{(I_1, \dots, I_n)}_{G,I}$ with the quotient $\Gr_{G,I}^{(I_1, \dots, I_n),\triv(N)} / H$.  To construct an inverse map, note that a point in $\Gr_{G,I}$ determines a collection of points $(x_i)_{i \in I}$, a $G$-bundle over the whole curve $\mathcal{E}$, and a Hecke modification $\mathcal{E} \to \mathcal{G}$ away from $\Gamma_{\sum_i x_i}$.  This gives a point of $\Gr_{G,I}^{\triv(N)} / H$ by restricting this Hecke modification to $\Gamma_{\sum_i \infty x_i + \sum_j \infty y_j} \setminus \Gamma_{\sum_i x_i}$.

When each of the legs $x_i \in C \setminus N$, the information of the reduction of structure group does not interact with the other bundle data defining the tuple, and $\Gr_{G,I}^{(I_1, \dots, I_n),\triv(N)}$ splits as a product $\Gr^{(I_1, \dots, I_n)}_{G,I} \times H$.
\end{proof}

\begin{proposition}\label{prop:stratification}
$\Gr_{G,I}^{(I_1, \dots, I_n),\triv(N)}$ is an ind-scheme admitting a Schubert stratification by tuples of irreducible representations $(W_1, \dots, W_{|I|}) \in \Rep(\widehat{G}^I)$.  The strata are stable under a left action of $G_{I,\infty,\sum \infty y_j}$.
\end{proposition}

The data of a tuple $(W_1, \dots, W_{|I|})$ of representations of $\widehat{G}$ is equivalent to a tuple of dominant cocharacters $(\lambda_1, \dots, \lambda_{|I|}) \in X_*(T)^+$ of a maximal torus corresponding to the highest weights of the dual torus.

\begin{proof}
This stratification is pulled back from the corresponding stratification of $\Gr^{(I_1, \dots, I_n)}_{G,I}$.
\end{proof}

As in the Beilinson-Drinfeld Grassmannian, there is a map
\[
\begin{aligned}
\Gr_{G,I}^{(I_1, \dots, I_n),\triv(N)} &\rightarrow \Gr_{G,I}^{\triv(N)} \\
((x_i)_{i \in I}, (\mathcal{E}_0, \dots, \mathcal{E}_n), (\phi_1,\dots,\phi_n,\theta), \varepsilon) &\mapsto ((x_i)_{i\in I}, \mathcal{E}_0, \theta \circ \phi_n \circ \dots \phi_0, \varepsilon).
\end{aligned}
\]
This map can be thought of as giving a convolution product over generalized diagonals.

\begin{proposition}\label{properness}
The map $\Gr_{G,I}^{(I_1, \dots, I_n),\triv(N)} \rightarrow \Gr_{G,I}^{\triv(N)}$ collapsing the legs is a ind-proper morphism.
\end{proposition}

We note that neither the source nor the target is ind-proper, as even over the generic fiber the map splits off a copy of $H$, which is an affine group scheme of positive dimension if $N$ is nonempty.

\begin{proof}
We have a Cartesian square
\begin{equation}
    \begin{tikzcd}
    \Gr_{G,I}^{(I_1, \dots, I_n),\triv(N)} \arrow[d] \arrow[r] & \Gr_{G,I}^{(I_1, \dots, I_n)} \arrow[d] \\
    \Gr_{G,I}^{\triv(N)} \arrow[r] & \Gr_{G,I}.
    \end{tikzcd}
\end{equation}
So the map in question is the pullback of a map between proper ind-schemes and is itself proper.
\end{proof}

Similar to the affine Grassmannian, we want a description of $\Gr_{G,I}^{\triv(N)}$ in terms of a loop group.

\begin{definition}
The positive loop group $L_I^+ G$ is defined as $G_{I,\infty}$, and the loop group $L_I G$ is the ind-group scheme whose $S$-points classify tuples $(x_i)_{i \in I} \in C^I(S)$ together with a trivialization of the trivial $G$-bundle on $\Gamma_{\sum \infty x_i} \setminus \Gamma_{\sum x_i}$.

We define another loop group, $L_I G_{\sum \infty y_j}$, to be the ind-group scheme over $C^I$ whose $S$-points classify a tuple $(x_i)_{i \in I} \in C^I(S)$ together with a trivialization of the trivial $G$-bundle on $\Gamma_{\sum \infty x_i + \sum \infty y_j} \setminus \Gamma_{\sum x_i}$.
\end{definition}

\begin{proposition}\label{prop:Ginfty action}
\begin{enumerate}
\item The quotient $L_I G_{\sum \infty y_j} / G_{I,\infty,\sum \infty y_j}$ classifies triples $((x_i)_{i \in I}, \mathcal{E}, \phi)$ given by $(1)-(3)$ in Definition~\ref{def:grgtriv} for $\Gr_{G,I}^{\triv(N)}$.
\item There is an isomorphism of ind-schemes $L_I G_{\sum \infty y_j} / G_{I,\infty,\sum \infty y_j} \cong L_I G / L_I^+ G$.
\item Let $G_{I,\infty,\sum \infty y_j}$ act on $\Gr_{G,I}^{\triv(N)}$ by composing with automorphisms of the trivial bundle $\mathcal{G}$ in the Hecke modification and by composing with automorphisms of the trivial $H$-bundle in the reduction of structure group.  Let $Z$ be the center of the split form of $G$ and let $Z_N$ be the preimage of the center in $G_N$.  Then the action of $G_{I,\infty,\sum \infty y_j}$ restricts to an action of $Z_{I,\infty}$.  The action of $(Z_N)_{I,\infty}$ is trivial.
\end{enumerate}
\end{proposition}

\begin{proof}
\begin{enumerate}
    \item The datum of an $S$-point $(x_i)_{i \in I}$ in $C^I$ together with $G$-bundle on the graph of $\sum_i \infty x_i + \sum_j \infty y_j$ is the same as an $S$-point of $BG_{I,\infty,\sum \infty y_j}$.  All such bundles are trivial, and the choice of a trivialization over $\Gamma_{\sum_i \infty x_i + \sum_j \infty y_j} \setminus \Gamma_{\sum_i x_i}$ is the choice of an $S$-point of $L_I G_{\sum \infty y_j}$.  Forgetting the trivialization then amounts to taking a quotient by $G_{I,\infty,\sum \infty y_j}$.
    \item There is a natural map $L_I G_{\sum \infty y_j} / G_{I,\infty,\sum \infty y_j} \to L_I G / L_I^+ G$ given by restriction.  The inverse map is given by identifying $L_I G / L_I^+ G$ with the Beauville-Laszlo moduli description of a collection of points $(x_i)_{i\in I}$, a $G$-bundle over the whole curve $C$, and an isomorphism to the trivial bundle away from the graphs $\sum_i x_i$.  We can then restrict from the curve to $\sum_i \infty x_i + \sum_j \infty y_j$.  The two maps are inverse to each other by construction.
    \item The action of $Z_{I,\infty}$ on the Beilinson-Drinfeld grassmannian is trivial, and $(Z_N)_{I,\infty}$ lives in the kernel of the map $G_{I,\infty,\sum \infty y_j} \to H$, so the action of $(Z_N)_{I,\infty}$ on $\Gr_{G,I}^{\triv(N)}$ is trivial.
\end{enumerate}
\end{proof}

Thus, given the ind-scheme $\Gr_{G,I}^{\triv(N)}$, we may take a quotient by an action of $G_{I,\infty,\sum \infty y_j}$ or similarly $(G_N)_{I,\infty,\sum \infty y_j}$.  At the risk of some repetition, we give moduli descriptions in the following proposition.

\begin{proposition}\label{prop:quotient moduli lemma}
The quotient $[(G_N)_{I,\infty,\sum \infty y_j} \backslash \Gr_{G,I}]$ classifies as its $S$-points:
\begin{enumerate}
    \item $(x_i)_{i \in I}\in C^I(S)$,
    \item $\mathcal{E}$ and $\mathcal{F}$ two $G$-bundles on $\Gamma_{\sum_i \infty x_i + \sum_j \infty y_j}$,
    \item $\phi$ a Hecke modification
    \[ \phi \colon \mathcal{E} |_{\Gamma_{\sum_i \infty x_i} \setminus \Gamma_{\sum_i x_i}} \simeq \mathcal{F} |_{\Gamma_{\sum_i \infty x_i} \setminus \Gamma_{\sum_i x_i}}, \]
    and
    \item for each $y_j$ in the support of $N$, a reduction of structure group $\varepsilon_{\mathcal{E}}$ of $\mathcal{E}$ from $G_{y_j}$ to $G_{y_j}^{a_j+}$.
\end{enumerate}

The quotient $[G_{I,\infty,\sum \infty y_j} \backslash \Gr_{G,I}^{\triv(N)}]$ classifies as its $S$-points:

\begin{enumerate}
    \item $(x_i)_{i \in I}\in C^I(S)$,
    \item $\mathcal{E}$ and $\mathcal{F}$ two $G$-bundles on $\Gamma_{\sum_i \infty x_i + \sum_j \infty y_j}$,
    \item $\phi$ a Hecke modification
    \[ \phi \colon \mathcal{E} |_{\Gamma_{\sum_i \infty x_i} \setminus \Gamma_{\sum_i x_i}} \simeq \mathcal{F} |_{\Gamma_{\sum_i \infty x_i} \setminus \Gamma_{\sum_i x_i}}, \]
    and
    \item for each $y_j$ in the support of $N$, a reduction of structure group $\varepsilon_{\mathcal{F}}$ of $\mathcal{F}$ from $G_{y_j}$ to $G_{y_j}^{a_j+}$.
\end{enumerate}
\end{proposition}

We note that the quotient $[G_{I,\infty,\sum \infty y_j} \backslash \Gr_{G,I}^{\triv(N)}]$, although not algebraic, is a ``nice'' quotient in the sense of \cite[Section~6, A.3]{gaitsgory1999construction}.  That is,
\begin{enumerate}
    \item $G_{I,\infty,\sum \infty y_j}$ is the projective limit of linear algebraic groups
    \item $G_{I,\infty,\sum \infty y_j}$ has a finite codimensional subgroup which is pro-unipotent,
    \item $\Gr_{G,I}^{\triv(N)}$ admits a stratification by closed subschemes stabilized by $G_{I,\infty,\sum \infty y_j}$, preimages of Schubert strata, such that any closed subscheme $Z \subset \Gr_{G,I}^{\triv(N)}$ is contained in such a finite union of strata, and
    \item the action of $G_{I,\infty,\sum \infty y_j}$ on closed stratum factors through a linear algebraic quotient.
\end{enumerate}

Here, the stratification on $\Gr_{G,I}^{\triv(N)}$ comes from the corresponding stratification on $\Gr_{G,I}$.  These niceness conditions will be satisfied for all such stacks of the form $[H \backslash X]$ that we consider when $H$ is infinite-dimensional.  Under these conditions, we may consider the category of perverse sheaves equivariant under the action of the projective limit, similarly the equivariant derived category of constructible sheaves, and this equivariant category is simply the limit of equivariant categories for the underlying linear algebraic groups.

\begin{proposition}
The ind-scheme $\Gr_{G,I}^{(I_1, \dots, I_n),\triv(N)}$ admits a map to
\begin{equation}\label{eqn:product grassmannian}
[(G_N)_{I_1,\infty,\sum \infty y_j} \backslash \Gr_{G,I_1}] \times \dots \times [G_{I_{n-1},\infty,\sum \infty y_j} \backslash \Gr_{G,I_{n-1}}] \times [G_{I_n,\infty,\sum \infty y_j} \backslash\Gr_{G,I_n}^{\triv(N)}]
\end{equation}
by sending $((x_i)_{i \in I}, (\mathcal{E}_k)_{k \in \{1,\dots,n\}}, \phi_k, \varepsilon_0, \theta)$ to the data of
\begin{enumerate}
    \item $((x_i)_{i \in I_1}, \mathcal{E}_0, \mathcal{E}_1, \phi_1, \varepsilon_0) \in [(G_N)_{I_1,\infty,\sum \infty y_j} \backslash \Gr_{G,I_1}]$ using the moduli description of Proposition~\ref{prop:quotient moduli lemma},
    \item for $1 < k < n$, $((x_i)_{i \in I_k}, \mathcal{E}_{k-1}, \mathcal{E}_k, \phi_k) \in [G_{I_k,\infty,\sum \infty y_j} \backslash \Gr_{G,I_k}]$,
    \item and $((x_i)_{i \in I_k}, \mathcal{E}_{k-1}, \mathcal{E}_k, \phi_k, \theta|_H) \in [G_{I_n,\infty,\sum \infty y_j} \backslash\Gr_{G,I_n}^{\triv(N)}]$ noting that the data of a reduction of a structure group in the moduli description is the same as a trivialization of associated $H$-bundles.
\end{enumerate}
\end{proposition}

This map is a torsor for an infinite-dimensional group scheme, namely $(G_N)_{I_1, \infty,\sum \infty y_j} \times \dots \times G_{I_{n-1},\infty,\sum \infty y_j} \times G_{I_n, \infty,\sum \infty y_j}$, so it inherits some smoothness properties.  Closed stratum by closed stratum, this map factors through a linear algebraic quotient, which is a bona fide (finite-dimensional) smooth map, similar to the smoothness of the maps $\kappa^{(I_1, \dots, I_k)}_{I, (\omega_i)_{i \in I}}$ from \cite[Equation~(1.12) and Lemme~1.16]{lafforgue2018chtoucas}.  Explicitly, we mean that there is a Schubert stratification $\Gr_{G,I,(W_i)_{i\in I}}^{(I_1, \dots, I_n),\triv(N)}$ and smooth maps
\begin{equation}\label{eqn: kappa map}
\begin{aligned}
\kappa^{(I_1, \dots, I_n)}_{(W_i)_{i\in I}} &\colon \Gr_{G,I,(W_i)_{i\in I}}^{(I_1, \dots, I_n),\triv(N)} \to \\
&[(G_N)_{I_1,\sum_{i\in I_1} n_i x_i} \backslash \Gr_{G,I_1,W_1}] \times \dots \\
&\qquad \times [G_{I_{n-1},\sum_{i\in I_{n-1}} n_i x_i} \backslash \Gr_{G,I_{n-1},W_{n-1}}] \times [G_{I_n,\sum_{i\in I_n} n_i x_i} \backslash\Gr_{G,I_n,W_n}^{\triv(N)}]
\end{aligned}
\end{equation}
for sufficiently large choices of $(n_i)_{i \in I}$ depending on $(W_i)_{i\in I}$.

Over the unramified locus, the map $\kappa$ in \eqref{eqn:product grassmannian} factors through the quotient
\begin{equation}\label{eqn:unramified product grassmannian}
    [G_{I_1,\infty} \backslash \Gr_{G,I_1}] \times \dots \times [G_{I_{n-1},\infty} \backslash \Gr_{G,I_{n-1}}] \times [G_{I_n,\infty} \backslash\Gr_{G,I_n}] \times H
\end{equation}
So for any local system $\zeta$ on $H$, and for any choices of $V_j \in \Rep(\widehat{G}^{I_j})$, we can form $\mathcal{S}_{V_1} \boxtimes \dots \boxtimes \mathcal{S}_{V_k} \boxtimes \zeta$ as a sheaf on the product \eqref{eqn:unramified product grassmannian}, and this sheaf is itself a pullback from a sheaf on the product in Equation~\eqref{eqn:product grassmannian}.  The pullback along the smooth map gives the usual twisted external product $\mathcal{S}_{V_1} \widetilde{\boxtimes} \dots \widetilde{\boxtimes} \mathcal{S}_{V_k} \boxtimes \zeta$ on $\Gr^{(I_1, \dots, I_k), \triv(N)}_{G,I}$ over the unramified locus.  The maximal partition is particularly important to us as it is quite close to the external product case before the trivialization by $\triv(N)$.  However, a general sheaf $\zeta$ on the copy of $H$ given by $\Gr_{G,I}^{(I_1, \dots, I_n),\triv(N)} \to \Gr_{G,I}^{(I_1, \dots, I_n)}$ forgetting $\triv(N)$ is not external in a way that Proposition~\ref{prop: kunneth factorizable} and Corollary~\ref{corr: Kunneth factorizable} can be applied.  To handle these cases, we can instead use a variant where each $\mathcal{E}_i$ is trivialized.

\begin{definition}\label{def:grgtrivmax}
Let $I = \{1,\dots,n\}$.  Let $\Gr_{G,I}^{\triv(N)^I}$ classify as its $S$-points:
\begin{enumerate}
    \item $(x_i)_{i \in I}\in C^I(S)$,
    \item $\mathcal{E}_0, \dots, \mathcal{E}_{n}$ a set of $n+1$ $G$-bundles on $\Gamma_{\sum_i \infty x_i + \sum_j \infty y_j}$,
    \item $\phi_1, \dots, \phi_{n}$ Hecke modifications
    \[ \phi_k \colon \mathcal{E}_k |_{\Gamma_{\sum_{i\in I_k} \infty x_i + \sum_j \infty y_j} \setminus \Gamma_{\sum_{i\in I_k} x_i}} \simeq \mathcal{E}_{k+1} |_{\Gamma_{\sum_{i\in I_k} \infty x_i + \sum_j \infty y_j} \setminus \Gamma_{\sum_{i\in I_k} x_i}} \]
    and $\theta$ an isomorphism of the last bundle to the trivial bundle
    \[ \theta \colon \mathcal{E}_n \simeq \mathcal{G}. \]
    \item $\varepsilon_1, \dots, \varepsilon_n$ reductions of structure group such that for each $y_j$ in the support of $N$, $\varepsilon_i$ is a reduction of structure group of $\mathcal{E}_{i-1}$ from $G_{y_j}$ to $G_{y_j}^{a_j+}$.
\end{enumerate}
There is a map $\Gr_{G,I\cup J}^{\triv(N)^{I\cup J}} \to C^I$ that sends a tuple classifying data above to the underlying points $(x_i)_{i \in I}$, and we will use $\mathfrak{q}_I$ to denote this map.
\end{definition}

\begin{proposition}\label{prop: grgtrivmax multiplication}
The forgetful map $\Gr_{G,I}^{\triv(N)^I} \to \Gr_{G,I}^{(1, \dots, n), \triv(N)}$ is an $H^{n-1}$-torsor.  Over the scheme $(C \setminus N)^I$, the diagram
\begin{equation}
\begin{tikzcd}
    \Gr_{G,I}^{(1,\dots,n)} \times H^I \arrow{r}{m} \arrow{d}& \Gr_{G,I}^{(1,\dots,n)} \times H \arrow{d} \\
    \Gr_{G,I}^{\triv(N)^I} \arrow{r}& \Gr_{G,I}^{(1, \dots, n), \triv(N)}
\end{tikzcd}
\end{equation}
commutes where the vertical arrows are isomorphisms trivializing the torsor and the top horizontal arrow pulls back the multiplication map $m \colon H^I \to H$.
\end{proposition}

\begin{proof}
    We consider the case $|I| = 2$, as the generalization to $I$ a finite set is clear.  The identification on the left gives the tuple $(h_i) \in H^I$ where $h_i$ is the element of $H$ represented by the isomorphism of associated $H$-bundles given by $\phi_i \colon \mathcal{E}_i \to \mathcal{E}_{i+1}$, both of which are trivial $H$-bundles by applying the reduction of structure group data.  Multiplication corresponds to the element representing the isomorphism of associated $H$-bundles given by the composition $\phi_n \circ \dots \circ \phi_1$, which is multiplication.
\end{proof}

Let $j_k$ denote any pullback of the inclusion of a power of the unramified locus into the whole curve $(C \setminus (N \cup R))^{I_k} \to C^{I_k}$.  Let $\mathfrak{q}_I$ denote the map $\Gr_{G,I\cup J}^{\triv(N)} \to C^I$ and its iterated variants.

\begin{proposition}\label{prop:psi-factorizable grassmannian}
The pullback constructed above agrees with the usual twisted external tensor product as in geometric Satake.

Let $\zeta$ be a rank-1 character sheaf corresponding to a representation of $H(\F_q)$.  Then the Satake sheaves on Beilinson-Drinfeld grassmannians as well as the pairs
\begin{equation} (\mathfrak{q}_I, (j_1)_! \mathcal{S}_{V_1} \widetilde{\boxtimes} \dots \widetilde{\boxtimes} (j_k)_! \mathcal{S}_{V_k} \widetilde{\boxtimes} j_! \mathcal{S}_{W} \boxtimes j_! \zeta \in D(\Gr^{(I_1, \dots, I_k), \triv(N)}_{G,I}, \qlbar)) \end{equation}
are $\Psi$-factorizable.
\end{proposition}

\begin{proof}
The first statement is essentially Theorem~1.17(c) of \cite{lafforgue2018chtoucas}.

Consider first the case that $(I_1, \dots, I_k)$ is a maximal partition such that each set $I_j$ has size $1$ in the constant sheaf case.  The external product of Satake sheaves is $\Psi$-factorizable by Proposition~\ref{prop: kunneth factorizable} and Corollary~\ref{corr: Kunneth factorizable}.  Being the pullback of an external product mapping to $\Gr_{G,I}^{(I_1, \dots, I_n)}$ along a smooth map forgetting the trivialization along $N$, the twisted external tensor products are $\Psi$-factorizable by Proposition~\ref{prop: psi factorizability under pullback} if $\zeta$ is $\qlbar$.  More generally, if $\zeta$ comes from a character sheaf on $H$, we can turn it into an external product by considering the pullback along $\Gr_{G,I}^{\triv(N)^I} \to \Gr_{G,I}^{(1, \dots, n), \triv(N)}$, which is an $H^{|I|-1}$-torsor.  Recall that for a character sheaf $\zeta$ on $H$, pullback on the multiplication map $m \colon H^I \to H$ gives $m^*(\zeta) = \zeta \boxtimes \dots \boxtimes \zeta$ as sheaves.  Along the identification in Proposition~\ref{prop: grgtrivmax multiplication}, the pullback of the sheaf in question when $W = W_1 \boxtimes \dots \boxtimes W_{|J|}$ is
\begin{equation}
\begin{aligned}
    (j_1)_! (\mathcal{S}_{V_1} \boxtimes \zeta) \widetilde{\boxtimes} \dots \widetilde{\boxtimes} (j_{|I|})_! (\mathcal{S}_{V_{|I|}} \boxtimes \zeta) \widetilde{\boxtimes} &\\
    (j_{|I|+1})_! (\mathcal{S}_{W_1} \boxtimes \zeta) \widetilde{\boxtimes} \dots \widetilde{\boxtimes} (j_{|I|+|J|})_! (\mathcal{S}_{W_{|J|}} \boxtimes \zeta)
    &\in D(\Gr_{G,I\cup J}^{\triv(N)^{I\cup J}}, \qlbar).
\end{aligned}
\end{equation}
The $\widetilde{\boxtimes}$ notation indicates that this sheaf is the pullback of an external direct product along a smooth map analogous to \eqref{eqn: kappa map} and is thus external.  Now, the result follows by the $\Psi$-factorizability of external products, pullback, and descent along a quotient.

Now consider case of a general partition.  Then the sheaves $\mathcal{S}_{V_1} \widetilde{\boxtimes} \dots \widetilde{\boxtimes} \mathcal{S}_{V_k} \boxtimes \zeta$ is a sum of pushforwards of twisted external products $\boxtimes \zeta$ along proper maps in Proposition~\ref{properness}.  Therefore, they are $\Psi$-factorizable by Proposition~\ref{prop: psi factorizability under pushforward}.
\end{proof}

\begin{remark}
    In particular, the above theorem applies for $\zeta$ a representation of the tori $T(\F_q)$ appearing when $G$ has Iwahori reduction at a point $y$ and $N = 0 \cdot y$ gives a reduction to the unipotent radical of the Iwahori.  The resulting local systems appear as isotypic components of Kummer covers.
\end{remark}
\begin{remark}
    In the case that $\zeta$ is trivial, a special case of this statement was given in \cite{salmon2023unipotent} and generalized in \cite[Theorem~3.2]{achar2023higher}.  The constructions in loc. cit. use a unipotent nearby cycles construction rather than the nearby cycles over general bases via the vanishing topos.
\end{remark}

For representations, we introduce the following definition to codify the hypothesis that a sheaf of the form $\mathcal{S}_{I,V} \boxtimes \zeta$ on $\Gr_{G,I}^{\triv(N)}$ is $\Psi$-factorizable.

\begin{definition}\label{def: factorizable representation}
    Let $N = \sum a_i y_i$ be a level structure with an identification of $H = \prod_i G_{y_i} / G_{y_i}^{a_i+}$ and $\zeta$ a representation of $H(\F_q)$.  We say that $\zeta$ is a $\Psi$-factorizable representation if $\zeta$ is a representation of $H(\F_q)$ in $\Lambda$ vector spaces and if for any tuple of representations $V \boxtimes W \in \Rep(\widehat{G}^{I\cup J})$, the pair $(\mathfrak{q}_I, j_! (\mathcal{S}_{I\cup J,V\boxtimes W} \boxtimes \zeta))$ is $\Psi$-factorizable.

    More generally, if $U \subset C$ is Zariski-open we say that $(N, \zeta)$ is a factorizable over $U$ if for any representation $V \boxtimes W \in \Rep(\widehat{G}^{I \cup J})$, the pair $(\mathfrak{q}_I, j_! (\mathcal{S}_{I\cup J,V \boxtimes W} \boxtimes \zeta))$ is $\Psi$-factorizable when restricted to $U^I$.
\end{definition}

To finish this section, we show that the main properties we are interested in can be reduced to the case of a divisor of the form $N = a \cdot y$ supported on a single point.

\begin{proposition}
    Let $N = N' + N''$ such that the support of $N'$ and $N''$ is disjoint.  For $N = \sum_j a_j y_j$, let $H$, resp. $H'$, $H''$ be the algebraic groups over $\F_q$ given by
    \[\prod_{y_j \in N} G_{y_j} / G_y^{a_j+},\]
    resp. for $N'$ and $N''$.  If $\zeta = \zeta' \boxtimes \zeta''$ as a representation of $H(\F_q) = H'(\F_q) \times H''(\F_q)$, and if $\zeta'$ is a $\Psi$-factorizable representation, then $\zeta''$ is factorizable over $C \setminus N''$.
\end{proposition}

\begin{proof}
    In this case, there is an isomorphism
    \begin{equation} \mathcal{S}_{I,V} \boxtimes \zeta \cong \mathcal{S}_{I\cup ,V\boxtimes W} \boxtimes \zeta' \boxtimes \zeta'', \end{equation}
    where $\zeta''$ is viewed as a local system over $H'' \to k$ with $k$ is the base field and $\mathcal{S}_{I\cup J,V \boxtimes W} \boxtimes \zeta'$ is viewed as a sheaf over $\mathfrak{q}_I \colon \Gr_{G,I\cup J}^{\triv(N')} \to (C \setminus N'')^I$.  The latter sheaf is $\Psi$-factorizable, and now Proposition~\ref{prop: kunneth factorizable} can be applied.
\end{proof}

\subsection{Global shtukas}

Let $C$ be a smooth, connected curve over $\mathbb{F}_q$.  Let $G$ be a parahoric group scheme over $\mathbb{F}_q$.  Let $R$ be the parahoric ramification of $G$, the finite subset of closed points in $C$ where the reduction of $G$ is not hyperspecial.  In this setting, there is a stack $\Bun_G$ of $G$-bundles on the curve (with parahoric structures according to the ramification $R$).

Given an $S$-point $\mathcal{E}$ of the stack of $G$-bundles on $C$, we denote the Frobenius pullback by ${}^{\tau} \mathcal{E} = (1 \times \Frob_S)^* \mathcal{E}$.  We recall the moduli space of shtukas \cite[Definition~2.1]{lafforgue2018chtoucas} \cite{xue2020cuspidal}.

\begin{definition}
An $S$-point of the moduli space $\Cht_{G,I}$ classifies a $G$-bundle $\mathcal{E} \in \Bun_G(S)$ on $C$, an $I$-tuple of points $(x_i)_{i \in I} \in C^I(S)$, and an isomorphism
\[ \phi \colon \mathcal{E}|_{C \times S \setminus \bigcup_{i\in I} \Gamma_{x_i}} \cong {}^{\tau}\mathcal{E}|_{C \times S \setminus \bigcup_{i\in I} \Gamma_{x_i}}. \]
\end{definition}

Let $\widehat{\Lambda}^{\mathbb{Q}}_G = \widehat{\Lambda}_G \otimes \mathbb{Q}$ be the rational vector space spanned by coweights of $G$.  There is a dominant cone $\widehat{\Lambda}^{+,\mathbb{Q}}_{G}$, a convex cone whose extreme rays are spanned by fundamental coweights.  There is a partial order on $\widehat{\Lambda}^{\mathbb{Q}}$ by $\mu_1 \ge \mu_2$ if and only if $\mu_1 - \mu_2$ is in the rational cone spanned by positive coroots.  To each $\mu \in \widehat{\Lambda}^{+,\mathbb{Q}}_{G}$, there is an open substack $\Bun_G^{\le \mu}$ which is of finite type as a stack for semisimple $G$, and running over a cofinal sequence of such $\mu$ give an exhaustive stratification of $\Bun_G$ known as the Harder-Narasimhan stratification \cite[Section~1.4]{xue2020cuspidal} and \cite[Section~7.3]{drinfeld2013compact}.  Similarly, for $G^{ad}$ the adjoint group and any $\mu \in \widehat{\Lambda}^{+,\mathbb{Q}}_{G^{ad}}$, there is an exhaustive stratification $\Cht^{\le \mu}_{G,I}$ given by open substacks which are each finite type \cite[Section~1.7]{xue2020cuspidal}.

Restricting to formal neighborhoods around the legs $(x_i)_{i \in I}$ gives a map
\[\epsilon \colon \Cht_{G,I} \to [G_{I,\infty} \backslash \Gr_{G,I}].\]
If $V \in \Rep(\widehat{G}^I)$, we can pull back the Satake sheaves $\mathcal{S}_V$ along the map $\epsilon$.  Let $Z$ be the center of the generic fiber of $G$.  If $\Xi$ is a discrete subgroup of $Z(\mathbb{A})$ such that $\Xi \cap Z(\mathbb{O}) Z(K) = \{ 1 \}$, then $\Xi$ maps to $\Bun_Z(\mathbb{F}_q)$ and acts on $\Cht_{G, I}$ via the $\Bun_Z(\mathbb{F}_q)$-action.  If furthermore $\Xi$ is a lattice in $Z(\mathbb{A})$, that is the quotient $\Bun_Z(\mathbb{F}_q) / \Xi$ is finite, then $\Cht_{G, I} / \Xi$ has finitely many connected components.  Furthermore, $\epsilon$ factors through $\Cht_{G,I} / \Xi =: \Cht_{G,\Xi,I}$, and thus we can speak of Satake sheaves $\mathcal{S}_V$ on $\Cht_{G,\Xi,I}$, which we will do as they have nice properties (in particular, the cohomology sheaves restricted to Harder-Narasimhan strata become constructible).

\begin{definition}\label{def: shtukas N level}
Let $N$ be the data of an effective $\mathbb{Q}$-divisor, a sum $\sum a_j y_j$ of points $y_j \in |C|$ with $a_j \in \mathbb{Q}$.  We may write $N \subset C$ to denote the collection of points $\{ y_i \}$ when no confusion may result.

The moduli space $\Cht_{G,\Xi,N,I}(S)$ of shtukas with $N$-level structure classifies the data of
\begin{enumerate}
    \item a shtuka $((x_i)_{i \in I}, \mathcal{E}, \phi) \in (\Cht_{G,\Xi,I} \times_{C^I} (C \setminus N)^I)(S)$,
    \item A trivialization $\mathcal{E}_H \simeq \mathcal{G}_H$ of the associated $H$-bundle of $\mathcal{E}$, compatible with the identification $\mathcal{E}_H \simeq {}^{\tau} \mathcal{E}_H$ coming from $\phi$.
\end{enumerate}
\end{definition}

For each $y_j$ in the support of $N$, a reduction of structure group of $\mathcal{E}$ from $G_{y_j}$ to $G_{y_j}^{a_j+}$ compatible with the shtuka structure in the sense that $\phi$ induces the same reduction of structure group ${}^{\tau} \mathcal{E}$ as the Frobenius pullback does.

\begin{proposition}\label{prop:etale cover shtukas}
The forgetful map $\Cht_{G,\Xi,N,I} \to \Cht_{G,\Xi,I}$ is a finite \'{e}tale cover over $(C \setminus N)^I$ with Galois group $\prod_j (G_{y_j} / G_{y_j}^{a_i+})(\mathbb{F}_q)$.
\end{proposition}

\begin{proof}
It suffices to prove this result before passing to the quotient by the discrete group $\Xi$, since the corresponding map before quotienting by $\Xi$ is $\Xi$-equivariant.

We can define an open subgroup $K_N$ of $G(\mathbb{O})$ as
\begin{equation}
K_N = G_N(\mathbb{O}) = \prod_{x \in |N|} G_x^{(a_j)+} \times \prod_{x \in |C \setminus N|} G_x \subset \prod_{x \in |C|} G_x = G(\mathbb{O}).
\end{equation}
Following the construction of $G_N$ as a N\'eron blowup, we can find a further divisor $D = \sum_j b_j y_j$ with $b_j > a_j$ and a subgroup
\begin{equation}
K_D = G_D(\mathbb{O}) = \prod_{x \in |N|} G_x^{b_j} \times \prod_{x \in |C \setminus N|} G_x \subset K_N,
\end{equation}
such that $K_N / K_D$ form the $\mathbb{F}_q$-rational points of some group scheme $J \subset G|_D$ for which $G_N$ is constructed by dilatation along the group scheme $J$ with respect to the divisor $D$.
We define the stack of shtukas $\Cht_{G,D,I}$ with level structures as the data of tuples $((x_i)_{i \in I}, \mathcal{E}, \phi, \varepsilon_D)$ of ordinary shtuka data away from the divisor $|D|$ with a level structure (trivialization)
\[ \varepsilon_D \colon \mathcal{E}|_{D \times S} \simeq \mathcal{G}|_{D \times S} \]
where $\mathcal{G}$ is the trivial $G$-bundle.  As a straightforward generalization of \cite[Proposition~2.16(b)]{varshavsky2004moduli}, the map $\Cht_{G,D,I} \to \Cht_{G,I}$ is Galois with Galois group $G(\mathbb{O}) / K_D$.  So it suffices to identify the map $\Cht_{G,N,I} \to \Cht_{G,I}$ with $\Cht_{G,D,I} / (K_N / K_D) \to \Cht_{G,I}$.  Such an identification can be made using \cite[Corollary~4.13]{mayeux2020n}, identifying $\Cht_{G,N,I}$ with shtukas for the group scheme $G_N$ and $\Cht_{G,D,I} / (K_N / K_D)$ with shtukas for $G$ with level-$(J,D)$-structures, both constructions occuring away from the level over $(C \setminus N)^I$.
\end{proof}

\subsection{A smoothness result}

In this section, our objects live over $\mathbb{F}_q$ or $\overline{\mathbb{F}_q}$ and we will often remember a Frobenius structure, or rather partial Frobenius structure.  We will occasionally add a subscript indicating the field to prevent ambiguity.  Recall that the moduli space of shtukas $\Cht_{G,\Xi,N,I}$ is equipped with a map forgetting all the information except for the legs
\[ \mathfrak{p} \colon \Cht_{G,\Xi,N,I} \to (C \setminus N)^I. \]

For fixed $\Cht_{G,\Xi,N,I}$, we say that the unramified locus of the moduli space of shtukas is the subset lying over $(C \setminus (N \cup R))^I$.  The significance of the unramified locus is that over this locus, the moduli space of shtukas is equipped with a smooth map
\[ \epsilon^{ur} \colon \Cht_{G,N,I} \to [G_{I,\infty} \backslash \Gr_G] \]
by restricting to formal neighborhoods around the legs.  Moreover, as explained in \cite[Section~1.3.2]{xue2020cuspidal}, the action of $G_{I,\infty}$ on $\Gr_G$ factors through $G^{ad}_{I,\infty}$, and the composition
\[ \Cht_{G,N,I} \to [G_{I,\infty} \backslash \Gr_G] \to [G^{ad}_{I,\infty} \backslash \Gr_G] \]
is $\Xi$-equivariant, thus producing
\begin{equation} \epsilon^{ur} \colon \Cht_{G,\Xi,N,I} \to [G^{ad}_{I,\infty} \backslash \Gr_G]. \end{equation}
See Definition~\ref{def:epsilon} and Theorem~\ref{thm:smooth model} for a generalization of this map.

Starting with Satake sheaves $\mathcal{S}_{I,V}$ for $V \in \Rep(\widehat{G}^I)$, we may pull them back along $\epsilon^{ur}$ to $\mathscr{F}_{I,V} := (\epsilon^{ur})^* \mathcal{S}_{I,V}$.  Then taking pushforward gives $\mathcal{H}_{I,V} := \mathfrak{p}_! \mathscr{F}_{I,V}$ as sheaves on $(C \setminus (N \cup R))^I_{\overline{\mathbb{F}_q}}$.  Similarly, for $\mu \in \widehat{\Lambda}^{+,\mathbb{Q}}_{\widehat{G}}$, we can form $\mathscr{F}^{\le \mu}_{I,V}$ by restricting to the open substack $\Cht^{\le \mu}_{G,\Xi,N,I}$ over the unramified locus and denote by $\mathcal{H}^{\le \mu}_{I,V}$ the corresponding pushforward.  These cohomology sheaves are in fact constructible.

We now recall some smoothness results of \cite{xue2020smoothness} and \cite{salmon2023unipotent}, using some of the framework of \cite{hemo2020constructible}.  Recall the following consequence of Xue's smoothness result:

\begin{theorem}\label{thm:Xue smoothness}
The sheaves $\mathcal{H}_{I,V}$ are smooth over $(C \setminus (N \cup R))^I$.
\end{theorem}

We are particularly interested in the stalks at generic points of the sheaves $\mathcal{H}_{I,V}$.  Let the geometric generic point of $(C \setminus (N \cup R))^I$ be denoted $\overline{\eta_I}$, and write $H_{I,V}$ for the vector space $\mathcal{H}_{I,V}|_{\overline{\eta_I}}$.  Since these sheaves are also equipped with partial Frobenius maps, this includes the statement that the cohomology sheaves $H_{I,V}$, a priori equipped with an absolute Galois group action at the geometric generic point, factor through the $I$th power of the Weil group.  Recall that the Weil group is defined by a square
\begin{equation}
\begin{tikzcd}
    \Weil((C \setminus (N \cup R))_{\mathbb{F}_q}, \overline{\eta}) \arrow{r} \arrow{d} & \mathbb{Z} \arrow{d} \\
    \Gal((C \setminus (N \cup R))_{\mathbb{F}_q}, \overline{\eta}) \arrow{r} & \widehat{\mathbb{Z}}.
\end{tikzcd}
\end{equation}
Recall also the Frobenius-Weil group is defined as in \cite[Section~1.2.6]{xue2020smoothness} so that we have a commutative diagram
\begin{equation}
\begin{tikzcd}
    \Gal(\overline{\eta_I} / \eta^I) \arrow{r} \arrow{d} & F\Weil(\eta^I, \overline{\eta_I}) \arrow{r} \arrow{d} & \mathbb{Z}^{I} \arrow[d, equal] \\
    \Gal(\overline{\eta} / \eta)^I \arrow{r} & \Weil(\eta, \overline{\eta})^I \arrow{r} & \mathbb{Z}^I.
\end{tikzcd}
\end{equation}
Both groups are extensions of a profinite group by a finitely generated group.  So a consequence of Xue's statement is that
\begin{corollary}
The action of $F\Weil(\eta^I, \overline{\eta_I})$ on $H_{I,V}$ factors through $\Weil(\eta, \overline{\eta})^I$.
\end{corollary}

The smoothness result of Theorem~\ref{thm:Xue smoothness} proved in \cite{xue2020smoothness} has been reinterpreted by \cite[Proposition~6.7]{hemo2020constructible} in a form that is directly useful to us.

Following \cite{bhatt2013pro} and \cite{hemo2020constructible}, inside $D(X, \Lambda)$, we may consider subcategories of lisse (smooth) and constructible (complexes of) sheaves which we denote $D_{\mathrm{lis}}(X, \Lambda)$ and $D_{\mathrm{cons}}(X, \Lambda)$, respectively.  We also have ind-lisse and ind-constructible sheaves which are filtered colimits of lisse and constructible (complexes of) sheaves, which we denote by $D_{\mathrm{indlis}}(X, \Lambda)$ and $D_{\mathrm{indcons}}(X, \Lambda)$.  We can artificially equip these categories with an additional action of partial Frobenius or many such partial Frobenius actions.  If $X = X_1 \times \dots \times X_n$, we define $\phi_{X_i} = \Frob_{X_i} \times \mathrm{id}_{\overline{\mathbb{F}_q}}$ and defining the Weil-pro\'{e}tale site $X_1^{\Weil} \times \dots \times X_n^{\Weil}$ as consisting of pairs $(U, (\varphi_1, \dots, \varphi_n))$ where $U \to X_1 \times \dots \times X_n$ is pro-\'{e}tale and $\varphi_i\colon U \to U$ is an endomorphism covering $\phi_{X_i}$ on the $i$th factor and the identity on all other factors (for a full treatment, see \cite[Section~4]{hemo2020constructible}).  We can similarly form the derived categories of ind-lisse and ind-constructible sheaves on this site.  Hemo, Richarz, and Scholbach prove

\begin{proposition}\label{thm:Weil smoothness}
The sheaves $\mathcal{H}_{I,V}$ lie in the essential image of the functor
\[ D_{\mathrm{indlis}}((C \setminus (N \cup R))^{\mathrm{Weil}}, \Lambda)^{\otimes I} \to D_{\mathrm{indcons}}(((C \setminus (N \cup R))^{\mathrm{Weil}})^I, \Lambda). \]
\end{proposition}

We will see that being in the essential image of this functor suffices to prove the following corollary.

\begin{corollary}
Let $j \colon (C \setminus (N \cup R))^I \to C^I$ be the open embedding.  The sheaves $j_! \mathcal{H}_{I,V}$ are $\Psi$-factorizable with respect to the identity map $C^I \to C^I$.
\end{corollary}

As in \cite{salmon2023unipotent}, we will generalize this corollary to Theorem~\ref{thm:smoothness}, which includes various slightly modified settings, such as pushforwards of certain nearby cycles sheaves that we will define later.  There is a functorial characterization of the properties that are sufficient to ensure that Xue's smoothness argument can be run.  Let us recall how this smoothness argument goes.  Let $\FinS$ be the category of finite sets and set maps between them.  We can define two categories $\Rep(\widehat{G}^{\bullet})$ and $D^b_c(X^{\bullet}, \Lambda)$, both of which are cofibered over $\FinS$, following \cite[Section~3]{salmon2023unipotent}.

Let $\mathcal{G}^{\le \mu}$ be a filtered collection of cocartesian functors $\Rep(\widehat{G}^{\bullet}) \to D^b_c(X^{\bullet}, \Lambda)$ for each rational dominant coweight $\mu$.  The archetypical example are the cohomology sheaves $\mathcal{H}^{\le \mu}_{I,V}$ constructed above.  Considering dominant coweights as a monoid with ordering given by positive coroots, we have a notion of $\mathcal{G}^{\le \mu}$ being filtered by partial Frobenius maps $F_{\bullet}$ as in \cite[Definition~3.2]{salmon2023unipotent}.  We are especially interested in
\[ \varinjlim_{\mu} \mathcal{G}^{\le \mu}(V) := \varinjlim_{\mu} \mathcal{G}^{\le \mu}(I,V) \]
as an ind-constructible sheaf over $X^I$ whose stalk at the geometric generic point $\overline{\eta_I}$ is equipped with an action of the Frobenius-Weil group $F\Weil(\eta^I, \overline{\eta_I})$.

Being constructible, $\mathcal{G}^{\le \mu}(V)$ is smooth over some open subset $\Omega \subset X^I$.  The Eichler-Shimura relation is proved in detail in \cite[Theorem~3.5]{salmon2023unipotent} for this setting.  Let $v$ be a closed point in this open subset and let $v_i$ be its projections to each coordinate $i \in I$.  Let $\mathscr{H}_{G, v_i}$ be the local excursion algebra at $v_i$ generated by generalized $S$-excursion operators as in \cite[Definition~3.4]{salmon2023unipotent}.  This will turns out to be equivalent to an action of a local Hecke algebra but we do not know this (or need this) a priori, but we do know that this algebra is finitely generated since it is isomorphic to $\mathscr{O}(\widehat{G})^{\widehat{G}}$, that is, $\widehat{G}$-conjugation-invariant regular functions on $\widehat{G}$.  We can construct a submodule $\mathfrak{M}_{\mu}$ of $\varinjlim \mathcal{G}^{\le \mu}$ generated by the image of $\mathcal{G}^{\le \mu} \to \varinjlim \mathcal{G}^{\le \mu}$ under the action of the tensor product of local excursion algebras $\otimes_{i \in I} \mathscr{H}_{G,v_i}$ together with the partial Frobenius maps.  The Eichler-Shimura relation produces a relation in this action that implies finite generation of this module under the tensor product of Hecke algebras.  As a result, we can apply Drinfeld's lemma in the form \cite[Lemma~3.2.13]{xue2018finiteness} and argue as in \cite[Proposition~1.4.3]{xue2020smoothness} that any specialization map between points over $\eta^I$ induces an isomorphism.  This allows us to conclude that $\varinjlim \mathcal{G}^{\le \mu}$ are ind-smooth over $\eta^I$, which is a Zariski open subset.

Thus, the corollary is a special case of the following theorem.

\begin{theorem}\label{thm:smoothness}
Let $\mathcal{G}^{\le \mu}$ be a collection of cocartesian functors $\Rep(\widehat{G}^{\bullet}) \to D^b_c(X^{\bullet}, \Lambda)$ between categories cofibered over $\FinS$ and filtered with respect to partial Frobenius maps $F_{\bullet}$, where $X$ is a Zariski open subset of $C$.  Let $j \colon X^I \to C^I$ be the open embedding.  Then the sheaves $j_! \left(\varinjlim_{\mu} \mathcal{G}^{\le \mu}(V)\right)$ are $\Psi$-factorizable with respect to any projection $X^I \to X^J$ coming from a map $J \to I$ of finite sets.
\end{theorem}

\begin{proof}[Proof of Theorem~\ref{thm:smoothness}]
By \cite[Theorem~3.6]{salmon2023unipotent} and the argument above, for any $I$ and $V \in \Rep(\widehat{G}^I)$, the sheaves $\varinjlim_{\mu} \mathcal{G}^{\le \mu}(V)$ are constant over $\overline{\eta}^I$ and smooth over $\eta^I$.

Consider first the case $J \to I$ is the identity map.  Then the sheaves are $\Psi$-good because they are ULA over the generic point and the complement of this point is quasi-finite over the base \cite[Theorem~6.1]{orgogozo2006modifications} \cite[Example~1.7(c)]{illusie2017around}.  We conclude that $\varinjlim_{\mu} \mathcal{G}^{\le \mu}$ is a filtered colimit of external tensor products of Weil local systems over the generic point $\eta$.  For the identity map, $R\Psi$ commutes with colimits because $f$ is the identity and $R\Psi_f$ is \emph{exact} if $f$ is finite by \cite[Theorem~1.14]{illusie2017around}.  Thus, the K\"unnneth theorem in Proposition~\ref{prop: kunneth factorizable} and its corollary~\ref{corr: Kunneth factorizable} implies $\Psi$-factorizability of an individual external tensor product (and its $!$ pushforward away from the unramified locus) and $R\Psi$ commuting with colimits will allow us to conclude the full statement for $\varinjlim_{\mu} \mathcal{G}^{\le \mu}$.

More generally, suppose $h \colon J \to I$ is a map of finite sets that is an injection.  Denote $p_h \colon X^I \to X^J$ the corresponding projection map on powers of $X$ and consider also the pullback $h^* \colon \Rep(\widehat{G}^I) \to \Rep(\widehat{G}^J)$.  Then over the generic point $\eta^I$ we can write $\varinjlim_{\mu} \mathcal{G}^{\le \mu}(V)$ as a filtered colimit of external products of Weil sheaves over the generic point $\eta$.  Writing $\Weil^I \cong \Weil^J \times \Weil^{I \setminus J}$, we can write simple objects in this category of Weil sheaves in the form $K_{J,\mu} \boxtimes K_{I \setminus J,\mu}$ where $K_{I \setminus J, \mu}$ extends to a sheaf on $\eta^{I \setminus J} \times X^J$ ULA over $X^J$.  Thus, for any given external product, the resulting sheaf is $\Psi$-factorizable with respect to $p_h$ and $R\Psi_{p_h}$ is exact with respect to the natural t-structure, since we can write $R\Psi_{p_h} (K_J \boxtimes K_{I \setminus J})$ as $R\Psi_{p_h} K_J \boxtimes K_{I \setminus J}$ where it exact on the first factor, related by base change to nearby cycles for the identity map $X^J \to X^J$ for the sheaf $K_J$.  For any filtered colimit of external product sheaves, $R\Psi_{p_h}$ is exact and we can conclude $\Psi$-factorizability with respect to the map $p_h$ by passing to the colimit.
\end{proof}

\section{Fusion on restricted shtukas}\label{sec:fusion}

\subsection{Restricted shtukas}

Restricted shtukas appeared in \cite{genestier2017chtoucas} and \cite{xiao2017cycles}.  They are more convenient for us than local shtukas.  Heuristically, for a normal (open compact) subgroup $K$ of a parahoric group $P$ of the loop group $LG$, local shtukas for $P$ ``look like'' a Frobenius-twisted quotient $LG /_{\sigma} P$, but restricted shtukas look like $[K \backslash LG / K] /_{\sigma} (P / K)$ on the special fiber.  If we consider just the single point given by $P \subset LG$, sheaves over the point $P /_{\sigma} P$ are representations of $P(\mathbb{F}_q)$, a category that is quite complicated.  On the other hand, restricted shtukas should only be able to ``see'' a residual $(P / K)(\mathbb{F}_q)$-representation.

\begin{definition}
Let $\Cht\mathcal{R}_{G,N,I}$ classify as its $S$-points:
\begin{enumerate}
    \item $(x_i)_{i \in I}\in C^I(S)$,
    \item $\mathcal{E}$ and $\mathcal{F}$ two $G_{I,\infty,\sum \infty y_j}$-torsors on $S$,
    \item $\phi$ a Hecke modification
    \[ \phi \colon \mathcal{E} |_{\Gamma_{\sum_i \infty x_i} \setminus \Gamma_{\sum_i x_i}} \simeq \mathcal{F} |_{\Gamma_{\sum_i \infty x_i} \setminus \Gamma_{\sum_i x_i}}, \]
    \item and an isomorphism of associated $H$-bundles $\varepsilon \colon \mathcal{F}_H \simeq {}^{\tau} \mathcal{E}_H$.
\end{enumerate}

Similarly, for a partition $(I_1, \dots, I_n)$ of $I$, let $\Cht\mathcal{R}^{(I_1, \dots, I_n)}_{G,N,I}$ classify as its $S$-points:
\begin{enumerate}
    \item $(x_i)_{i \in I}\in C^I(S)$,
    \item $\mathcal{E}_0, \dots, \mathcal{E}_n$ a set of $n+1$ $G$-bundles on $\Gamma_{\sum_i \infty x_i + \sum_j \infty y_j}$,
    \item $\phi_1, \dots, \phi_{n}$ Hecke modifications
    \[ \phi_k \colon \mathcal{E}_k |_{\Gamma_{\sum_i \infty x_i} \setminus \Gamma_{\sum_i x_i}} \simeq \mathcal{E}_{k+1} |_{\Gamma_{\sum_i \infty x_i} \setminus \Gamma_{\sum_i x_i}}, \]
    \item and an isomorphism of associated $H$-bundles $\varepsilon \colon (\mathcal{E}_n)_H \simeq {}^{\tau} (\mathcal{E}_0)_H$.
\end{enumerate}
\end{definition}


When the legs $(x_i)_{i \in I}$ live in $C \setminus N$, we can consider the left and right action of $H$ on $\Gr_{G,I}^{\triv(N)}$.  We can then define the Frobenius-twisted action of $H$ by letting $H$ act through $\Frob(h)$ on the left and $h^{-1}$ on the right.

\begin{proposition}\label{prop:restricted shtuka quotient definition}
There is a natural map $\Gr_{G,I}^{\triv(N)} \to \Cht\mathcal{R}_{G,N,I}$ that is a $G_{I,\infty,\sum \infty y_j}$-torsor.  The $G_{I,\infty,\sum \infty y_j}$ action here is a twisted action that acts on the left by the usual action from Proposition~\ref{prop:Ginfty action}~(3) and on the right through $G_{I,\infty,\sum \infty y_j} \to H$ by changing the trivialization of ${}^{\tau} \mathcal{G}_H$ by $\Frob(h)$.

Over $C \setminus N$, restricted shtukas $\Cht\mathcal{R}_{G,N,I}$ are the quotient of $\Gr_{G,I}^{\triv(N)} \cong \Gr_{G,I} \times H$ by $G_{I,\infty,\sum \infty y_j}$ where the action on $\Gr_{G,I}^{\triv(N)}$ is the usual action and the action on $H$ is the twisted adjoint action through the map $G_{I,\infty,\sum \infty y_j} \to H$.  In particular, the stack identifies with $[G_{I,\infty} \backslash \Gr_{G,I}] \times [U_H \backslash BH(\mathbb{F}_q)]$ where $U_H$ is the unipotent group
\[ \prod_{j \in J} G_{y_j}^{a_j+}. \]
\end{proposition}

\begin{proof}
Let $((x_i)_{i \in I}, \mathcal{E}, \mathcal{F}, \phi, \varepsilon)$ be a point in $\Cht\mathcal{R}_{G,N,I}(S)$.  Consider adding the data of a $G_{I,\infty,\sum_{j} \infty y_j}$-trivialization of $\mathcal{E}$.  Then $((x_i)_{i \in I}, \mathcal{F}, \phi^{-1}, \varepsilon)$ defines a point in $\Gr_{G,I}^{\triv(N)}$, which identifies
\[\Cht\mathcal{R}_{G,N,I} \cong [G_{I,\infty,\sum_j \infty y_j} \backslash \Gr_{G,I}^{\triv(N)}].\]

Note that over $C \setminus N$, any $S$-point in $\Gr_{G,I}^{\triv(N)}$ decomposes as a point in $\Gr_{G,I}$ and a reduction of structure group, since these two pieces of data do not interact with each other.  Similarly,
\[G_{I, \infty, N} \cong G_{I, \infty} \times H \cong G_{I,\infty,\sum_j \infty y_j} / U_H \]
over $C \setminus N$, and the $U_H$ quotient of $G_{I,\infty,\sum_j \infty y_j}$ acts trivially on $\Gr_{G,I}^{\triv(N)}$.  So what remains is how to describe the $H$ action on this locus.

Away from $N$, we have two trivializations of the associated $H$-bundle $\mathcal{F}_H$.  The Hecke modification restricts to $\phi_H \colon \mathcal{F}_H \simeq \mathcal{E}_H \simeq \mathcal{G}_H$ where $\mathcal{G}$ is a trivial bundle.  At the same time, we have an additional trivialization of $\mathcal{F}_H$ coming from the reduction of structure group.  When we pass to the quotient by $G_{I,\infty,\sum \infty y_j}$, the composition of the two trivializations gives $\mathcal{G}_H \simeq \mathcal{E}_H \simeq_{\phi_H^{-1}} \mathcal{F}_H \simeq {}^{\tau} \mathcal{E}_H$, where action of $h \in H$ changes of trivialization of $\mathcal{E}_H$ modifies the left composition by $h^{-1}$ and the right composition by $\Frob(h)$.  So the $H$ action is as described.  The quotient of $H$ by this action is $BH(\mathbb{F}_q)$ as this is a Lang map.  We are left with a residual action of the unipotent group $U_H$ that acts trivially.
\end{proof}

We also have an ``adjoint group analogue'' of the definition of restricted shtukas that is necessary for compatibility with the action of the discrete group $\Xi$.  Let $Z$ be the center of $G$ and let $G^{ad}$ denote the adjoint quotient $G / Z$.  The corresponding dilatation $Z_N$ of the center is compatible with the formation of the dilatations $G_N$ and $G^{ad}_N$, where the latter group is the $N$-dilatation of $G^{ad}$.  That is, we have a diagram
\begin{equation}\label{eq:center-adjoint sequence}
    \begin{tikzcd}
    1 \arrow{r} & Z_N \arrow{r}\arrow{d}& G_N \arrow{r}\arrow{d}& G_N^{ad} \arrow{r}\arrow{d}& 1, \\
    1 \arrow{r}& Z \arrow{r} \arrow{d}& G \arrow{r} \arrow{d}& G^{ad} \arrow{r} \arrow{d}& 1, \\
    1 \arrow{r}& j_*(Z|_D / Z_D^{a_j+}) \arrow{r}& j_*(G|_D / G_D^{a_j+}) \arrow{r}& j_*(G^{ad}_D / (G^{ad})^{a_j+}_D) \arrow{r}& 1
    \end{tikzcd}
\end{equation}
where the maps on the top row are constructed by the commutativity of the bottom squares, and the maps on the bottom row are constructed by compatibilities of Moy-Prasad filtrations on the formal disk.

We can consider subgroups $Z_{I,\infty} \subset G_{I,\infty}$ and $(Z_N)_{I,\infty} \subset (G_N)_{I,\infty}$, similarly $Z_{I,\infty,\sum \infty y_j}$.  We observe that the group $(Z_N)_{I,\infty,\sum \infty y_j}$ acts trivially on $\Gr^{\triv(N)}_{G,I}$, and the action of $Z_{I,\infty,\sum \infty y_j}$ on $\Gr^{\triv(N)}_{G,I}$ factors through the central subgroup
\[Z_H = \prod (Z_G)_{y_j} / (Z_G)_{y_j}^{(a_j)+} \subset H.\]

\begin{definition}
$\Cht\mathcal{R}^{ad}_{G,N,I}$ is defined as the quotient $[G^{ad}_{I,\infty,\sum \infty y_j} \backslash (Z_H \backslash \Gr^{\triv(N)}_{G,I})]$ using the action from Proposition~\ref{prop:restricted shtuka quotient definition}.
\end{definition}

Let us explain the two quotients above.  The first quotient by $Z_H$ is the Frobenius adjoint action by the action of $Z_H \to H$.  On the quotient by $Z_H$, the central subgroup $Z_{I,\infty}$ acts trivially by Proposition~\ref{prop:Ginfty action}.  This allows us to form the diagram
\begin{equation}
\begin{tikzcd}
    \Cht\mathcal{R}_{G,N,I} \arrow{r}\arrow{d} & \Cht\mathcal{R}^{ad}_{G,N,I} \arrow{d} \\
    \left[G_{I,\infty} \backslash \Gr_{G,I}\right] \arrow{r} & \left[G^{ad}_{I,\infty} \backslash \Gr_{G,I}\right],
\end{tikzcd}
\end{equation}
where the vertical arrows forget reductions of structure group.

\begin{example}
Consider the case of restricted shtukas with $|I| = 1$ and Iwahori structure at $x$ with $N = 0 \cdot x$.  Let $I_x^+$ denote the unipotent radical of the Iwahori at $x$.  Over $C \setminus \{ x \}$ this is just $[G_{I,\infty} \times I_x^+ \backslash \Gr_{G,I} \times BT(\mathbb{F}_q)]$, where $T$ is the Cartan and where $BT(\mathbb{F}_q)$ is the classifying space of the finite group $T(\mathbb{F}_q)$.  Over the point $x$ this is $\left[ \frac{I_x^+ \backslash LG / I_x^+}{\Ad_{\Frob}(T)} \right]$.
\end{example}

\subsection{The smooth model}

\begin{definition}\label{def:epsilon}
The map
\[ \epsilon \colon \Cht_{G,I} \to \Cht\mathcal{R}_{G,N,I} \]
is defined by sending $((x_i)_{i \in I}, \mathcal{E}, \phi \colon \mathcal{E}_0 |_{C \setminus \Gamma_{\sum_i x_i}} \to {}^{\tau} \mathcal{E} |_{C \setminus \Gamma_{\sum_i x_i}})$ to:
\begin{enumerate}
    \item $(x_i)_{i \in I}$ the same tuple of $S$-points in $C$,
    \item the associated $G_{I,\infty,\sum \infty y_j}$-bundle coming from Weil restricting $\mathcal{E}$ and $\mathcal{F} = {}^{\tau} \mathcal{E}$ on a formal neighborhood $\sum \infty x_i + \sum \infty y_j$, forgetting the fact that $\mathcal{F}$ is actually the Frobenius pullback,
    \item the underlying Hecke modification coming from the restriction to the formal neighborhood $\sum \infty x_i + \sum \infty y_j$, and
    \item the tautological isomorphism of associated $H$-bundles $\mathcal{F}_H \cong {}^{\tau} \mathcal{E}_H$.
\end{enumerate}
\end{definition}

The following is a smooth local model theorem analogous to \cite[Proposition~2.11]{genestier2017chtoucas} but also handling more general ``Moy-Prasad'' level structures.

\begin{theorem}\label{thm:smooth model}
The map
\[ \epsilon \colon \Cht_{G,I} \to \Cht\mathcal{R}_{G,N,I} \]
is smooth over $C^I$.  Let $\Xi$ be a discrete subgroup of $Z(\mathbb{A})$ such that $\Xi \cap Z(\mathbb{O}) Z(K) = \{ 1 \}$.  This map factors to a map
\[ \epsilon \colon \Cht_{G,\Xi,I} \to \Cht\mathcal{R}^{ad}_{G,N,I} \]
which is $\Xi$-equivariant for the trivial action on the right.  Over $(C \setminus N)^I$, this map restricts to a map of the form
\[ \epsilon \colon \Cht_{G,\Xi,I} \to [G^{ad}_{I,\infty} \backslash \Gr_{G,I}] \times BH(\mathbb{F}_q) \times BU \]
for some unipotent group $U$, where the map to $BH(\mathbb{F}_q)$ classifies the $H(\mathbb{F}_q)$-cover $\Cht_{G,\Xi,N,I} \to \Cht_{G,\Xi,I}$.
\end{theorem}

\begin{proof}
We note that in the case of level subgroups, this is the smoothness of the restriction map in the case $r = 0$ in \cite[Proposition~2.15]{genestier2017chtoucas}.  Our situation is simpler because we do not need to remove any graphs of Frobenius in $X^I$ (in the notation of \cite{genestier2017chtoucas}, $U_{I,0}$ is just the curve itself).  We can reduce to the case when the level structure is $N = a y$ at a single point $y$.  Consider an $S$-point $z$ of $\Cht\mathcal{R}_{G,N,I}$ with the goal of showing that $\Cht_{G,\Xi,I} \times_{\Cht\mathcal{R}_{G,N,I}} S$ is smooth.  We can write this fiber product as the equalizer of two maps
\begin{equation}
    b_1, b_2 \colon BG_{I,\infty,N} \to BH.
\end{equation}
The first map starts with $\mathcal{F}$ and forgets the level structure by passing to the associated $H$-bundle, and is a $(G_N)_{I,\infty}$-torsor, and is smooth.  The second map is the composition of modification by the underlying Hecke modification giving $\mathcal{E}$ as an $H$-bundle, applies the relative Frobenius over $X^I$, then takes the isomorphism to recover $\mathcal{F}_H$.  This composition has zero differential.  Therefore, the coequalizer is smooth over $S$.

Consider the composition
\[ \Cht_{G,I} \to [G_{I,\infty,\sum \infty y_j} \backslash \Gr_{G,I}^{\triv(N)}] \to [G^{ad}_{I,\infty,\sum \infty y_j} \backslash \Gr_{G,I}^{\triv(N)}] = \Cht\mathcal{R}^{ad}_{G,N,I}. \]
By the assumption above, the map above is equivariant for the action of the discrete group $\Xi$ where it acts trivially on $\Cht\mathcal{R}^{ad}_{G,N,I}$.  Therefore, this constructs the second map.  The last claim comes from the moduli description of the unramified locus in Proposition~\ref{prop:restricted shtuka quotient definition} and Proposition~\ref{prop:etale cover shtukas}.
\end{proof}

\subsection{$\Psi$-factorizable sheaves on shtukas}

Using Theorem~\ref{thm:smooth model} above, computing nearby cycles of sheaves pulled back from $\Cht\mathcal{R}^{ad}_{G,N,I}$ on $\Cht_{G,\Xi,I}$ reduces to computing nearby cycles first on $\Cht\mathcal{R}^{ad}_{G,N,I}$ and then pulling back.  Recall that the stack $[(G_N)_{I,\infty, \sum \infty y_j} \backslash \Gr_{G,I}^{\triv(N)}]$ classifies as its $S$-points the following data:
\begin{enumerate}
    \item $(x_i)_{i \in I}\in C^I(S)$,
    \item $\mathcal{E}$ and $\mathcal{F}$ two $G_{I,\infty,\sum \infty y_j}$-torsors on $S$,
    \item $\phi$ a Hecke modification
    \[ \phi \colon \mathcal{E} |_{\Gamma_{\sum_i \infty x_i + \sum_j \infty y_j} \setminus \Gamma_{\sum_i x_i}} \simeq \mathcal{F} |_{\Gamma_{\sum_i \infty x_i + \sum_j \infty y_j} \setminus \Gamma_{\sum_i x_i}}, \]
    \item and two isomorphisms of associated $H$-bundles $\varepsilon_1 \colon \mathcal{F}_H \simeq \mathcal{G}_H$ and $\varepsilon_2 \colon \mathcal{E}_H \simeq \mathcal{G}_H$.
\end{enumerate}

Let $\sigma \colon \mathcal{G}_H \to {}^{\tau} \mathcal{G}_H$ be a fixed isomorphism.  We can consider the action map
\begin{equation}
H \times [(G_N)_{I,\infty,\sum_j \infty y_j} \backslash \Gr_{G,I}^{\triv(N)}] \to [(G_N)_{I,\infty,\sum_j \infty y_j} \backslash \Gr_{G,I}^{\triv(N)}]
\end{equation}
giving
\begin{equation}
H \backslash [(G_N)_{I,\infty, \sum_j \infty y_j} \backslash \Gr_{G,I}^{\triv(N)}] \cong \Cht\mathcal{R}_{G,N,I}
\end{equation}
as a quotient.  We note that $G_y / G_y^{(a_y+)} \cong H$ and $G_y^{(a_y+)}$ acts trivially.  Consider the diagonal action changing the trivialization on both $\varepsilon_1$ and $\varepsilon_2$ simultaneously.  If $h \colon \mathcal{G}_H \to \mathcal{G}_H$, and then we get an isomorphism $\Frob(h) \colon {}^{\tau} \mathcal{G}_H \to {}^{\tau} \mathcal{G}_H$ such that the diagram commutes:
\begin{equation}
\begin{tikzcd}
\mathcal{G}_H \arrow{r}{h} \arrow{d}{\sigma} & \mathcal{G}_H \arrow{d}{\sigma} \\
{}^{\tau} \mathcal{G}_H \arrow{r}{\Frob(h)} & {}^{\tau} \mathcal{G}_H.
\end{tikzcd}
\end{equation}

So in the natural map $q \colon [(G_N)_{I,\infty,\sum_j \infty y_j} \backslash \Gr_{G,I}^{\triv(N)}] \to \Cht\mathcal{R}_{G,N,I}$ sending $(\varepsilon_1, \varepsilon_2)$ to the composition
\begin{equation}
\mathcal{F}_H \simeq \mathcal{G}_H \simeq {}^{\tau} \mathcal{G}_H \simeq {}^{\tau} \mathcal{E}_H,
\end{equation}
the quotient map is invariant under the action of $H$ changing the trivialization of $\varepsilon_1$ and ${}^{\tau} \varepsilon_2$ by composing with $(h, \Frob(h))$.  For this reason, we denote the action of $H$ by $\Ad_F(H)$.

Consider the diagram
\begin{equation}\label{HResolution}
    \begin{tikzcd}
    H \times \left[(G_N)_{I,\infty,\sum_j \infty y_j} \backslash \Gr_{G,I}^{\triv(N)}\right] \arrow[r,yshift=0.7ex,"p"] \arrow[r,yshift=-0.7ex,swap,"a"] & \left[(G_N)_{I,\infty,\sum_j \infty y_j} \backslash \Gr_{G,I}^{\triv(N)}\right] \arrow{r}{q} & \Cht\mathcal{R}_{G,N,I},
    \end{tikzcd}
\end{equation}
where $p$ and $a$ are projection and action maps, respectively, and $q$ is a quotient map.  Let $V \in \Rep(\widehat{G}^I)$.  We have Satake sheaves with the sheaves $\mathcal{S}_V \boxtimes \zeta$ on the locus of $\Cht\mathcal{R}_{G,N,I}$ living over $(C \setminus (N \cup R))^I$.  Consider $\zeta$ as an equivariant sheaf on $H / \Ad_{\Frob}(H)$, which is the same as supplying its pullback $q^* \zeta$ to a local system on $H$, together with an isomorphism $\vartheta \colon p^* q^* \zeta \simeq a^* q^* \zeta$ satisfying additional coherence diagrams with respect to multiplication maps.

Now suppose for simplicity $I = \{ 1 \}$, $N = \{ y \}$ corresponding to the divisor $a y$ and we wanted to compute nearby cycles of $\mathcal{S}_V \boxtimes \zeta$ on $\Cht\mathcal{R}_{G,N,I}$ to the special fiber $y_j$.  Nearby cycles gives an equivariant sheaf on the special fiber, which can be identified with
\[ \frac{[G_{y}^{a+} \backslash LG / G_{y}^{a+}]}{\Ad_{\Frob}(H)}. \]
In terms of descent data, the nearby cycles can be computed by supplying the sheaf $q^* \Psi(\mathcal{S}_V \boxtimes \zeta) \cong \Psi(\mathcal{S}_V \boxtimes q^* \zeta)$ together with the isomorphism
\[ p^* q^* \Psi(\mathcal{S}_V \boxtimes \zeta) \cong \Psi(\mathcal{S}_V \boxtimes p^* q^* \zeta) \cong \Psi(\mathcal{S}_V \boxtimes a^* q^* \zeta) \cong a^* q^* \Psi(\mathcal{S}_V \boxtimes \zeta). \]
These considerations clearly carry over to the case of multiple legs.  Such nearby cycles were considered in \cite{genestier2017chtoucas} in the case when $\zeta$ corresponds to the regular representation.

\begin{definition}
For $\zeta \in \Rep(H(\mathbb{F}_q))$ and $V \in \Rep(\widehat{G}^I)$, we may consider $\mathcal{S}_V \boxtimes \zeta$ on the part of $\Cht\mathcal{R}^{ad}_{G,N,I}$ over the unramified locus.  We define $\mathscr{F}_{I,V,\zeta}$ on $\Cht_{G,\Xi,N,I}|_{(C \setminus (N \cup R))^I}$ as $\epsilon^*(\mathcal{S}_V \boxtimes \zeta)$.
\end{definition}

Letting $j \colon \Cht_{G,\Xi,N,I}|_{(C \setminus (N \cup R))^I} \to \Cht_{G,\Xi,N,I}$ be the open embedding, the sheaf $j_! \mathscr{F}_{I,V,\zeta}$ is constructible on $\Cht_{G,\Xi,N,I}$.

\begin{theorem}\label{thm:nearbyfusion}
Let $J = \{ j_1, \dots, j_n \}$, and let $\mathfrak{p}_J \colon \Cht_{G,I\cup J} \to C^J$.  Let $V \in \Rep(\widehat{G}^I)$ and $W \in \Rep(\widehat{G}^J)$.  Geometric Satake produces sheaves $\mathscr{F}_{V \boxtimes W, \zeta}$ over $\Cht_{G,I}$ for any $\zeta \in \Rep(H(\mathbb{F}_q))$, considered as a sheaf on $H / \Ad_{\Frob}(H)$.

Let $U$ be an Zariski open subset of $C$ such that $C \setminus (N \cup R) \subseteq U \subseteq C$, and assume $(N,\zeta)$ is factorizable over $U$ in the sense of Definition~\ref{def: factorizable representation}.  Then the pairs $(\mathfrak{p}_J, j_! \mathscr{F}_{V \boxtimes W, \zeta})$ are $\Psi$-factorizable, where $\mathfrak{p}_J$ denotes the composite map
\[\Cht_{G,I \cup J}|_{U^{I\cup J}} \to U^{I \cup J} \to U^J.\]
\end{theorem}

\begin{proof}
Since $\mathscr{F}_{V \boxtimes W,\zeta} \cong \epsilon^* (\mathcal{S}_{V \boxtimes W} \boxtimes \zeta)$, it suffices to prove that $(\mathfrak{p}_J, j_!(\mathcal{S}_{V \boxtimes W} \boxtimes \zeta))$ are $\Psi$-factorizable where $\mathfrak{p}_J$ now denotes the composite map $\Cht\mathcal{R}_{G, N, I \cup J}|_{U^{I \cup J}} \to U^{I \cup J} \to U^J$ and $j$ now denotes the open embedding $\Cht\mathcal{R}_{G, N, I \cup J}|_{(C \setminus (N \cup R))^{I\cup J}} \to \Cht\mathcal{R}_{G, N, I \cup J}|_{U^{I \cup J}}$.  We note that the data of the sheaf $\mathcal{S}_{V \boxtimes W} \boxtimes \zeta$ and its nearby cycle descend along the action maps and projection maps in equation~\ref{HResolution}.  Let $u \to t \to s$ be a composition of specializations of geometric points on $U^J$.  We can handle the descent data by noting that squares of the form
\begin{equation}
    \begin{tikzcd}
    R(\Psi_{\mathfrak{p}_J})_{u}^{s} a^* \mathcal{S}_{V \boxtimes W} \boxtimes \zeta \arrow[r] \arrow[d] & R(\Psi_{\mathfrak{p}_J})_{t}^{s} R(\Psi_{\mathfrak{p}_J})_{u}^{t} a^* \mathcal{S}_{V \boxtimes W} \boxtimes \zeta \arrow[d] \\
    R(\Psi_{\mathfrak{p}_J})_{u}^{s} p^* \mathcal{S}_{V \boxtimes W} \boxtimes \zeta \arrow[r] & R(\Psi_{\mathfrak{p}_J})_{t}^{s} R(\Psi_{\mathfrak{p}_J})_{u}^{t} p^* \mathcal{S}_{V \boxtimes W} \boxtimes \zeta
    \end{tikzcd}
\end{equation}
commute where the vertical maps are isomorphisms given by the descent data on $\mathcal{S}_{V \boxtimes W} \boxtimes \zeta$.  Therefore, we can reduce to checking that $(\mathfrak{p}_J, j_!(\mathcal{S}_{V \boxtimes W} \boxtimes \zeta))$ is $\Psi$-factorizable after pulling back to $\Gr^{\triv(N)}_{G,I}$, where $\mathcal{S}_{V \boxtimes W}$ are sheaves constructed by geometric Satake and $\zeta$ is now a local system on $H$.  But this $\Psi$-factorizability statement is our hypothesis on $(N,\zeta)$.
\end{proof}

\begin{remark}
In \cite[Remarque~3.5]{genestier2017chtoucas}, Genestier and Lafforgue wonder whether the nearby cycles for restricted shtukas over general bases is equivalent to iterated nearby cycles.  The above proof gives a criterion under which these sheaves are equivalent in a special case where $r = 0$.
\end{remark}

\begin{remark}
Let $N$ be supported on a single point $y$ and $H = G_y / G_y^{a+}$.  Consider the categories
\begin{align*}
    \mathscr{C}_1 &= D\left(\frac{[G_{y}^{a+} \backslash LG / G_{y}^{a+}]}{\Ad_{\Frob}(H)}, \qlbar\right), \\
    \mathscr{C}_{n+2} &= D\left(\frac{[G_{y}^{a+} \backslash LG / G_{y} \times (G_y \backslash LG / G_y)^n \times G_{y} \backslash LG / G_{y}^{a+}]}{\Ad_{\Frob}(H)}, \qlbar\right).
\end{align*}
For any fixed sheaf $\zeta$ on $H / \Ad_{\Frob}(H)$ giving a $\Psi$-factorizable representation, the nearby cycles on restricted shtukas can be seen as giving a functor
\[ Z_{n,\zeta} \colon \Rep(\widehat{G}^n) \rightarrow \mathscr{C}_n. \]
Note that $\mathscr{C}_n$ is not naturally monoidal, but is equipped with certain contraction maps $m_g \colon \mathscr{C}_n \to \mathscr{C}_m$ for each surjective, order-preserving map $g \colon [n] \to [m]$.

Over a two-dimensional base, the fusion property of nearby cycles exhibited above can be seen as encoding a braiding compatibility analogous to \cite[Property~3]{gaitsgory2004braiding}.
\begin{equation}
    \begin{tikzcd}
    m_{2,1} Z_{2,\zeta}(V \boxtimes W) \arrow{r} \arrow{d} & m_{2,1} Z_{2,\zeta}(W \boxtimes V) \arrow{d} \\
    Z_{1,\zeta}(V \otimes W) \arrow{r} & Z_{1,\zeta}(W \otimes V).
    \end{tikzcd}
\end{equation}
\end{remark}

\subsection{Applying the Zorro lemma}

We keep the notations of $N$ and $H$ from the previous section, and let $s \in N$ be a fixed point.  Let $\mathfrak{p}_J$ denote the map $\Cht_{G,I \cup J} \to C^{I \cup J} \to C^J$.  We let $\Psi_J K$ be $R(\Psi_{\mathfrak{p}_J})_{\overline{\eta_J}}^{\Delta(s)}$ be a nearby cycles over arbitrary bases.  The following is an analogue of \cite[Proposition~3.3]{salmon2023unipotent}.  Let $\zeta \in \Rep(H(\mathbb{F}_q))$ be a finite-dimensional representation in $\Lambda$-vector spaces, where $\Lambda$ is a finite extension of $\mathbb{Q}_{\ell}$ or $\qlbar$, which we also consider as a sheaf on $BH(\mathbb{F}_q)$ so we can form $\mathcal{S}_{I,V} \boxtimes \zeta$ on $\Cht\mathcal{R}_{G,N,I}$ which pulls back to $\mathscr{F}_{I, V, \zeta}$ on $\Cht_{G, \Xi, I}$.  We also consider restrictions of this sheaf to various Harder-Narasimhan truncations for any dominant coroot of $G$, which we denote $\mathscr{F}^{\le \mu}_{I, V, \zeta}$, which lives in the bounded derived category of constructible sheaves on a finite type Deligne-Mumford stack.

\begin{proposition}
The functor $(I,V) \mapsto \mathfrak{p}_! \Psi_J \mathscr{F}^{\le \mu}_{I \cup J, W \boxtimes V, \zeta}$ is a collection of cocartesian functors $\Rep(\widehat{G}^{\bullet}) \to D^b_c((C \setminus (N \cup R))^{\bullet}, \Lambda)$, filtered by partial Frobenius, between categories cofibered over $\FinS$.
\end{proposition}

\begin{proof}
For any Harder-Narasimhan truncation, $\mu$, we check constructibility of the pushforward.  Since $\mathscr{F}^{\le \mu}_{I \cup J, W \boxtimes V, \zeta}$ is $\Psi$-good, we conclude that $\Psi_J \mathscr{F}^{\le \mu}_{I \cup J, W \boxtimes V, \zeta}$ is constructible, so the pushfoward by $\mathfrak{p}$ is constructible.  So the functor $\Rep(\widehat{G}^{I}) \to D^b_c((C \setminus (N \cup R))^I, \Lambda)$ is well-defined for any $I$.

We need to check that the cocartesian arrows are sent to cocartesian arrows.  Let $\mu \colon I_1 \to I_2$ be any morphism of finite sets and let $\Delta_{\mu} \colon C^{I_2} \to C^{I_1}$ be the map on a power of a curve, which restricts to the locus $(C \setminus (N \cup R))^{\bullet}$.  Just as in \cite[Proposition~3.5]{salmon2023unipotent}, we can reduce to the following statement on $\Gr_{G,I \cup J}^{\triv(N)}$
\begin{equation}
\Delta_{\mu}^* \Psi_J (\mathcal{S}_{I_1 \cup J, V \boxtimes W} \boxtimes \zeta) \cong \Psi_J (\mathcal{S}_{I_2 \cup J, V^{\mu} \boxtimes W} \boxtimes \zeta).
\end{equation}
Using properness of convolution on $\Gr_{G,I_i \cup J}^{(I_i, J),\triv(N)}$, this can be further reduced to
\begin{equation}
\Delta_{\mu}^* \mathcal{S}_{I_1,V} \tilde{\boxtimes} \Psi_J (\mathcal{S}_{J, W} \boxtimes \zeta) \cong \mathcal{S}_{I_2,V^\mu} \tilde{\boxtimes} \Psi_J (\mathcal{S}_{J, W} \boxtimes \zeta),
\end{equation}
where the result now follows by fusion of Satake sheaves on the Beilinson-Drinfeld grassmannian.

Finally, we need to check that these functors are filtered by partial Frobenius.  By the discussion in \cite[Proposition~3.5]{salmon2023unipotent} and \cite[Proposition~3.1, Proposition~3.3]{lafforgue2018chtoucas}, it suffices to prove that
\begin{equation}
    \left( \mathrm{Fr}_{I_1}^{(I_1, \dots, I_n, J)} \right)^* \left( \Psi_J \mathscr{F}^{\le \mu + \kappa}_{I \cup J, V \boxtimes W, \zeta} \right) \cong \Psi_J \mathscr{F}^{\le \mu}_{I \cup J, V \boxtimes W, \zeta}
\end{equation}
for any partition $(I_1, \dots, I_n)$ of $I$.  Since partial Frobenius are local homeomorphisms, they commute with nearby cycles.  The only difference from \cite[Proposition~3.3]{lafforgue2018chtoucas} is that when $V \boxtimes W$ is irreducible, we identify that
\[\mathscr{F}_{I \cup J, V \boxtimes W, \zeta}\]
as IC sheaf extensions for a local system on the open stratum pulling back the corresponding Schubert cell on the Beilinson-Drinfeld grassmannian, rather than just the IC extension of the constant sheaf.  The local system in question is the pullback of the local system $\Lambda \boxtimes \zeta$ on the open Schubert cell in $[G_{I,\infty} \backslash \Gr_{G,I\cup J,V \boxtimes W}] \times [U \backslash BH(\mathbb{F}_q)]$ corresponding to the representation $V \boxtimes W$, and the identification follows by smooth base change, using Theorem~\ref{thm:smooth model}.  That is, if $j$ is the inclusion of this cell, there is an identification
\[ j_{!*} \epsilon^* (\Lambda \boxtimes \zeta) \cong \epsilon^* j_{!*} (\Lambda \boxtimes \zeta) \cong \mathscr{F}_{I \cup J, V \boxtimes W, \zeta}. \]
\end{proof}

The argument proving the following theorem originated in the work of Xue \cite[Section~3.2]{xue2020smoothness} and is the key to making the pushforward $\mathfrak{p}_!$ commute with nearby cycles $\Psi_{J}$.

\begin{theorem}\label{thm:xuezorro}
For the functor $\mathscr{G} = \mathscr{F}_{I \cup J, W \boxtimes V, \zeta}$ as above, assume further that $(N,\zeta)$ is factorizable over $U$ in the sense of Definition~\ref{def: factorizable representation}.  Then the natural map
\begin{equation}
\mathfrak{p}_! \Psi_{J} \mathscr{G}_{I \cup J, W \boxtimes V} \to \Psi_{J} \mathfrak{p}_! \mathscr{G}_{I \cup J, W \boxtimes V}
\end{equation}
is an isomorphism.
\end{theorem}

\begin{remark}
    This holds over points with parahoric reduction or unipotent radical of the Iwahori reduction by Proposition~\ref{prop:psi-factorizable grassmannian}.  Later, Theorem~\ref{thm: psi-factorizable GFq} will prove that this also holds over $\F_q$-points $y$ where $G_y$ is the positive loop group $L^+ G$ and the divisor $N$ contains $y$ with multiplicity $0$, so that $G_y^{0+}$ corresponds to the unipotent radical of the positive loop group.
\end{remark}

\begin{proof}
We construct the inverse map.  Let $J_1, J_2, J_3$ be three identical, disjoint copies of $J$.
\begin{align*}
    \Psi_J R\mathfrak{p}_! \mathscr{G}_{I \cup J,W \boxtimes V} &\cong \Psi_{J_1} R\mathfrak{p}_! \Psi_{J_2} \mathscr{G}_{I \cup J_1 \cup J_2, W \boxtimes V \boxtimes 1} \\
    &\to \Psi_{J_1} R\mathfrak{p}_! \Psi_{J_2} \mathscr{G}_{I \cup J_1 \cup J_2, W \boxtimes V \boxtimes (V^* \otimes V)} \\
    &\cong \Psi_{J_1} R\mathfrak{p}_! \Psi_{J_2} \Psi_{J_3} \mathscr{G}_{I \cup J_1 \cup J_2 \cup J_3, W \boxtimes V \boxtimes V^* \boxtimes V} \\
    &\to \Psi_{J_1} \Psi_{J_2} R\mathfrak{p}_! \Psi_{J_3} \mathscr{G}_{I \cup J_1 \cup J_2 \cup J_3, W \boxtimes V \boxtimes V^* \boxtimes V} \\
    &\cong \Psi_{J_2} R\mathfrak{p}_! \Psi_{J_3} \mathscr{G}_{I \cup J_2 \cup J_3, W \boxtimes (V \otimes V^*) \boxtimes V} \\
    &\to \Psi_{J_2} R\mathfrak{p}_! \Psi_{J_3} \mathscr{G}_{I \cup J_2 \cup J_3, W \boxtimes 1 \boxtimes V} \\
    &\cong R\mathfrak{p}_! \Psi_{J} \mathscr{G}_{I \cup J, V \boxtimes W}.
\end{align*}
The remaining considerations are exactly identical to \cite[Theorem~5.2]{salmon2023unipotent}, except now nearby cycles are not unipotent but taken over general bases.
\end{proof}

\begin{remark}
If $s$ is an unramified point, this recovers Xue's original argument showing that specialization maps are isomorphisms, using the fact that Satake sheaves are ULA over the unramified locus \cite{richarz2014new}.  If $s$ is a point with parahoric reduction, this recovers \cite[Theorem~5.2]{salmon2023unipotent}.
\end{remark}

\section{Lusztig's horocycle correspondence for restricted shtukas}\label{sec:horocycle}

\subsection{Review of non-unipotent representations}\label{subsec:review nonunipotent}

We now further restrict our assumption on $G$.  Let $G$ be a split reductive group with connected center, with a Frobenius action $\sigma$ that gives a $\mathbb{F}_q$-rational structure.  Our assumption of connected center will be to simplify the discussion of blocks.  In this section, we fix a Borel $B$ and a maximal torus $T$, such that $\sigma(B) = B$ and $\sigma(T) = T$. Recall that the Weyl group $W$ can be identified with $N(T) / T$ where $N(T)$ is the normalizer of $T$, and it is equipped with a map $N(T) \to W$ which admits a cross-section $W \to N(T)$, which we denote by $w \mapsto \dot{w}$ so that if $\ell(w) + \ell(w') = \ell(ww')$, then $\dot{(ww')} = \dot{w} \dot{w'}$.  Let $U$ be the unipotent radical of $B$.  Throughout this section and the next, we will use $D(X)$ as an abbreviation for the derived category $D^b_c(X, \qlbar)$ of constructible sheaves with $\qlbar$ coefficients.

\subsubsection{Rank $1$ character sheaves and monodromic sheaves on $U \backslash G / U$}

Let $m \colon T \times T \to T$ be the multiplication map and let $e \colon \Spec \fqbar \to T$ be the identity.  A rank one character sheaf on $T$ is a local system $\mathcal{L}$ such that
\[
\begin{aligned}
    m^*(\mathcal{L}) &\cong \mathcal{L} \boxtimes \mathcal{L}, \\
    e^*(\mathcal{L}) &\cong \qlbar,
\end{aligned}
\]
satisfying associativity and commutativity axioms as in \cite[Appendix~A]{yun2014rigidity}.  All our character sheaves will be taken to live over $\fqbar$.  Such character sheaves will satisfy $\mathcal{L}^{\otimes n} \cong \qlbar$ for some $n > 0$.

Let's suppose we give $T$ an $\mathbb{F}_q$-rational structure that is not necessarily split.  A rank one character sheaf on $T$ over $\mathbb{F}_q$ defines a trace function $t_{\mathcal{L}} \colon T(\mathbb{F}_q) \to \qlbar^{\times}$.  The Lang isogeny itself gives a map
\[ \pi^t_1(T/\mathbb{F}_q) \to T(\mathbb{F}_q), \]
and the trace function is related to the representation of $\pi^t_1(T/\mathbb{F}_q)$ as a local system by the composition \cite[Lemma~2.14]{sawintemplier2021ramanujan}
\[
\begin{tikzcd}
    \pi^t_1(T/\mathbb{F}_q) \arrow{r}& T(\mathbb{F}_q) \arrow{r}{t_{\mathcal{L}}^{-1}}& \qlbar^{\times}.
\end{tikzcd}
\]
On the other hand, $\pi^t_1(T/\mathbb{F}_q) \cong X_*(T) \otimes_{\mathbb{Z}} \widehat{\mathbb{Z}}'(1)$ using the prime-to-$p$ completion of the integers, so choosing a generator for the tame fundamental group gives a map $X_*(T) \to \qlbar^{\times}$ and thus an element in $X^*(T) \otimes \qlbar^{\times} \cong \widehat{T}(\qlbar)$.

We denote $G / B$ as $\mathcal{B}$ and $G / U$ as $\btilde$.  We let $T \times T$ act on $U \backslash G / U \cong G \backslash (\btilde \times \btilde)$ by
\begin{equation}
\begin{aligned}
    a \colon T \times U \backslash G / U \times T &\to U \backslash G / U \\
    a(t',g,t) &= t'gt^{-1}.
\end{aligned}
\end{equation}
Following \cite{lusztig2015non} and \cite{lusztig2020endoscopy}, we can consider $T \times T$-monodromic sheaves on $U \backslash G / U$ or equivalently, $G$-equivariant monodromic sheaves on $\btilde^2$.  Specifically, if $\mathcal{L}$ and $\mathcal{L}'$ are two rank $1$ character sheaves on $T$, we say that a sheaf $M$ on $U \backslash G / U$ is $(T \times T, \mathcal{L}' \boxtimes \mathcal{L})$-monodromic if
\[ a^*(M) \cong \mathcal{L}' \boxtimes M \boxtimes \mathcal{L}^{-1}. \]
We denote the category of $(T \times T, \mathcal{L}' \boxtimes \mathcal{L}^{-1})$-monodromic sheaves as ${}_{\mathcal{L}'} D_{\mathcal{L}}(U \backslash G / U)$.  We note that these subcategories are denoted ${}_{\mathcal{L'}} \mathcal{\underline{D}}_{\mathcal{L}}$ in \cite[Section~2.7]{lusztig2020endoscopy}.  Inside this derived category, there is a heart of a perverse $t$-structure giving an abelian monodromic Hecke category ${}_{\mathcal{L}'} \mathcal{P}_{\mathcal{L}}(U \backslash G / U)$ consisting of perverse sheaves.

The connected center assumption on $G$ will simplify some of the details in \cite{lusztig2020endoscopy}, and we do not believe it is an essential assumption.  For a character sheaf $\mathcal{L}$ on $T$, there is a group $W_{\mathcal{L}} \subset W$ that stabilizes $\mathcal{L}$.  There is a subgroup $W_{\mathcal{L}}^{\circ} \subset W_{\mathcal{L}}$ that is the Weyl group of a root system defined in terms of cocharacters whose pullback trivializse $\mathcal{L}$.  When the center is connected, these two groups coincide.  As a result, the category ${}_{\mathcal{L}'} D_{\mathcal{L}}(U \backslash G / U)$ only contains a single block, and two-sided cells will be in bijection with unipotent conjugacy classes in the endoscopic group.

\subsubsection{The stack $G /_{\sigma} G$}
To describe representations of the finite group $G(\mathbb{F}_q)$, we need to the consider action of $G$ on itself by $g.h = gh\sigma(g)^{-1}$. This action is transitive and the stabilizer at identity 1 is $G^{\sigma} = G(\mathbb{F}_q)$. Therefore, $G$-equivariant simple perverse sheaves on $G$ are equivalently the $\qlbar$-representations of the finite group $G(\mathbb{F}_q)$.  We will temporarily adopt the perverse shift by $\Delta =: \dim G$, noting that when we make the connection to global shtukas, this shift will no longer be used for technical reasons.  We denote the category of $G$-equivariant sheaves on $G$ with the Frobenius-twisted action by $D_G(G) = D(G /_{\sigma} G) = D(G / \Ad_F(G))$.  Crucially, unlike many other categories we will consider, this category is semisimple.

\subsubsection{The ``finite'' horocycle correspondence}
We define varieties
\[\dot{Z} = \{(B', B'', g)) \in \mathcal{B} \times \mathcal{B} \times G| g\sigma(B')g^{-1} = B''\}\]
and 
\[Z = \{(B', B'', g\sigma(U_{B'})) \in \mathcal{B} \times \mathcal{B} \times G/\sigma(U_{B'})| g\sigma(B')g^{-1} = B''\}.\]
There is a smooth map
\begin{equation}
\begin{aligned}
f \colon \dot{Z} &\to Z \\
(B', B'', g) &\mapsto (B', B'', g\sigma(U_{B'})).
\end{aligned}
\end{equation}
The $G$-action on $\dot{Z}$ is given by $h \colon (B', B'', g) \mapsto (hBgh^{-1}, hB'h^{-1}, hg\sigma(h)^{-1})$ and similarly on $Z$. We have a proper map $\pi\colon \dot{Z} \to G$ being the projection to the third component. Both $f$ and $\pi$ are $G$-equivariant. We will use correspondences to relate sheaves on $G$ and sheaves on $Z$. In particular, we have an induction functor $\chi= \pi_!f^* \colon D(Z) \to D(G)$ which is compatible with equivariant structures.

The variety $Z$ is closely related to $\btilde^2$. 
In fact, we have a map
\begin{equation}\label{eqn:gu epsilon}
\epsilon \colon \btilde^2 \to Z,
\end{equation}
by $(xU, yU) \mapsto (xBx^{-1}, yBy^{-1}, yU\sigma(x)^{-1})$. 
There is an action of $T$ on $\btilde^2$ given by $t \colon (xU, yU) \mapsto (xtU, y\sigma(t)U)$. It is easy to see that $\epsilon$ induces a map $\epsilon \colon T\backslash \btilde^2 \to Z$ and in \cite[Section~4.1(a)]{lusztig2016non} Lusztig shows this is an isomorphism. Since we now allow monodromy, simple objects on $Z$ are labeled by two parameters $\lambda \colon T \to \qlbar$ and $w\in W$.

Define the variety $X$ as the fiber product $\dot{Z} \times_Z \btilde^2$.  So our horocycle correspondence consists of the following diagram:
\begin{equation}\label{eqn: finite horocycle correspondence}
    \begin{tikzcd}
        & X \arrow{ld} \arrow{d} \arrow{r} & \btilde^2 \arrow{d}{\epsilon} \\
        G & \dot{Z} \arrow{l}{\pi} \arrow{r}{f} & Z.
    \end{tikzcd}
\end{equation}
Using the definition of the map $\epsilon$ and $f$, it follows that the moduli problem for $X$ can be written as classifying triples
\begin{equation}\label{eqn: X moduli problem}
    X = \{ (xU, yU, g) \in \btilde \times \btilde \times G | g \sigma(xU) = yU \}.
\end{equation}

A sheaf in ${}_{L_{\lambda'}} D_{L_{\lambda}}(U \backslash G / U)$ descends to $Z$ if and only if $\sigma(\lambda') = \lambda$.  When we consider $\pi_! f^* M$ for a sheaf $M$ in ${}_{L_{\lambda'}} D_{L_{\lambda}}(U \backslash G / U)$, it is equivalent to take $\pi_! f^* \epsilon_! M$ with the sheaf $\epsilon_! M$ on $D(Z)$ that it descends to, noting that this is $0$ if $\sigma(\lambda') \ne \lambda$.

\subsubsection{The sheaves $\mathbb{L}_{\lambda}^{\dot{w}}$ and $\mathcal{L}_{\lambda}^{\dot{w}}$}
A semisimple parameter is a $W$-orbit of torus characters of finite order.  That is, let $T_n$ denote $n$-division points of $T$, and set $\mathfrak{s}_n = \textrm{Hom}(T_n, \qlbar^{\times})$.  We have $\mathfrak{s}_n \subset \mathfrak{s}_{n'}$ if $n'/n \in \mathbb{Z}$, and form the union $\mathfrak{s}_{\infty}$, so each $\mathfrak{s}_n$ and $\mathfrak{s}_{\infty}$ admits an action by $W$ and also by $\sigma$ where $\sigma(\lambda)(t) = \lambda(\sigma^{-1}(t))$.  Set $I = W \times \mathfrak{s}_{\infty}$.  We write $w \cdot \lambda$ instead of $(w, \lambda)$ and the action of $w$ on $\lambda$ by $w(\lambda)$.  We say that $w \cdot \lambda$ is fixed by $\sigma$ if $w(\lambda) = \sigma^{-1}(\lambda)$, which defines the set $I^1$ in \cite[Section~2.5]{lusztig2015non}.

Here $\mathcal{O}_w$ is the $G$-orbit in $\mathcal{B}^2$ labeled by $w$. In order words, $B'$ and $B''$ have relative position $w$. Let $\tilde{\mathcal{O}}_w$, resp. $Z_w$, be the preimage of $\mathcal{O}_w$ under the map $\btilde^2 \to \mathcal{B}^2$, resp. $Z \to \mathcal{B}^2$. For $\omega \in N(T)$ such that $\kappa(\omega) = w$, we define a map $j^{\omega} \colon \tilde{\mathcal{O}}_w \to T$ by $j^{\omega}(xU, yU) = t$ where $t \in T$ is the unique element such that $Ux^{-1} yU = U \omega t U$.

If $\lambda \in \mathfrak{s}_n$, then since $T_n$ is the kernel of the $n$th power map, we can take the push forward of the trivial local system $[n]_*(\qlbar)$ along the $n$th power map and take the $\lambda$-isotypic component.  The choice of using the $\lambda$-isotypic component is the \emph{opposite} of the sign convention in \cite[(A.5)]{yun2014rigidity}.  In particular, under our convention, the trace of the local system $L_{\lambda}$ gives a restriction of $\lambda^{-1}$ to $T(\mathbb{F}_q)$-points and the representation attached to the local system $L_{\lambda}$ is the composition $\pi_1^t(T) \to T_n \to \qlbar$ where the second map is $\lambda$.

This defines a rank 1 character sheaf $L_{\lambda}$ on $T$ labeled by $\lambda$.  For any $\omega \in N(T)$, we obtain a rank 1 local system $L_{\lambda}^\omega = (j^{\omega})^* L_{\lambda}$ on $\tilde{\mathcal{O}}_w$.  In fact, all $L_{\lambda}^{\omega}$ are non-canonically isomorphic to $L_{\lambda}^{\dot{w}}$ for any $\omega \in \kappa^{-1}(w)$. We denote by $\mathcal{L}_{\lambda}^{\dot{w}}$ the extension by $0$ and $\mathbb{L}_{\lambda}^{\dot{w}}$ the intermediate extension of $L_{\lambda}^{\dot{w}}$. $\mathbb{L}_{\lambda}^{\dot{w}}$ is a simple perverse sheaf.  In particular, $\chi(\mathbb{L}_{\lambda}^{\dot{w}}) = 0$ if $w \cdot \lambda$ is not fixed by $\sigma$.

We can consider $\mathbb{L}_{\lambda}^{\dot{w}}$ as living in a category of monodromic sheaves.  The local system $L_{\lambda}$ lives in $T /_n T$ where $T$ acts on itself through the modified action $t_1 \cdot t_2 = t_1^n t_2$.  The pullback along $j^{\dot{w}}$ gives a $G \times T$-equivariant sheaf on $\tilde{\mathcal{O}}_w$ that descends to $Z_w$ if $w(\lambda) = \sigma^{-1}(\lambda)$.  So the sheaves $\mathbb{L}_\lambda^{\dot{w}}$ and $\mathcal{L}_{\lambda}^{\dot{w}}$ are naturally $G \times T \times T$-equivariant sheaves on $\btilde^2$ where the $T$ action on $\btilde$ is by $t \cdot gU = g t^n U$.  Equivalently, we may consider these as sheaves on $(U \backslash G / U) /_n (T \times T)$.  In fact, both $\mathcal{L}_{\lambda}^{\dot{w}}$ and $\mathbb{L}_{\lambda}^{\dot{w}}$ are $(T \times T, L_{w(\lambda)} \boxtimes L_{\lambda}^{-1})$-monodromic and can be viewed as living in ${}_{L_{w(\lambda)}} D_{L_{\lambda}}(U \backslash G / U)$.

We define $D_G(\btilde^2)$ as the category limit over $n$ of $\varinjlim_n D((U \backslash G / U) /_n (T \times T))$ where the arrows in the limit are pullback maps coming from arrows $n \to n'$ where $n | n'$ are positive integers.  There is a monoidal structure on $D_G(\btilde^2)$.  Let $p_{ij} \colon \btilde^3 \to \btilde^2$ be projection identifying the $i$ and $j$ copies together.  Then for sheaves $M_1$ and $M_2$ on $\btilde^2$, we define the convolution $M_1 * M_2$ as $(p_{13})_! (p_{12}^* M_1 \otimes p_{13}^* M_2)$.  The category $D_G(\btilde^2)$ natural contains as subcategories the categories ${}_{L_{\lambda'}} D_{L_{\lambda}}(U \backslash G / U)$.  As in \cite[Section~3]{lusztig2020endoscopy}, the convolution above defines a collection of bifunctors
\[ * \colon {}_{\mathcal{L}''} D_{\mathcal{L}'}(U \backslash G / U) \times {}_{\mathcal{L}'} D_{\mathcal{L}}(U \backslash G / U) \to {}_{\mathcal{L}''} D_{\mathcal{L}}(U \backslash G / U). \]
For the monoidal structure on ${}_{L_{\lambda}} D_{L_{\lambda}}(U \backslash G / U)$ with respect to which $L_{\lambda}^{1}$ is a monoidal unit.

Now we look at the image of these sheaves under $\chi$. Notice that $\pi^{-1}(1) \cap f^{-1}(Z_w) = X_w$, the Deligne-Lusztig variety on $G$. Similarly, there is a local system $\mathcal{F}_{\lambda}^{\dot{w}}$ on $X_w$ that restricts $f^*(\mathcal{L}_{\lambda}^{\dot{w}})$ to the fiber over $1$. The image $\chi(\mathcal{L}_{\lambda}^{\dot{w}})^j$ ($j$-th cohomology sheaf) is a $G$-equivariant local system on $G$, and the stalk at 1 is $H_c^j(X_w, \mathcal{F}_{\lambda}^{\dot{w}})$. 
The image $\chi(\mathbb{L}_{\lambda}^{\dot{w}})^j$ has stalk at 1 being $IH^j(X_w, \mathcal{F}_{\lambda}^{\dot{w}})$. The cohomology of these varieties has the structure of a $G(\mathbb{F}_q)$-representation, and we may talk about multiplicities of an irreducible representation $\rho$ inside another representation $W$, denoted $(\rho : W) = \dim \Hom(\rho, W)$.
Recall from \cite{deligne1976representations} and \cite{Lusztig1984}, for any irreducible representation $\rho$ of $G(\mathbb{F}_q)$,
\begin{enumerate}
    \item there exists $w \cdot \lambda$ fixed by $\sigma$, such that $(\rho: \oplus_j H_c^j(X^w, \mathcal{F}_{\lambda}^{\dot{w}})) \neq 0$, and 
    \item there exists $w \cdot \lambda$ fixed by $\sigma$, such that $(\rho: \oplus_j IH^j(X^w, \mathcal{F}_{\lambda}^{\dot{w}})) \neq 0$. 
\end{enumerate}
Let us restate the above in terms of sheaves.  To each irreducible representation $\zeta \in \Rep(G(\mathbb{F}_q))$, we attach an irreducible local system on $G /_{\sigma} G$.  In the language of the above proposition, for every irreducible $\zeta$, there is a sheaf $\mathbb{L}_{\lambda}^w$ such that $\zeta$ appears in $H^j(\pi_* f^* \mathbb{L}_{\lambda}^w)$ for some $j$, that is, $\zeta$ is a simple factor of $\chi(\mathbb{L}_{\lambda}^{\dot{w}})$ for some $(w, \lambda)$.  Following the notation of Lusztig \cite[Section~6.1]{lusztig2016non}, let $\mathfrak{R}^{\dot{w}}_{\lambda} = \chi(\mathcal{L}_{\lambda}^{\dot{w}})$ and $R^{\dot{w}}_{\lambda} = \chi(\mathbb{L}_{\lambda}^{\dot{w}})$.

The set of $\lambda \in \mathfrak{s}_\infty$ such that this multiplicity is nonvanishing is a closed orbit under $W$ and defines a semisimple parameter $\mathfrak{o}$, which we identify with this $W$-orbit $W\lambda$.  Later, we will see how this semisimple parameter will give us the semisimple conjugacy class containing the image of the tame generator in $\widehat{G}$ for Langlands parameters coming from the cohomology of shtukas.  Philosophically, we should then believe there is some refinement of this semisimple data that captures the unipotent part of the image of the tame generator.  Such a unipotent part is a unipotent conjugacy class in the endoscopic group, which is in an order-preserving bijection with two-sided cells.

\subsubsection{Two-sided cells}
There is a notion of two-sided cells $\textbf{c}$ attached to these $w\cdot \lambda$.  The set $I = W \times \mathfrak{s}_\infty$ admits a partition into two-sided cells, that is, $I = \sqcup_{\textbf{c}} \textbf{c}$.  By our assumption of connected center, the two-sided cells in $W \times \mathfrak{o}$ where $\mathfrak{o} = W \lambda$ are in bijection with two-sided cells for the stabilizer $W_{\lambda}$.  This bijection is characterized by the property that $W_{\lambda} \times \{ \lambda \} \cap \textbf{c}$ is a two-sided cell for $W_{\lambda}$.  In the general case where the center is not assumed to be connected, see \cite[11.4]{lusztig2020endoscopy} for a characterization of two-sided cells.

Each two-sided cell is contained in the finite set $W \times \mathfrak{o}$ for some unique $W$-orbit $\mathfrak{o} \subset \mathfrak{s}_{\infty}$.  If $\widehat{H} \subset \widehat{G}$ is the centralizer of $\lambda \in \widehat{T}$, then we call $H$ the endoscopic group.  If $\mathfrak{o} = W \lambda$ as an orbit, two-sided cells in $W \times \mathfrak{o}$ are in bijection with two-sided cells of the Weyl group of $H$, which is in turn in bijection with unipotent conjugacy classes in $\widehat{H}$.

For a two-sided cell $\textbf{c}$, we consider the subcategory $D^{\textbf{c}}_G(\btilde^2)$, resp. $D^{\le \textbf{c}}_{G}(\btilde^2)$, $D^{< \textbf{c}}_{G}(\btilde^2)$, of $D_G(\btilde^2)$ whose simple factors are $\mathbb{L}_{\lambda}^{\dot{w}}$ for $w \cdot \lambda \in \textbf{c}$, resp. $w \cdot \lambda \in \textbf{c}'$ for $\textbf{c}' \le \textbf{c}$, $< \textbf{c}$.  Denote the heart of perverse sheaves in these derived categories by $\mathcal{P}^{\textbf{c}}_G(\btilde^2)$, $\mathcal{P}^{\le \textbf{c}}_G(\btilde^2)$, and $\mathcal{P}^{< \textbf{c}}_G(\btilde^2)$, respectively.  Lusztig also attaches a two-sided cell to each irreducible representation of $G(\mathbb{F}_q)$.  Let $\Rep(G(\mathbb{F}_q))_{\textbf{c}}$ be the category whose objects are sums of irreducible representations attached to $\textbf{c}$.  Similarly, let $\mathcal{P}(G /_{\sigma} G)_{\textbf{c}}$, resp. $D(G /_{\sigma} G)_{\textbf{c}}$ be the subcategory of the category consisting of perverse sheaves, resp. complexes of sheaves that correspond to representations in $\Rep(G(\mathbb{F}_q))_{\textbf{c}}$.

There is a related category ${}_{L_{\lambda}} \mathcal{P}_{L_{\lambda}}(U \backslash G / U)_{\textbf{c}}$ defined as a Serre quotient.  For $W_{\lambda}$ the stabilizer of $\lambda$ in $W$, we can consider the category of perverse sheaves ${}_{L_{\lambda}} \mathcal{P}_{L_{\lambda}}(U \backslash G / U)_{\le \textbf{c}}$, resp. ${}_{L_{\lambda}} \mathcal{P}_{L_{\lambda}}(U \backslash G / U)_{< \textbf{c}}$ such that simple composition factors are $\mathbb{L}^{\dot{w}}_{\lambda}$ for $w \in \textbf{c}$.  Then we define the Serre quotient category
\begin{equation}
{}_{L_{\lambda}} \mathcal{P}_{L_{\lambda}}(U \backslash G / U)_{\textbf{c}} := {}_{L_{\lambda}} \mathcal{P}_{L_{\lambda}}(U \backslash G / U)_{\le \textbf{c}} / {}_{L_{\lambda}} \mathcal{P}_{L_{\lambda}}(U \backslash G / U)_{< \textbf{c}}
\end{equation}
whose simple objects are of the form $\mathbb{L}_{\lambda}^{\dot{w}}[|w|]$ for $w \in W_{\lambda}$, the subset of $W$ fixing $\lambda$.  We denote the corresponding derived category as ${}_{L_{\lambda}} D_{L_{\lambda}}(U \backslash G / U)_{\textbf{c}}$.

There is an $a$-function from two-sided cells to integers.  The main role that the $a$-function will play for us is in defining truncated induction and truncated convolution.  For an element $z \cdot \lambda \in \textbf{c}$, we define the number $n_z = a(\textbf{c}) + |z| + \Delta$, where $|z|$ is the length of $z$ in the Weyl group.  For any representation $\rho$, there exists $w \cdot \lambda$ fixed by $\sigma$, such that $(\rho: IH_c^{n_w}(X^w, \mathcal{F}_{\lambda}^{\dot{w}})) \neq 0$.

We restate this result on the level of sheaves.  The following proposition is from \cite[Section~6.1-6.2]{lusztig2016non}.
\begin{proposition}[Lusztig]
Every irreducible perverse sheaf $\zeta$ in $\mathcal{P}(G /_{\sigma} G)_{\textbf{c}}$ is a summand of a sheaf of the form $(R^{\dot{w}}_{\lambda})^{n_w}$ for some $w \cdot \lambda$.  Moreover, if $\zeta$ is a summand of $R^{\dot{w}}_{\lambda}$, then the two-sided cell $\textbf{c}'$ containing $w \cdot \lambda$ must satisfy $\textbf{c} \le \textbf{c}'$.
\end{proposition}

In light of the above proposition, we define the functor
\[\underline{R^{n_w} \pi_! f^*} = \underline{\chi} \colon \mathcal{P}^{\textbf{c}}(\btilde^2) \to \mathcal{P}(G /_{\sigma} G)_{\textbf{c}} \]
as the \textit{truncated induction} functor, where the underline denotes that we pass to the summand of $R^{n_w} \pi_! f^* M$ such that any irreducible occuring in the representation is associated to the two-sided cell $\textbf{c}$.

\subsubsection{Truncated convolution}

The category $\mathcal{P}^{\textbf{c}}(\btilde^2)$ is not only an abelian category but also monoidal under truncated convolution.  Following Lusztig \cite[Section~6.19]{lusztig2016non}, we define truncated convolution on $\mathcal{P}^{\textbf{c}}(\btilde^2)$ as
\[M \bullet M' := \underline{R^{2 a(\textbf{c})+\dim T} (p_{13})_! (p_{12}^* M \otimes p_{23}^* M')},\]
where the underline is the Serre quotient functor sending $\mathcal{P}^{\le \textbf{c}}(\btilde^2) \to \mathcal{P}^{\textbf{c}}(\btilde^2)$.  Truncated convolution defines a functor
\[ \bullet \colon \mathcal{P}^{\textbf{c}}(\btilde^2) \times \mathcal{P}^{\textbf{c}}(\btilde^2) \to \mathcal{P}^{\textbf{c}}(\btilde^2). \]
It turns out this is essentially taking the ``top cohomology'' of convolution.  We note that if we consider our sheaves as living on the quotient space $\btilde^2 /_n (T \times T)$, the correct dimension is to take $R^{2 a(\textbf{c})} (p_{13})_!$ in the above definition.

Inside a Weyl group $W$ there are special elements $d$ called Duflo involutions.  If $W_{\lambda}$ is the stabilizer of $\lambda$, let the set of Duflo involutions in $W_{\lambda}$ be $\textbf{D}_{\lambda}$, and let $\textbf{D} \subset I$ be the set of all $w \cdot \lambda$ such that $w \in \textbf{D}_{\lambda}$ as in \cite[Section~1.11~(vi)]{lusztig2015non}.  Let $\textbf{D}_{\textbf{c}} = \textbf{c} \cap \textbf{D}$ be the Duflo involutions inside $\textbf{c}$.  The corresponding sheaf $\bigoplus \mathbb{L}_{\lambda}^{\dot{d}}$ with the sum running over all Duflo involutions $d \cdot \lambda \in \textbf{D}_{\textbf{c}}$ serves as a monoidal unit in the truncated convolution category $\mathcal{P}^{\textbf{c}}(\btilde^2)$ \cite[Section~1.5]{lusztig1989cells}.  Duflo involutions further partition two-sided cells into left cells where $w \cdot \lambda$ is in the left cell defined by $d \cdot \lambda'$ if and only if $\mathbb{L}_{\lambda'}^{\dot{d}} \bullet \mathbb{L}_{\lambda}^{\dot{w}} = \mathbb{L}_{\lambda}^{\dot{w}}$.

Truncated convolution restricts to a functor
\[ \bullet \colon {}_{L_{\lambda}} \mathcal{P}_{L_{\lambda}}(U \backslash G / U)_{\textbf{c}} \times {}_{L_{\lambda}} \mathcal{P}_{L_{\lambda}}(U \backslash G / U)_{\textbf{c}} \to {}_{L_{\lambda}} \mathcal{P}_{L_{\lambda}}(U \backslash G / U)_{\textbf{c}} \]
Similarly have a monoidal unit $\bigoplus \mathbb{L}_{\lambda}^{\dot{d}}[|d|]$ running over Duflo involutions $d \cdot \lambda \in \textbf{D}_{\textbf{c}} \cap \textbf{D}_{\lambda}$.  We denote the latter set by $\textbf{D}_{\textbf{c},\lambda}$.

\subsubsection{Blocks and minimal IC sheaves}

Let $\mathfrak{o}$ be a $W$-orbit in $\mathfrak{s}_{\infty}$.  We have already indicated that our notion of two-sided cells on $W \times \mathfrak{o}$ arises from two-sided cells on two-sided cells on $W_{\lambda}$ for $\lambda \in \mathfrak{o}$.  To elaborate more on this point, we will explain that the category $\mathcal{P}^{\textbf{c}}(\btilde^2)$ can be built from the categories ${}_{L_{\lambda}} D_{L_{\lambda}}(U \backslash G / U)_{\textbf{c}}$ together with certain minimal IC sheaves.

Let ${}_{\lambda'} W_{\lambda} = \{ w \in W | w(\lambda) = \lambda' \}$.  If $\lambda,\lambda' \in \mathfrak{o}$, then ${}_{\lambda'} W_{\lambda}$ is nonempty, and in this case it is a $(W_{\lambda'}, W_{\lambda})$-bitorsor and contains a unique minimal element, which we denote by $w(\lambda',\lambda)$.  Moreover, this minimal element satisfies $w(\lambda'',\lambda') w(\lambda',\lambda) = w(\lambda'',\lambda)$.  Moreover, minimal IC sheaves induce equivalences of categories under convolution \cite[Proposition~5.2]{lusztig2020endoscopy}.
\begin{proposition}[Special case of Lusztig-Yun]\label{prop: minimal equivalence}
    Let $w = w(\lambda',\lambda)$.  Then
    \begin{equation}
        M \mapsto M * \mathbb{L}^{\dot{w}}_{\lambda} \colon {}_{L_{\lambda''}} D_{L_{\lambda'}}(U \backslash G / U) \to {}_{L_{\lambda''}} D_{L_{\lambda}}(U \backslash G / U)
    \end{equation}
    is an equivalence of categories with inverse $M \mapsto M * \mathbb{L}^{\dot{w}^{-1}}_{\lambda'}$.
\end{proposition}
By the uniqueness of minimal elements any element $w \in W$ can be written as $w' m$ where $w' \in W_{w(\lambda)}$ and $m$ is the minimal element $w(w(\lambda),\lambda)$.  For such a decomposition,
\begin{equation}
    \mathbb{L}^{\dot{w}}_{\lambda} \cong \mathbb{L}^{\dot{w'}}_{w(\lambda)} * \mathbb{L}^{\dot{m}}_{\lambda}.
\end{equation}

\subsection{The affine monodromic Hecke category}

\subsubsection{Definition of the affine monodromic Hecke category}

As before, we let $G$ be a split group, let $B \subset G$ be a Borel subgroup, and let $U \subset B$ be its unipotent radical.  Let $C$ be a fixed curve and $y \in C(\mathbb{F}_q)$ a closed $\mathbb{F}_q$-point.  Let $L G$ denote the loop group of $G$, $L^+ G$ the positive loop group, $L^{\circ} G$ the unipotent radical of the positive loop group, and $I$ the Iwahori subgroup such that $I \subset L^+ G$ gives $B \subset G$ upon passing to the quotient by $L^{\circ} G$.  Let $I^{\circ}$ denote the unipotent radical of $I$.  Recall the ind-schemes
\[
\begin{aligned}
    \Gr_G &= LG / L^+ G, \\
    \Fl_G &= LG / I, \\
    \fltilde_G &= LG / I^{\circ}.
\end{aligned}
\]
Allowing $L^+ G$ to act on $\mathcal{B}$ and $\btilde$ through the quotient $G$, we have the following identifications
\[
\begin{aligned}
    LG \backslash (\Fl_G \times \Fl_G) &\cong [I \backslash LG / I] \cong [L^+ G \backslash (\Fl_G \times \mathcal{B})], \\
    LG \backslash (\fltilde_G \times \fltilde_G) &\cong [I^{\circ} \backslash LG / I^{\circ}] \cong [L^+ G \backslash (\fltilde_G \times \btilde)],
\end{aligned}
\]
where the first isomorphism in each row is ``set-theoretic'' but the second can be understood as an isomorphism of ind-schemes acted on by an infinite-dimensional group such that the quotient is a nice quotient.  We can make sense of equivariant derived categories for these and we will implicitly use the equivalence between these two equivariant derived categories coming from the isomorphisms.  The first row gives the affine Hecke category, while the second will give us an affine monodromic Hecke category once we define a $T \times T$-action.

We may define the action map
\begin{equation}
    \begin{aligned}
        a \colon T \times [I^{\circ} \backslash LG / I^{\circ}] \times T &\to [I^{\circ} \backslash LG / I^{\circ}] \\
        a(t',g,t) &= t'gt^{-1},
    \end{aligned}
\end{equation}
analogous to and compatible with the corresponding action map on $U \backslash G / U$.  We also have corresponding $n$-th power actions
\begin{equation}
    \begin{aligned}
        a^n &\colon T \times [I^{\circ} \backslash LG / I^{\circ}] \times T \to [I^{\circ} \backslash LG / I^{\circ}] \\
        a^n(t',g,t) &= (t')^ngt^{-n},
    \end{aligned}
\end{equation}
and define an equivariant derived category
\[ D_{LG}(\fltilde^2 /_n (T \times T)) = D([I^{\circ} \backslash LG / I^{\circ}] /_n (T \times T)).\]
Taking the limit over the $n$-th power actions gives a category of monodromic sheaves that we will denote $D_{LG}(\fltilde^2)$.  The notation here $D_{LG}(\fltilde^2)$ is only by analogy with the category $D_G(\btilde^2)$, and we do \emph{not} claim that we are making sense of the action $LG$ here.  We will call the category $D_{LG}(\fltilde^2)$ the affine monodromic Hecke category.

\subsubsection{Blocks and minimal IC sheaves}

We now discuss the block structure of the affine monodromic Hecke category, which is simplified by the fact that we have assumed that $G$ has connected center.  For two character sheaves $\mathcal{L}$ and $\mathcal{L}'$, the category ${}_{\mathcal{L}'} D_{\mathcal{L}}(I^{\circ} \backslash LG / I^{\circ})$ is the derived category of constructible sheaves on $I^{\circ} \backslash LG / I^{\circ}$ that is equivariant with respect to $(T \times T, \mathcal{L}' \boxtimes \mathcal{L}^{-1})$.  Such a category was studied in \cite[Section~3]{li2020endoscopy} where the mixed version of this category is denoted there by ${}_{\mathcal{L}'} D_{\mathcal{L}}$, giving affine generalizations of the results from \cite{lusztig2020endoscopy}.  In particular, there is a convolution product \cite[Section~4]{li2020endoscopy}
\[ * \colon {}_{\mathcal{L}''} D_{\mathcal{L}}'(I^{\circ} \backslash LG / I^{\circ}) \times {}_{\mathcal{L}'} D_{\mathcal{L}}(I^{\circ} \backslash LG / I^{\circ}) \to {}_{\mathcal{L}''} D_{\mathcal{L}}(I^{\circ} \backslash LG / I^{\circ}), \]
analogous to the finite convolution product and restricting to it along the embedding coming from $\btilde \to \fltilde$.

Let $\wtilde$ denote the Weyl group of $LG$.  We choose a cross section $w \mapsto \dot{w} \colon \wtilde \to NT$ extending the cross section for the finite group, where $NT$ is now the normalizer of the torus inside the loop group.  The action of $W$ on the set $\mathfrak{s}_n$ and character sheaves extends to an action of $\tilde{W}$ on this set by making $X_*(T)$ act trivially.  There are IC sheaves
\[ \mathbb{L}^{\dot{w}}_{\lambda} \in {}_{L_{w(\lambda)}} D_{L_{\lambda}}(I^{\circ} \backslash LG / I^{\circ}). \]
These sheaves called ``positive IC sheaves'' in \cite[Section~3.3.1]{li2020endoscopy} and give a representatives of simple objects in the affine monodromic Hecke category.

Let $\wtilde_{\lambda} \subset \wtilde$ be the subset of this Weyl group that fixes $\lambda \in \mathfrak{s}_\infty$.  For a two-sided cell $\textbf{c}$ in $I = \wtilde \times \mathfrak{s}_{\infty}$, the intersection $\wtilde_{\lambda} \times \{ \lambda \} \cap \textbf{c}$ defines a corresponding two-sided cell decomposition of $\wtilde_{\lambda}$.

Under the inclusion $\btilde \to \fltilde$, the minimal IC sheaves on the finite monodromic Hecke category are naturally minimal IC sheaves on the affine monodromic Hecke category under this embedding.  We will need a property of minimal IC sheaves, which is implicit in \cite[Section~5]{li2020endoscopy}.
\begin{proposition}
    Let $w = w(\lambda',\lambda)$.  Then
    \begin{equation}
        M \mapsto M * \mathbb{L}^{\dot{w}}_{\lambda} \colon {}_{L_{\lambda''}} D_{L_{\lambda'}}(I^{\circ} \backslash LG / I^{\circ}) \to {}_{L_{\lambda''}} D_{L_{\lambda}}(I^{\circ} \backslash LG / I^{\circ})
    \end{equation}
    is an equivalence of categories with inverse $M \mapsto M * \mathbb{L}^{\dot{w}^{-1}}_{\lambda'}$.
\end{proposition}

\begin{proof}
    The proof of \cite[Proposition~5.2]{lusztig2020endoscopy} is by induction on the length of $w$, noting the base case is trivial, by writing $w$ as a product of $w' s$ for $s$ a simple reflection.  The proof shows that $w' = w(\lambda',s(\lambda))$ and by the inductive hypothesis $(-) * \mathbb{L}^{\dot{w'}}_{s(\lambda)}$ is an equivalence of categories, so by associativity it suffices to show that
    \[ (-) * \mathbb{L}^{\dot{s}}_{\lambda} \colon {}_{L_{\lambda''}} D_{L_{s(\lambda)}}(I^{\circ} \backslash LG / I^{\circ}) \to {}_{L_{\lambda''}} D_{L_{\lambda}}(I^{\circ} \backslash LG / I^{\circ}) \]
    is an equivalence of categories with inverse $(-) * \mathbb{L}^{\dot{s}}_{s(\lambda)}$.  But this follows by \cite[Lemma~5.4]{li2020endoscopy} and \cite[Lemma~5.5]{li2020endoscopy}.
\end{proof}

\subsubsection{Central sheaves}

There are certain sheaves $Z_V^{\lambda}$ in the center of the category ${}_{L_{\lambda}} D_{L_{\lambda}}(I^{\circ} \backslash LG / I^{\circ})$ with respect to the monoidal convolution product.  In the ``classical'' case $\lambda = 1$, there are two constructions of central sheaves in the literature.  The first, due to Gaitsgory \cite{gaitsgory1999construction}, defines central sheaves $Z_V$ as nearby cycles sheaves along a variety, which we will denote $\Gr\mathcal{B}_G$.
\begin{definition}
Let $\Gr\mathcal{B}_G$, resp. $\Gr\btilde_G$, be the variety classifying as its $S$-points:
\begin{enumerate}
    \item $x \in C(S)$,
    \item $\mathcal{E}$ an $S$-point of $\Bun_G$,
    \item $\phi$ a Hecke modification
    \[ \phi \colon \mathcal{E} |_{C \times S \setminus \Gamma_{x}} \simeq \mathcal{G} |_{C \times S \setminus \Gamma_{x_i}}, \]
    \item and a reduction of structure group at $y$ to $B \subset G$, resp. $U \subset G$.
\end{enumerate}
\end{definition}

The general fiber of $\Gr\mathcal{B}_G$ over $C \setminus \{ y \}$ is $\Gr_G \times G / B$, while the special fiber over $y$ is $\Fl_G = LG / I$.  Gaitsgory defines central sheaves as $\Psi(\mathcal{S}_V \boxtimes \delta_e)[1]$ where $\mathcal{S}_V$ is the Satake sheaf shifted to be perverse and $\delta_e$ is a skyscraper sheaf on $eB$.  The shift appears in the definition to ensure that the resulting sheaf is perverse.  In what follows, we will drop the shift from nearby cycles to preserve perversity with respect to a map to the base, in general a power of a curve, rather than the perversity with respect to the map to $\mathbb{F}_q$ or $\fqbar$.

\begin{definition}
The central sheaf $Z_V$ is defined as $\Psi(\mathcal{S}_V \boxtimes \delta_e)$, taking nearby cycles on $\Gr\mathcal{B}_G$.

Let $\lambda \in \mathfrak{s}_{\infty}$.  Then the central sheaf $Z_V^{\lambda}$ is defined as $\Psi(\mathcal{S}_V \boxtimes L_{\lambda})$, taking nearby cycles on on $\Gr\btilde_G$.
\end{definition}

A second construction appears in the work of Zhu \cite{pappas2013local} \cite{zhu2011geometric}, see also Richarz \cite{richarz2016affine}.  For this construction, we let $G$ be a parahoric group scheme over $C$ which is split over $C \setminus \{ y \}$ and has Iwahori reduction at $y$.  Then the Beilinson-Drinfeld grassmannian has general fiber $\Gr_G$ and special fiber $\Fl_G$.  Indeed, we have a closed embedding of this Beilinson-Drinfeld grassmannian into $\Gr\mathcal{B}_G$, and the two definitions can be seen to be equivalent by noting that pushforward along a proper map commutes with nearby cycles.

\subsubsection{Compatibility between finite convolution and nearby cycles}

In this section, we will be using the first degeneration, $\Gr\mathcal{B}_G$, and various generalizations.  The importance of this degeneration for us is that $\Gr\mathcal{B}_G$ interpolates between convolution on the finite flag variety and convolution on the affine flag variety.  This property was first noted by Gaitsgory \cite[Proposition~6]{gaitsgory1999construction}, and we will recall the construction of a convolution diagram used there, which we also adapt to the monodromic setting.

\begin{definition}
Let $\Gr G_G$ be the variety classifying as its $S$-points:
\begin{enumerate}
    \item $x \in C(S)$,
    \item $\mathcal{E}$ an $S$-point of $\Bun_G$,
    \item $\phi$ a Hecke modification
    \[ \phi \colon \mathcal{E} |_{C \times S \setminus \Gamma_{x}} \simeq \mathcal{G} |_{C \times S \setminus \Gamma_{x}}, \]
    \item and a trivialization $\varepsilon \colon \mathcal{E}_y \simeq \mathcal{G}_y$ where $\mathcal{G}$ is the trivial $G$-bundle.
\end{enumerate}
\end{definition}
We now define the relevant ``interpolating'' convolution diagram.
\begin{definition}
The maps $f$ and $m$ in the diagram
\begin{equation}\label{eqn:convolution diagram}
\begin{tikzcd}
    \Gr\mathcal{B}_G \times [B \backslash G / B] & \Gr G_G \times^B G / B \arrow{l}{f} \arrow{r}{m} & \Gr\mathcal{B}_G.
\end{tikzcd}
\end{equation}
are defined as follows.  The map $f$ sends a tuple $(x, \mathcal{E}, \phi, \varepsilon, gB)$ to $(x, \mathcal{E}, \phi, \varepsilon_B, BgB)$, where $\varepsilon_B$ is the reduction of structure group $\varepsilon_B \colon \mathcal{E}_y \simeq G \times^B \mathcal{F}_y$ where $\mathcal{F}_y$ is the (trivial) $B$-bundle.  The map $m$ sends a tuple $(x, \mathcal{E}, \phi, \varepsilon, gB)$ to $(x, \mathcal{E}, \phi, \varepsilon')$ where $\varepsilon'$ is the composition $\mathcal{E} \simeq G \simeq G$ where the first equivalence is $\varepsilon$ and the second equivalence is the reduction of structure group of the trivial $G$-bundle to $B$ represented by $gB \in G / B$.

A priori these maps are maps with source $\Gr G_G \times G / B$, but we may observe that they factor through the quotient by $B$ which acts on $G / B$ by the left action and on $\Gr G_G$ by changing the trivialization of the trivial $G$-bundle.

Similarly, define the maps $\tilde{f}$ and $\tilde{m}$ as
\begin{equation}
\begin{aligned}
\tilde{f} = (x, \mathcal{E}, \phi, \varepsilon, gU) \mapsto (x, \mathcal{E}, \phi, \varepsilon_U, UgU) &\colon \Gr G_G \times^U \btilde \to \Gr\btilde_G \times [U \backslash G / U] \\
\tilde{m} = (x, \mathcal{E}, \phi, \varepsilon, gU) \mapsto (x, \mathcal{E}, \phi, \varepsilon') &\colon \Gr G_G \times^U \btilde \to \Gr\btilde_G
\end{aligned}
\end{equation}
analogous to the definitions of $f$ and $m$ above.  These maps are $T \times T$-equivariant under the action of $T \times T$ through the $n$th power, and we define $f_n$ and $m_n$ as quotients by this action, so that $m_1 = m$ and $f_1$ is equal to $f$ up to an additional quotient on the target
\[ \Gr\btilde_G \times^T [U \backslash G / B] \to \Gr\mathcal{B}_G \times [B \backslash G / B]. \]
\end{definition}

\begin{proposition}
$f_n$ is smooth and $m_n$ is proper.
\end{proposition}

\begin{proof}
The smoothness of $f_n$ follows from that of $\tilde{f}$, noting that $\tilde{f}$ is in fact a principal $U$-bundle.  The properness of $m_n$ follows by noting the quotients of the source and target are ind-proper Deligne-Mumford stacks.
\end{proof}

Using the above proposition, we can now interpret the convolution diagram~\eqref{eqn:convolution diagram} as interpolating between the finite and affine flag varieties.  Over $C \setminus \{ y \}$, the diagram gives the convolution diagram on $G / B$ with an additional constant factor of the Beilinson-Drinfeld grassmannian $\Gr_G$.  Over the special fiber, the diagram gives
\begin{equation}
    \begin{tikzcd}
    \tilde{\Fl}_G \times [U \backslash G / U] & LG \times^{I^{\circ}} \btilde \arrow{l} \arrow{r} & \tilde{\Fl_G},
    \end{tikzcd}
\end{equation}
which identifies with a closed subset of the convolution diagram for the $T$-cover of the affine flag variety under the identification $G / B \cong L^+ G / I \subset LG / I$.

The following proposition generalizes \cite[Proposition~6]{gaitsgory1999construction}.

\begin{proposition}\label{prop:nearby central convolution}
Let $M$ be in $D(B \backslash G / B)$.  We may consider $M$ as an $I$-equivariant sheaf on the flag variety under the inclusion $G / B \to LG / I$.  Then,
\begin{equation}
    \Psi(\mathcal{S}_V \boxtimes M) \cong Z_V * M,
\end{equation}
where $*$ denotes convolution in the affine flag variety.

More generally, consider ${}_{\mathcal{L}'} M_{\mathcal{L}}$, a sheaf in ${}_{\mathcal{L}'} D_{\mathcal{L}}(U \backslash G / U)$.  Then,
\begin{equation}
    \Psi(\mathcal{S}_V \boxtimes {}_{\mathcal{L}'} M_{\mathcal{L}}) \cong Z_V^{\mathcal{L}'} * {}_{\mathcal{L}'} M_{\mathcal{L}}.
\end{equation}
\end{proposition}

\begin{proof}
We prove the monodromic statement since the first statement follows from the case $\mathcal{L} = \qlbar$.  Consider the sheaf $\mathcal{S}_V \boxtimes \mathcal{L}'$ on the generic fiber of $\Gr\btilde_G /_n T$.  The properness of $m$ and smoothness of $f$ implies that
\[ m_! f^* \Psi(\mathcal{S}_V \boxtimes \mathcal{L}' \boxtimes {}_{\mathcal{L}'} M_{\mathcal{L}}) \cong \Psi(\mathcal{S}_V \boxtimes m_! f^*(\mathcal{L}' \boxtimes {}_{\mathcal{L}'} M_{\mathcal{L}})), \]
where $m$ and $f$ are finite convolution.  But $\mathcal{L}'$ acts by the identity on the convolution category ${}_{\mathcal{L}'} D_{\mathcal{L}}(U \backslash G / U)$ by Proposition~\ref{prop: minimal equivalence}, so the right hand side is equivalent to $\Psi(\mathcal{S}_V \boxtimes {}_{\mathcal{L}'} M_{\mathcal{L}})$.  Moreover, the left hand side is equivalent to $Z^{\mathcal{L}'}_V * {}_{\mathcal{L}'} M_{\mathcal{L}}$ using the definition of $Z^{\mathcal{L}'}_V$.
\end{proof}

\subsubsection{Affine monodromic truncated convolution categories}

In addition to affine monodromic Hecke categories, there are also affine monodromic truncated convolution categories, studied in \cite{bezrukavnikov2009tensor}.  We will fix $\lambda \in \mathfrak{s}_\infty$ and $\mathfrak{o} = W \lambda$, the $W$-orbit of $\lambda$.
\begin{definition}
Let $\mathcal{H}^{\lambda} = {}_{L_{\lambda}} D_{L_{\lambda}}(I^{\circ} \backslash LG / I^{\circ})$, and define the category
\[ \mathcal{H}^{\mathfrak{o}} = \bigoplus_{\lambda',\lambda'' \in \mathfrak{o}} {}_{L_{\lambda''}} D_{L_{\lambda'}}(I^{\circ} \backslash LG / I^{\circ}) \]
as a full subcategory of $D_{LG}(\fltilde^2)$.
\end{definition}

Simple objects in $\mathcal{H}^{\lambda}$ are indexed by $\wtilde_{\lambda}$ and simple objects in $\mathcal{H}^{\mathfrak{o}}$ are indexed by the discrete set $\wtilde \times \mathfrak{o}$.  There is a two-sided cell theory on $\wtilde_{\lambda}$ which we extend to $\wtilde \times \mathfrak{o}$ by rules analogous to \cite[Section~11.4]{lusztig2020endoscopy}.  Given a two-sided cell $\textbf{c} \subset \wtilde \times \mathfrak{o}$, we may consider the corresponding two-sided cell $\textbf{c} \cap \wtilde_{\lambda} \times \{ \lambda \}$, which by abuse of notation, we also denote $\textbf{c}$.

\begin{definition}
The full subcategories $\mathcal{H}^{\lambda}_{\le \textbf{c}}$, resp. $\mathcal{H}^{\lambda}_{< \textbf{c}}$, are the full subcategories of the $\mathcal{H}^{\lambda}$ generated by $\mathbb{L}_{\lambda}^{\dot{w}}$ (where we apply appropriate shifts so these objects are perverse) for $w \in \bigcup_{\textbf{c}' \le \textbf{c}} \textbf{c}'$, resp. $w \in \bigcup_{\textbf{c}' < \textbf{c}} \textbf{c}'$.  Similarly, we define $\mathcal{H}^{\mathfrak{o}}_{\le \textbf{c}}$ and $\mathcal{H}^{\mathfrak{o}}_{< \textbf{c}}$.

The truncated convolution category $\mathcal{A}^{\lambda}_{\textbf{c}}$, resp. $\mathcal{A}^{\mathfrak{o}}_{\textbf{c}}$, is subcategory of the Serre quotient $\mathcal{H}^{\lambda}_{\le \textbf{c}} / \mathcal{H}^{\lambda}_{< \textbf{c}}$, resp. $\mathcal{H}^{\mathfrak{o}}_{\le \textbf{c}} / \mathcal{H}^{\mathfrak{o}}_{< \textbf{c}}$, generated by objects $Z_V^{w(\lambda)} * \mathbb{L}_{\lambda}^{\dot{w}}$ for $w \in \textbf{c}$, understood as a two-sided cell in $\wtilde_{\lambda}$, resp. $\wtilde \times \mathfrak{o}$.  The tensor category structure on $\mathcal{A}^{\lambda}_{\textbf{c}}$ is given by truncated convolution $M \bullet M' = {}^p H^{a(\textbf{c})}(M * M')$.

Let $d$ be a Duflo involution.  The Tannakian category $\mathcal{A}^{\lambda}_d$ is the subcategory of $\mathcal{A}^{\lambda}_{\textbf{c}}$ generated by objects that are subquotients of $Z_V^{\lambda} * \mathbb{L}_{\lambda}^{\dot{d}}$.
\end{definition}

The Tannakian structure on can be seen by the fact that $\mathbb{L}_{\lambda}^{\dot{d}}$ is a unit object under truncated convolution and the map $\Res_d := V \mapsto Z_V^{\lambda} * \mathbb{L}_{\lambda}^{\dot{d}}$ gives a fiber functor $\Rep(\widehat{G}) \to \mathcal{A}^{\lambda}_d$.  The category $\mathcal{A}_d$ is also equipped with a tensor automorphism $\mathfrak{M}_d$ coming from monodromy on $Z_V^{\lambda} * \mathbb{L}_{\lambda}^{\dot{d}}$.

Recall that fixing a generator of the tame inertia allows us to view $L_{\lambda}$ as an element of $\widehat{G}$.  Let $\lambda$ be the element of $\widehat{G}$ represented by $L_{\lambda}$, and let $G_{\lambda}$ denote the centralizer of $\lambda$.  Suppose we fix a two-sided cell $\textbf{c} \subset W_{\lambda}$.  It is contained a unique two-sided cell $\textbf{c}'$ in the affine Weyl group analogue $\wtilde_{\lambda}$.  Moreover, because of our connected center assumption $\textbf{c}'$ corresponds to a special unipotent conjugacy class of $\widehat{G}_{\lambda}$ and $\mathbf{c}' \cap W = \mathbf{c}$ \cite[Section~4.8]{lusztig1989cells}.

By \cite[Corollary~1.9(d)]{lusztig1987cells}, the $a$-invariant of $\mathbf{c}$ computed in the finite Weyl group is the same as $a(\mathbf{c}')$ computed in the affine Weyl group. Therefore the inclusion $L^+ G \subset LG$ gives a monoidal functor for the truncated convolution categories
\[Z^{\lambda}_V * (-) \colon {}_{L_{\lambda}} D_{L_{\lambda}}(U \backslash G / U)_{\textbf{c}} \to \mathcal{A}^{\lambda}_{\textbf{c}'}.\]

\begin{proposition}
Let $w \in \textbf{c} \cap W_{\lambda}$.  In the truncated convolution category $\mathcal{A}^{\lambda}_{\textbf{c}}$,
\begin{equation}
    \Psi(\mathcal{S}_V \boxtimes \mathbb{L}_{\lambda}^{\dot{w}}) \cong \bigoplus_d (Z^{\lambda}_V * \mathbb{L}^{\dot{d}}_{\lambda}) \bullet \mathbb{L}^{\dot{w}}_{\lambda}
\end{equation}
where the sum runs over Duflo involutions $d \in \textbf{D}_{\textbf{c},\lambda}$.
\end{proposition}

\begin{proof}
    Since $\bigoplus_{d\in \textbf{D}_{\textbf{c},\lambda}} \mathbb{L}_{\lambda}^{\dot{d}} \in {}_{L_{\lambda}} D_{L_{\lambda}}(U\backslash G/U)_{\textbf{c}}$ is a monoidal unit of truncated convolution, we may write
    \[ \mathbb{L}_{\lambda}^{\dot{w}} \cong \bigoplus_d \mathbb{L}_{\lambda}^{\dot{d}} \bullet \mathbb{L}_{\lambda}^{\dot{w}}, \]
    where the result now follows from Proposition~\ref{prop:nearby central convolution} together with the associativity of convolution and the agreement of the $a$-function.
\end{proof}

The following is a straightforward application of \cite[Proposition~1 and Lemma~5]{bezrukavnikov2004tensor} to the affine monodromic truncated convolution category given by a Duflo involution $d$.

\begin{proposition}
There exists a pair $\widehat{H}_d$, $N_d$ with $\widehat{H}_d \subset \widehat{G}_{\lambda} \subset \widehat{G}_{\qlbar}$, and $N_d$ a unipotent element commuting with $\widehat{H}_d$; an equivalence of tensor categories $\Phi_d \colon \Rep(\widehat{H}_d) \simeq \mathcal{A}^{\lambda}_d$, and an isomorphism $\Res_{\widehat{H}_d}^{\widehat{G}} \cong \Phi_d \circ \Res_d$, which intertwines the tensor automorphisms $\mathrm{Aut}_{N_d}$ and $\mathfrak{M}_d$.  The pair $(\widehat{H}_d, N_d)$ is unique up to conjugacy.
\end{proposition}

\subsection{The horocycle correspondence for restricted shtukas}

Throughout this section, we will fix an identification of group schemes $G \simeq {}^{\tau} G$ over $y$ in order to make sense of the compatibility with Frobenius structure.


Let $G$ be a reductive group over $C$ with parahoric reduction.  Let $y$ be a fixed $\mathbb{F}_q$-point with $y = y_1 \in N$ and let $N = \sum a_j y_j$ with $a_1 = 0$.  Fix the standard maximal parahoric group $G_y = L^+ G$, let $G_y^{\circ} \subset I_y \subset G_y$ be a fixed Iwahori with unipotent radical $I^{\circ}_y$.  Define a torus $T = I_y / I^{\circ}_y$.  Since $G_y / G_y^{\circ} \cong G$, this torus can also be identified with $B / U$, where $B = I_y / G^{\circ}_y$ is a Borel subgroup and $U = I^{\circ}_y / G^{\circ}_y$ is its unipotent radical.

\begin{definition}
Let $\Cht\btilde_{G,N,I,y}$, resp. $\Cht\mathcal{B}_{G,N,I,y}$ classify as its $S$-points:
\begin{enumerate}
    \item $(x_i)_{i \in I}\in C^I(S)$,
    \item $\mathcal{E}$ and $\mathcal{F}$ two $G_{I,\infty,\sum \infty y_j}$-torsors on $S$,
    \item $\phi$ a Hecke modification
    \[ \phi \colon \mathcal{E} |_{\Gamma_{\sum_i \infty x_i} \setminus \Gamma_{\sum_i x_i}} \simeq \mathcal{F} |_{\Gamma_{\sum_i \infty x_i} \setminus \Gamma_{\sum_i x_i}}, \]
    \item an isomorphism of associated $G$-bundle fibers at $y$, $\varepsilon \colon \mathcal{F}_G \simeq {}^{\tau} \mathcal{E}_G$,
    \item and two reductions of structure group of $\mathcal{E} |_{\Gamma_{\infty y_1}}$ and $\mathcal{F} |_{\Gamma_{\infty y_1}}$ to $I^{\circ}_y$, resp. $I_y$, such that these reductions of structure group are compatible with the Frobenius structure.
\end{enumerate}

Let us be more explicit about what we mean by compatibility with Frobenius structure.  Restricting from $N$ to the single point $y$ (that, is, taking the associated bundle along the quotient map $H \to G = G_y / G_y^{\circ}$) gives an isomorphism of associated $G$-bundles
\[ \varepsilon_y \colon \mathcal{F}_G \simeq {}^{\tau} \mathcal{E}_G. \]
At the same time, we have identifications of these bundles on both sides as
\begin{align*}
\varepsilon_{\mathcal{F}} &\colon \mathcal{F}_G \simeq \mathcal{F}_U \times_U G \\
\varepsilon_{\mathcal{E}} &\colon \mathcal{E}_G \simeq \mathcal{E}_U \times_U G,
\end{align*}
where $U \subset G$ is the quotient $I^{\circ}_y / G_y^{\circ}$.  The compatibility condition is then that the isomorphism $\varepsilon_y$ must come from a Frobenius structure on $U$-bundles as
\[ \varepsilon_U \colon \mathcal{F}_U \simeq {}^{\tau} \mathcal{E}_U. \]

Let $\Gr\btilde_{G,N,I,y}$, resp. $\Gr\mathcal{B}_{G,N,I,y}$, classify as its $S$-points:
\begin{enumerate}
    \item $(x_i)_{i \in I}\in C^I(S)$,
    \item $\mathcal{E}$ and $\mathcal{F}$ two $G_{I,\infty,\sum \infty y_j}$-torsors on $S$,
    \item $\phi$ a Hecke modification
    \[ \phi \colon \mathcal{E} |_{\Gamma_{\sum_i \infty x_i} \setminus \Gamma_{\sum_i x_i}} \simeq \mathcal{F} |_{\Gamma_{\sum_i \infty x_i} \setminus \Gamma_{\sum_i x_i}}, \]
    \item and two reductions of structure group of $\mathcal{E} |_{\Gamma_{\infty y_1}}$ and $\mathcal{F} |_{\Gamma_{\infty y_1}}$ to $I^{\circ}_y$, resp. $I_y$.
\end{enumerate}
\end{definition}

The versions with $\btilde$ are $T \times T$-torsors over the versions with $\mathcal{B}$.  That is, we have an action of $T \times T$ on $\Cht\btilde_{G,N,I,y}$ and $\Gr\btilde_{G,N,I,y}$ by changing the reduction of structure group of $\mathcal{E}$ and $\mathcal{F}$ from $I_y$ to $I^{\circ}_y$ such that the quotients are $\Cht\mathcal{B}_{G,N,I,y}$ and $\Gr\mathcal{B}_{G,N,I,y}$.

We have forgetful maps
\[
\begin{aligned}
    \tilde{\pi} &\colon \Cht\btilde_{G,N,I,y} \to \Cht\mathcal{R}_{G,N,I}, \\
    \tilde{f} &\colon \Cht\btilde_{G,N,I,y} \to \Gr\btilde_{G,N,I,y}, \\
    \pi &\colon \Cht\mathcal{B}_{G,N,I,y} \to \Cht\mathcal{R}_{G,N,I},
\end{aligned}
\]
such that $\tilde{f}$ is $T \times T$-equivariant and such that $\tilde{\pi}$ is the composition of the quotient of the $T \times T$-torsor and $\pi$.

Let $T$ act on itself by the $n$th power, that is, $t_1 \cdot t_2 = t_1^n t_2$.  Then the quotient $T /_n T \cong BT_n$, the classifying space of the finite group of $n$-division points of $T$.  Consider $T \times T$ acting on $\Cht\btilde_{G,N,I,y}$ and $\Gr\btilde_{G,N,I,y}$ through the $n$th power on both copies of $T$.  For any $n > 0$, forgetful map $\tilde{\pi}$ can be written as the composition $\pi_n \circ q_n$ where $q_n$ is the quotient by the $n$th power action of $T \times T$ and $\pi_n$ is a forgetful map from a twisted quotient of $\Cht{\btilde}_{G,N,I,y} /_n (T \times T)$.  Along the same lines, we can consider a map $f_n \colon \Cht{\btilde}_{G,N,I,y} /_n (T \times T) \to \Gr\btilde_{G,N,I,y} /_n (T \times T)$ as well as perverse sheaves and objects in the derived category of constructible sheaves on these spaces.

\begin{proposition}
    The following hold:
    \begin{enumerate}
    \item Over $(C \setminus N)^I$, we have the following identifications.
    \begin{equation}
        \begin{aligned}
            \Cht\mathcal{R}_{G,N,I} &\simeq [G_{I,\infty} \backslash \Gr_{G,I}] \times H /_{\sigma} G_{\sum \infty y_j}, \\
            \Cht\btilde_{G,N,I,y} &\simeq [G_{I,\infty} \backslash \Gr_{G,I}] \times X / G_{\sum \infty y_j}, \\
            \Gr\btilde_{G,N,I,y} &\simeq [G_{I,\infty} \backslash \Gr_{G,I}] \times (\btilde \times \btilde) / G_{\sum \infty y_j}.
        \end{aligned}
    \end{equation}
    \item The map $\tilde{f} = (1 \times \tilde{g}) \circ (1 \times f)$ where
    \begin{equation}
    f \colon X \to \btilde^2
    \end{equation}
    is the map in Equation~\eqref{eqn: finite horocycle correspondence} and $\tilde{g} \colon BG(\mathbb{F}_q) \to BH(\mathbb{F}_q)$ comes from writing $H$ as a product of the form $G \times H'$.
    \item The map $\tilde{\pi}$ pulls back the map $\pi \colon X \to G$ in Equation~\eqref{eqn: finite horocycle correspondence}.
    \end{enumerate}
\end{proposition}

\begin{proof}
\begin{enumerate}
    \item The identification of $\Cht\mathcal{R}$ has already been observed.  The other identifications, as we note that the data of $(1),(2),(3)$ classifies the stack $[G_{I,\infty,\sum_j \infty y_j} \backslash \Gr_{G,I}]$ which is almost independent of the data of $(4)$ and $(5)$, resp. $(4)$, which is local on $y$ and becomes independent of it if we give a trivialization of $\mathcal{E}_G|_{\Gamma_{\sum_j \infty y_j}}$.  Additionally, the reduction of structure group from $\Gamma_{\infty y}$ to $I_y^{\circ}$ is the same thing as a reduction of structure group from $G$ to $U$, naturally identified with $G / U = \btilde$.  A trivialization of $\mathcal{E}|_{\Gamma_{\sum_j \infty y_j}}$ gives a trivialization of $\mathcal{E}_G \simeq G$, under which we inherit two trivializations of $\mathcal{F}_G$, one under the Hecke modification and another under a trivialization of ${}^{\tau} \mathcal{E}_G \simeq \mathcal{F}_G$ under the identification $G \simeq {}^{\tau} G$.  Comparing these two trivializations gives an element of $G$.  The compatibility of the $U$ reductions of $\mathcal{E}_G$ and $\mathcal{F}_G$ translates to the moduli problem of $X$ in Equation~\eqref{eqn: X moduli problem}.
    \item The factorization results from noting that the data defining restricted shtukas decomposes as ``restricted shtukas for the level $N' = y$'' and ``restricted shtukas for the level $N'' = \sum_{j\ge 2} a_j y_j$''.  More specifically, away from $N$, $\Cht\mathcal{R}_{G,N,I}$ fits into a Cartesian diagram
    \[
    \begin{tikzcd}
   \Cht\mathcal{R}_{G,N,I} \arrow{r} \arrow{d} & \Cht\mathcal{R}_{G,N',I} \times \left[*/G_{\sum_{j\ge 2} \infty y_j}\right] \arrow{d} \\
    \Cht\mathcal{R}_{G,N'',I} \times \left[*/G_{\infty y_1}\right] \arrow{r} & \left[G_{I,\infty} \backslash \Gr_{G,I}\right] \times \left[*/G_{\sum_{j} \infty y_j}\right].
    \end{tikzcd}
    \]
    Moreover, both maps on the right and bottom have sections that come from writing $H$ as a product $G \times H'$.  Now the result follows by noting that forgetting the reductions to $\mathcal{E}_U$ and $\mathcal{F}_U$ maps to the moduli space $\Cht\mathcal{R}_{G,N',I} \times [*/G_{\sum_{j\ge 2} \infty y_j}]$ and the map to $\Cht\mathcal{R}_{G,N,I}$ is given by the section noted above.
    \item The maps $\pi \colon X \to \btilde^2$ in the top and bottom rows are projections to different factors, the left vertical arrow is the identity, and the right vertical arrow is the map $\epsilon$ in Equation~\eqref{eqn:gu epsilon}.
\end{enumerate}
\end{proof}

\begin{remark}
    The constructions above are similar to parabolic character sheaves constructed by Lusztig \cite{lusztig2010parabolic}.  The horocycle correspondence for parabolic character sheaves has also been considered by Li, Nadler, and Yun \cite{li2023functions}.
\end{remark}

The proposition above allows us to describe sheaves over the portion of $\Gr\btilde_{G,N,I,y}$ over $(C \setminus N)^I$ in terms of categories of perverse sheaves already in the literature.  The action of $G_{\sum \infty y_j}$ on $\btilde^2$ factors through $G$.  Thus, we can form sheaves over the unramified locus
\begin{equation}
    \mathcal{S}_{I,V} \boxtimes {}_{\mathcal{L}'} M_{\mathcal{L}} \in D(\Gr\btilde_{G,N,I,y})
\end{equation}
for a Satake sheaf $\mathcal{S}_{I,V}$ and a sheaf ${}_{\mathcal{L}'} M_{\mathcal{L}} \in {}_{\mathcal{L}'} D_{\mathcal{L}}(U \backslash G / U) \subset D_G(\btilde^2)$.  We will use the same notation to denote the $!$-extension to all of $\Gr\btilde_{G,N,I,y}$.  We wish to understand the objects on restricted shtukas related to these objects under the horocycle transform and nearby cycles in terms of what we have already discussed for the finite and affine monodromic Hecke categories and Lusztig's horocycle correspondence.

\begin{proposition}\label{prop: horocycle proper smooth}
$\pi_n$ is proper and $f_n$ is smooth.
\end{proposition}

\begin{proof}
    $\Cht \btilde \to \Cht\mathcal{R}$ is a composition of an inclusion that forgets the Frobenius compatibility condition together with a projection that forgets the $I_y^{\circ}$ structures on $\mathcal{F}$ and $\mathcal{E}$.  The first map is a closed embedding, while the second map is a $\btilde^2$-bundle.  After quotienting by the twisted action of $T \times T$, the latter map becomes proper, so the properness of $\pi_n$ is proved.

    Since the diagram
    \begin{equation}
        \begin{tikzcd}
        \Cht{\mathcal{B}}_{G,N,I,y} \arrow{r} \arrow{d}& \Gr\btilde_{G,N,I,y} \arrow{d} \\
        \Cht{\mathcal{B}}_{G,N,I,y} /_n (T \times T) \arrow{r}& \Gr\btilde_{G,N,I,y} /_n (T \times T)
        \end{tikzcd}
    \end{equation}
    is Cartesian, it suffices to show that $\tilde{f}$ is smooth.  The fiber of $\tilde{f}$ over an $S$-point $((x_i)_{i\in I}, \mathcal{E}, \mathcal{F}, \varepsilon_{\mathcal{E}}, \varepsilon_{\mathcal{F}})$ can be identified with choices of $\varepsilon_{U} \colon \mathcal{E}_U \simeq {}^{\tau} \mathcal{F}_U$, which is thus a choice of a point in $\sigma(U)$.
\end{proof}

\subsection{$\Psi$-factorizability and the horocycle correspondence}

The horocycle correspondence in the previous section will be the key tool that we will use to ``unwind'' a representation $\zeta$ of $G(\F_q)$ to a rank-1 character sheaf $L_{\lambda}$ on the torus.  The character sheaf property of pulling back under multiplication will give us $\Psi$-factorizability for $\mathcal{S}_{I,V} \boxtimes \zeta$.

Let $\mathfrak{p}_J$ be the map $\Gr\btilde_{G,N,I\cup J,y} \to C^J$.  We also define an iterated version $\Gr\btilde^{(I_1, \dots, I_n)}_{G,I\cup J,y}$.

\begin{definition}
    For a partition $(I_1, \dots, I_n)$ of $I$, let $\Gr\btilde^{(I_1, \dots, I_n)}_{G,I,y}$ classify as its $S$-points:
\begin{enumerate}
    \item $(x_i)_{i \in I}\in C^I(S)$,
    \item $\mathcal{E}_0, \dots, \mathcal{E}_n$ a collection of $n+1$ $G_{I,\infty,\infty y}$-torsors on $S$,
    \item $\phi_i$ Hecke modifications
    \[ \phi_i \colon \mathcal{E}_{i-1} |_{\Gamma_{\sum_i \infty x_i} \setminus \Gamma_{\sum_i x_i}} \simeq \mathcal{E}_i |_{\Gamma_{\sum_i \infty x_i} \setminus \Gamma_{\sum_i x_i}}, \]
    \item and reductions of structure group $\varepsilon_i$ of $\mathcal{E}_0 |_{\Gamma_{\infty y}}$ and $\mathcal{E}_{n}|_{\Gamma_{\infty y}}$ to $I^{\circ}_y$.
\end{enumerate}
\end{definition}

The map collapsing the legs $c \colon \Gr\btilde^{(I_1, \dots, I_n)}_{G,I\cup J,y} \to \Gr\btilde_{G,0\cdot y,I\cup J,y}$ is a proper map that pulls back the map $\Gr_{G,I}^{(I_1, \dots, I_n)} \to \Gr_{G,I}$ on Beilinson-Drinfeld grassmannians.

\begin{proposition}\label{prop:Grbtilde psi factorizable}
    For ${}_{\mathcal{L}'} M_{\mathcal{L}} \in {}_{\mathcal{L}'} D_{\mathcal{L}}(U \backslash G / U)$, the pairs $(\mathfrak{p}_J, j_!(\mathcal{S}_{I \cup J,V \boxtimes W} \boxtimes {}_{\mathcal{L}'} M_{\mathcal{L}}))$ are $\Psi$-factorizable.
\end{proposition}

\begin{proof}
    The proof is similar to the proof of Proposition~\ref{prop:psi-factorizable grassmannian}.  We will work at the infinite level, noting that we could work with maps at the finite level using maps $\kappa$ similar to those used for Beilinson-Drinfeld grassmannians.
    
    First, consider the ind-scheme $(\Gr\btilde_G)^{I\cup J}$.  Write our desired Satake sheaf as $\mathcal{S}_{I \cup J,V \boxtimes W}$ and assume without loss of generality that $V \boxtimes W \cong V_1 \boxtimes \dots \boxtimes V_{|I|+|J|}$.  Over $(C \setminus (\{ y \} \cup R))^{I\cup J}$, we can construct a Satake sheaves
    \begin{equation} (\mathcal{S}_{V_1} \boxtimes {}_{\mathcal{L}'} M_{\mathcal{L}}) \boxtimes \left(\mathop{\boxtimes}_{i=2}^{|I|+|J|} (\mathcal{S}_{V_i} \boxtimes \mathcal{L}) \right), \end{equation}
    where we consider $\mathcal{L}$ as the identity in ${}_{\mathcal{L}} D_{\mathcal{L}}(U \backslash G / U)$.  Let $j_i$ denote open immersion of the $i$-th factor $C \setminus (\{y\} \cup R) \to C$ and any pullback of this map when this is clear, $j$ the open immersion $j_1 \times \dots \times j_{|I|}$.  So the sheaf
    \begin{equation}\label{eqn: GrB sheaf} (j_1)_! (\mathcal{S}_{V_1} \boxtimes {}_{\mathcal{L}'} M_{\mathcal{L}}) \boxtimes \left(\mathop{\boxtimes}_{i=2}^{|I|+|J|} (j_i)_! (\mathcal{S}_{V_i} \boxtimes \mathcal{L}) \right), \end{equation}
    is an external tensor product, so Proposition~\ref{prop: kunneth factorizable} shows that this sheaf is $\Psi$-factorizable with respect to the map projecting to $C^I$.  Now consider the convolution diagram
    \begin{equation}
        \begin{tikzcd}
            (\Gr\btilde_G)^{I\cup J} & \Gr G_G \times_U \dots \times_U \Gr G_G \arrow{l}{f} \arrow{r}{q}& \frac{\Gr G_G \times_B \dots \times_B \Gr G_G}{G_{1,\infty,\sum_{\infty y}}^{I\cup J}} \arrow{r}{m}& \Gr \btilde^{(1,\dots,|I|+|J|)}_{G,I\cup J,y}.
        \end{tikzcd}
    \end{equation}
    The pullback $f$ is smooth, $q$ descends by the quotient of a group $G_{1,\infty,\sum_{\infty y}}^{I\cup J} \times T^{|I|+|J|-1}$ such that the sheaf \eqref{eqn: GrB sheaf} is equivariant with respect to this group, and the composite functor $m_! f^*$ was shown earlier to give a convolution product on the $U \backslash G / U$ factors.  Since the convolution is the identity on all except for one factor, the result after applying these functors is the sheaf
    \begin{equation}
        j_! \left(\mathcal{S}_{V_1} \widetilde{\boxtimes} \dots \widetilde{\boxtimes} \mathcal{S}_{V_{|I|+|J|}} \boxtimes {}_{\mathcal{L}'} M_{\mathcal{L}}\right),
    \end{equation}
    and by the properness of $m$ and smoothness of $f$, this sheaf is $\Psi$-factorizable with respect to maps to the base $C^J$.  By pushing forward along the map $c$ collapsing the legs, we arrive at the sheaf $j_! ( \mathcal{S}_{I\cup J,V \boxtimes W} \boxtimes {}_{\mathcal{L}'} M_{\mathcal{L}} )$, which is the sheaf we wanted to show to be $\Psi$-factorizable.
\end{proof}

\begin{remark}\label{remark: psi-factorizable GrB}
The relationships that are encoded by the $\Psi$-factorizability of the sheaves $(\mathfrak{p}_I, \mathcal{S}_{I,V} \boxtimes L^{\dot{w}}_{\lambda})$ together with the compatibility with the collapsing of the legs map $c$ should encode all the compatibilities showing that the nearby cycles is a central functor. 
\end{remark}

\begin{theorem}\label{thm: psi-factorizable GFq}
    Let $N = 0 \cdot y$ as above, so that $H = G$.  Let $\zeta$ be a representation of $G(\F_q)$.  Then $\zeta$ is a $\Psi$-factorizable representation in the sense of Definition~\ref{def: factorizable representation}.
\end{theorem}

\begin{proof}
    Write $\zeta$ as a summand of $\pi_! f^* \mathbb{L}^{\dot{w}}_{\lambda}$ for some $w \cdot \lambda$.  Then $j_! (\mathcal{S}_{I,V} \boxtimes \zeta)$ is a summand of $j_! (\pi_n)_! (f_n)^* (\mathcal{S}_{I,V} \boxtimes \mathbb{L}^{\dot{w}}_{\lambda})$, where $\pi_n$ and $f_n$ are now coming from the horocycle correspondence over restricted shtukas.  The result now follows from Proposition~\ref{prop: horocycle proper smooth}.
\end{proof}

\begin{remark}
    In Section~\ref{sec:depth 0}, we will freely use the fact that the above theorem implies that $\mathfrak{p}_!$ commutes with nearby cycles.
\end{remark}

\subsection{Truncated induction for the affine truncated convolution category}

Let $\Cht\mathcal{R}_0$ be the special fiber of restricted shtukas at an unramified point $y \in C(\mathbb{F}_q)$ as before, with $N = 0 \cdot y$ and with a single leg $I = \{ 1 \}$.  Similarly, we consider the special fibers $\Cht\btilde_{G,N,I,y}$ at $y$, which we denote $\Cht\btilde_0$, and $\Gr\btilde_{G,N,I,y}$ at $y$, which we denote $\Gr\btilde_0$.  The latter moduli space can be identified with $[I^{\circ} \backslash LG / I^{\circ}] \times [* / L^{\circ} G]$.  As in the previous section, we will keep a semisimple parameter $\mathfrak{o} = W \lambda$ fixed.

We consider the following diagram of derived categories of sheaves, where the vertical arrows are given by nearby cycles and residual unipotent parts of groups are ignored:
\usetikzlibrary {decorations.pathmorphing}
\begin{equation}
    \begin{tikzcd}
    D(G /_{\sigma} G \times L^+ G \backslash \Gr_G) \arrow[d, squiggly] & D(X \times L^+ G \backslash \Gr_G) \arrow{l}{\pi} \arrow{r}{f} \arrow[d, squiggly] & D(U \backslash G / U \times L^+ G \backslash \Gr_G) \arrow[d, squiggly] \\
    D(\Cht\mathcal{R}_0) & D(\Cht\btilde_0) \arrow{r}{f} \arrow{l}{\pi} & D(I^{\circ} \backslash LG / I^{\circ}).
    \end{tikzcd}
\end{equation}

The goal of this section will be to construct a category $\mathcal{P}(\Cht\mathcal{R}_0)_{\textbf{c}}$ and functors $\underline{\chi} \colon \mathcal{A}^{\lambda}_{\textbf{c}} \to \mathcal{P}(\Cht\mathcal{R}_0)_{\textbf{c}}$ and $\Psi \colon \mathcal{P}(G /_{\sigma} G)_{\textbf{c}} \to \mathcal{P}(\Cht\mathcal{R}_0)_{\textbf{c}}$.  These functors fit into a commutative diagram
\begin{equation}\label{eqn: truncated induction functors}
\begin{tikzcd}
    \mathcal{P}^{\textbf{c}}_{G}(\btilde^2) \arrow{r}{\underline{\chi}} \arrow{d}{\Psi}& \mathcal{P}(G /_{\sigma} G)_{\textbf{c}} \arrow{d}{\Psi} \\
    \mathcal{A}^{\mathfrak{o}}_{\textbf{c}} \arrow{r}{\underline{\chi}}& \mathcal{P}(\Cht\mathcal{R}_0)_{\textbf{c}}.
\end{tikzcd}
\end{equation}

\begin{definition}
    We define the category $\mathcal{P}(\Cht\mathcal{R}_0)_{\le \textbf{c}}$ as the category generated by the image of all subquotients of $R^j \pi_! f^* (Z^{w(\lambda)}_V * \mathbb{L}^{\dot{w}}_{\lambda})$ as $j\in \mathbb{Z}$, $V$ runs over all representations in $\Rep(\widehat{G})$, and $w \cdot \lambda$ runs over elements of all two-sided cells $\textbf{c}' \le \textbf{c} \subset \wtilde \times \mathfrak{o}$.  Similarly, we define $\mathcal{P}(\Cht\mathcal{R}_0)_{< \textbf{c}}$ using the two-sided cells strictly less than $\textbf{c}$.

    We define $\mathcal{P}(\Cht\mathcal{R}_0)_{\textbf{c}}$ as the Serre quotient $\mathcal{P}(\Cht\mathcal{R}_0)_{\le \textbf{c}} / \mathcal{P}(\Cht\mathcal{R}_0)_{< \textbf{c}}$.  By construction, this comes with a functor $\underline{\chi} \colon \mathcal{A}^{\mathfrak{o}}_{\textbf{c}} \to \mathcal{P}(\Cht\mathcal{R}_0)_{\textbf{c}}$.  The functor $\Psi$ on the right hand side of diagram~\eqref{eqn: truncated induction functors} is defined as nearby cycles of sheaves supported on $G /_{\sigma} G$ from the generic fiber of restricted shtukas with one leg to the special fiber.  Since this closed subset of restricted shtukas is constant over the curve, this is the same as the functor by direct image under the map $i_G \colon G /_{\sigma} G \to \Cht\mathcal{R}_0$.
\end{definition}

We want to justify that the functor $(i_G)_* = \Psi_{G /_{\sigma} G}$ is conservative after passing to Serre quotient categories.  This is justified in the following lemma.

\begin{lemma}
    If $\zeta$ is a representation of $G(\mathbb{F}_q)$ attached to the two-sided cell $\textbf{c}$, then considered as a sheaf on $G /_{\sigma} G$ embedded in $\Cht\mathcal{R}_0$, $\zeta$ does not live in $\mathcal{P}(\Cht\mathcal{R}_0)_{< \textbf{c}}$.
\end{lemma}

\begin{proof}
Let $i_G \colon G /_{\sigma} G \to \Cht\mathcal{R}_0$ be the embedding, which is closed up to a torsor for a unipotent group.  We also have closed (up to a torsor for a unipotent group) embeddings $i_X \colon X / G \to \Cht\btilde_0$ and $i_{\btilde} \colon (\btilde^2)/G \to [ I^{\circ} \backslash LG / I^{\circ} ]$.  By proper base change,
\begin{align*}
    i_G^* \pi_! f^* (Z^{w(\lambda)}_V * \mathbb{L}_{\lambda}^{\dot{w}}) &\cong \pi_! i_X^* f^* (Z^{w(\lambda)}_V * \mathbb{L}_{\lambda}^{\dot{w}}) \\
    &\cong \pi_! f^* i_{\btilde}^* (Z^{w(\lambda)}_V * \mathbb{L}^{\dot{w}}_{\lambda}).
\end{align*}
If $w \in \textbf{c}'$, the object $(Z^{w(\lambda)}_V * \mathbb{L}^{\dot{w}}_{\lambda})$ lives in the subcategory $\mathcal{H}^{\lambda}_{\le \textbf{c}'}$, and so $i_{\btilde}^* (Z^{w(\lambda)}_V * \mathbb{L}^{\dot{w}}_{\lambda})$ lives in the subcategory of the finite Hecke category generated by objects $\mathbb{L}^{\dot{w'}}_{\lambda}$ such that $w' \in \textbf{c}'' \le \textbf{c}' \cap W$.
\end{proof}

\begin{corollary}
The functor $(i_G)_* = \Psi_{G /_{\sigma} G}$ is faithful.
\end{corollary}

\begin{proof}
    By the above lemma, $(i_G)_*$ sends irreducible local systems attached to the two-sided cell $\textbf{c} \cap W \times \mathfrak{o}$ for the finite Weyl group to irreducible local systems attached to $\textbf{c}$.
\end{proof}

\begin{proposition}
    The diagram~\eqref{eqn: truncated induction functors} commutes.
\end{proposition}

\begin{proof}
Nearby cycles commutes with $R^{n_\textbf{c}} \pi_!$ and $f^*$ since $\pi$ is proper, $f$ is smooth, and the nearby cycles functor is exact, so the diagram commutes before passing to Serre quotient categories.  The diagram is well defined in terms of Serre quotient categories, so it is also commutative after passing to Serre quotient categories.
\end{proof}

\subsection{Computing monodromy of nearby cycles}

Our aim in this section will be to compute the monodromy of nearby cycles.  This surveys \cite{bezrukavnikov2004tensor} and \cite{bezrukavnikov2009tensor}.

\begin{definition}
    For $\lambda \in \mathfrak{s}_\infty$, we may form the character sheaf $L_{\lambda}$.  Making a choice of a tame generator produces an element in $\widehat{T}(\qlbar)$, which we also denote $\lambda$.  A two-sided cell $\textbf{c}$ in the affine Weyl group stabilizer $\wtilde_\lambda$ gives a unipotent conjugacy class in $\widehat{G}_{\lambda}$, let $\mathfrak{o}_{\textbf{c}}$ denote the conjugacy class in $\widehat{G}$ having semisimple part representing the element in $X^*(T) \otimes \qlbar^{\times}$ representing the map
    \[ \begin{tikzcd} X_*(T) \arrow{r}& \pi_1^t(T,\overline{1}) \arrow{r}& T[n] \arrow{r}{\lambda}& \qlbar^{\times} \end{tikzcd} \]
    and unipotent part given by the unipotent conjugacy class in $\widehat{G}_{\lambda} \subset \widehat{G}$ corresponding to $\textbf{c} \subset \wtilde_\lambda$.
\end{definition}

\begin{proposition}\label{prop:conjugacy}
Let $d$ be a Duflo involution in the two-sided cell $\textbf{c} \cap W_{\lambda} \times \mathfrak{o}$.  In the truncated convolution category $\mathcal{A}^{\lambda}_{\textbf{c}}$, the action of monodromy on $\Psi(S_V \boxtimes \mathbb{L}^{\dot{d}}_{\lambda})$ factors through the tame fundamental group, where the tame generator acts by an element of $\widehat{G}$ and has semisimple part given by $\lambda^{-1}$ and unipotent part given by an element $N_{d}$ in the unipotent conjugacy class of $\widehat{G}_{\lambda}(\qlbar) \subset \widehat{G}(\qlbar)$ associated to $\textbf{c}$.

That is, the functor $\Rep(\widehat{G}) \to \mathcal{A}_{\textbf{c}}$ given by $V \mapsto Z^{\lambda}_V * \mathbb{L}^{\dot{d}}_{\lambda}$ extends to a functor $\Rep(\widehat{H}) \to \mathcal{A}_{\textbf{c}}$ for some subgroup $\widehat{H} \subset \widehat{G}_{\lambda} \subset \widehat{G}$, under which the monodromy automorphism acts on $Z^{\lambda}_V * \mathbb{L}^{\dot{d}}_{\lambda}$ by $v \mapsto m \cdot v \colon V \to V$ as a morphism in $\Rep(\widehat{H})$.  Here, $m \in \widehat{G}$ given by $\lambda^{-1} N_d \in \mathfrak{o}_{\textbf{c}}$.
\end{proposition}

This proof is a straightforward adaptation of \cite[Theorem~2]{bezrukavnikov2004tensor} and we will indicate what needs to be changed in the proof.  First, what \cite[Theorem~1]{bezrukavnikov2009tensor} already prove is that the functor $V \mapsto Z^{\lambda}_V$ can be upgraded to a central functor
\[ \Rep(\widehat{G}_{\lambda}) \to {}_{L_{\lambda}} D_{L_{\lambda}}(I^{\circ} \backslash LG / I^{\circ}), \]
and we may write $Z^{\lambda}_V$ for any $V \in \Rep(\widehat{G}_{\lambda})$, which is unambiguous because this functor is compatible with restriction from $\Rep(\widehat{G})$.

We will review some facts about the monodromic Hecke algebra $\mathbf{H}$.  For any algebra $A$, let $Z(A)$ denote the center of $A$, and let $\mathcal{A} = \mathbb{Z}[v,v^{-1}]$, using a subscript $\mathcal{A}$ to denote tensoring with $\mathcal{A}$ over the integers.  Following \cite{bezrukavnikov2009tensor}, let $\mathbf{H}_{\lambda}$ be the affine Hecke algebra of the reductive stabilizer $\widehat{G}_{\lambda}$, which is possibly disconnected, which is an algebra over $\mathbb{Z}[v,v^{-1}]$ with a Kazhdan-Lusztig basis given by $C_w$ for $w \in W_{\lambda}$.  We have an isomorphism
\[ \mathbf{H}_{\lambda} \cong K_0({}_{L_{\lambda}} D_{L_{\lambda}}(I^{\circ} \backslash LG / I^{\circ}))_{\mathcal{A}}, \]
given by sending $C_w$ to the class $[\mathbb{L}^{\dot{w}}_{\lambda}]$.  Two-sided cells define two-sided ideals in $\mathbf{H}_{\lambda}$ and the quotients of these two sided ideals define algebras $\mathbf{J}^{\lambda}_{\textbf{c}}$, summands of the affine asymptotic Hecke algebra $\mathbf{J}^{\lambda}$, which are Grothendieck groups of the categories $\mathcal{A}^{\lambda}_{\textbf{c}}$:
\[ \mathbf{J}^{\lambda}_{\textbf{c}} \cong K_0(\mathcal{A}^{\lambda}_{\textbf{c}}). \]
The algebra $\mathbf{J}^{\lambda}_{\textbf{c}}$ has a basis $t_w = [\mathbb{L}^{\dot{w}}_{\lambda}]$ for $w \in \wtilde_{\lambda}$.  The construction of central sheaves give algebra maps
\begin{equation}
\begin{aligned}
B := V \mapsto [Z_V^{\lambda}] &\colon K_0(\Rep(\widehat{G}_{\lambda}))[v,v^{-1}] \to \mathbf{H}_{\lambda}, \\
V \mapsto \left[\bigoplus_d Z_V^{\lambda} * \mathbb{L}^{\dot{d}}\right] &\colon K_0(\Rep(\widehat{G}_{\lambda}))_{\mathcal{A}} \to (\mathbf{J}^{\lambda}_{\textbf{c}})_{\mathcal{A}},
\end{aligned}
\end{equation}
into the center of the target algebras.  There is also an algebra map $\phi_{\textbf{c}} \colon Z(\mathbf{H}_{\lambda}) \to Z(\mathbf{J}^{\lambda}_{\textbf{c}})_{\mathcal{A}}$ sending an element $h$ to $h * [\sum_{d \in \mathbf{D}_{\textbf{c}} \cap \wtilde_{\lambda} \times \{ \lambda \}} t_d]$ and an algebra map $\phi_d$ using just a single Duflo involution $d$.

We use the following characterization of the unipotent orbit defined by $\textbf{c}$.  A unipotent element $N$ of $\widehat{G}_{\lambda}$ defines a map $SL_2 \to \widehat{G}_{\lambda}$ sending the matrix $E = \begin{pmatrix} 1 & 1 \\ 0 & 1 \end{pmatrix}$ to $N$.  Let $s^v_N \in \widehat{G}_{\lambda}(\qlbar(v))$ be the image diagonal matrix $\begin{pmatrix} v & 0 \\ 0 & v^{-1} \end{pmatrix}$ under the map to $\widehat{G}_{\lambda}$.

We now fix a Duflo involution $d$.  For Tannakian reasons (\cite[Theorem~1]{bezrukavnikov2004tensor} and ), there is a subgroup $\widehat{H}_d \subset \widehat{G}_{\lambda}$ such that $\Rep(\widehat{H}_d) \cong \mathcal{A}^{\lambda}_d$ compatible with the upgraded central sheaf functor $\Rep(\widehat{G}_{\lambda}) \to \mathcal{A}^{\lambda}_d$.  The monodromy automorphism defines an element $N_d$ of $\widehat{G}_{\lambda}$ that commutes with $\widehat{H}_d$.  Semisimple conjugacy classes of $\widehat{H}_d$ correspond to characters of $K_0(\Rep(\widehat{H}_d))$ and therefore characters of $K_0(\mathcal{A}^{\lambda}_d)$.  So given a semisimple conjugacy class $s$, we arrive at a character of $K_0(\mathcal{A}^{\lambda}_d)$ and by extension of scalars a character of $\mathbf{H}_{\lambda}$, which in turn defines a semisimple conjugacy class $\Omega_s$ of $\widehat{G}_{\lambda}$, which is characterized as being of the form $s^v_{N_{\textbf{c}}} \cdot s$ for some $N_{\textbf{c}}$ in the unipotent orbit determined by $\textbf{c}$.

\begin{proof}[Proof of Proposition~\ref{prop:conjugacy}]
We can write
\begin{equation}
\Psi(S_V \boxtimes \mathbb{L}^{\dot{d}}_{\lambda}) \cong Z^{\lambda}_V * \mathbb{L}^{\dot{d}}_{\lambda},
\end{equation}
and we note that the semisimple part of monodromy is determined by its action on $Z^{\lambda}_V$, where it acts by $\lambda^{-1}$ by \cite[Section~5.2]{arkhipov2009perverse}.  When $d$ is trivial, the unipotent part of monodromy was considered in \cite{bezrukavnikov2009tensor}, and we review the generalization to arbitrary Duflo involutions. In the case that $\lambda$ is trivial, the computation of the unipotent part reduces to \cite[Theorem~2]{bezrukavnikov2004tensor}.  In the general case, we need to compute the element $N_d$ above and show that it corresponds to the conjugacy class of the two-sided cell $\textbf{c}$.

The proof in \cite[Theorem~2]{bezrukavnikov2004tensor} uses Bernstein's isomorphism, which we need for monodromic Hecke algebras and their summands.  The logarithm of $N_d$ on $Z^{\lambda}_V * \mathbb{L}^{\dot{w}}_{\lambda}$ induces a monodromy filtration of this object, given by the Wakimoto filtration coming from \cite{bezrukavnikov2009tensor}  First, we need that for any $V \in \Rep(\widehat{G}_{\lambda})$,
\begin{equation}\label{eqn:bez-monodromy-weight}
    \phi_{\textbf{c}}(B(V)) = \sum_{d} \sum_i v^i \left[ \mathrm{gr}^i(V) \right].
\end{equation}
This property follows by the coincidence of the monodromy filtration and the weight filtration for a perverse sheaf pure of weight $0$, after which the proof in \cite[Lemma~6]{bezrukavnikov2004tensor} applies.

Now for any irreducible $\mathbf{J}^{\lambda}_{\textbf{c}}$-module $M$ over $\qlbar$ such that $t_d$ acts by a nonzero operator, we produce an action of $t_d \mathbf{J}^{\lambda}_{\textbf{c}} t_d$ on $t_d M$, which we can think of as a coherent sheaf over the affine space of semisimple conjugacy classes in $\widehat{H}_d$, as the latter algebra is commutative and identified with $K_0(\mathcal{A}^{\lambda}_d)$.  Since $t_d M$ is finite dimensional, we find a simultaneous eigenvector $u$ which defines a semisimple conjugacy class $s_u$.  The character of $Z(\mathbf{H}^{\lambda})$ acts as well and defines a semisimple conjugacy class $s'_u$.  Equation~\eqref{eqn:bez-monodromy-weight} computes the relation between these two conjugacy classes $s'_u$ and $s_u$ by $s'_u = s^v_{N_d} \cdot s_u$.  But this property also characterizes $N_{\textbf{c}}$, so $N_d$ and $N_{\textbf{c}}$ must coincide as conjugacy classes in $\widehat{G}_{\lambda}$.
\end{proof}

\begin{corollary}\label{corr:conjugacy Alambda}
    Let $w \in W_\lambda$ be an element of the finite Weyl group stabilizing $\lambda$ such that $w \cdot \lambda$ lives in the two-sided cell $\textbf{c}$.  Suppose that, inside this two-sided cell, $w$ lives in the left cell containing the Duflo involution $d \in W_{\lambda}$ in the finite Weyl group.  In the truncated convolution category $\mathcal{A}^{\lambda}_{\textbf{c}}$, the funcator $V \mapsto \Psi(S_V \boxtimes \mathbb{L}^{\dot{w}}_{\lambda})$ can be written as
    \[ V \mapsto (Z^{\lambda}_V * \mathbb{L}^{\dot{d}}_{\lambda}) \bullet \mathbb{L}^{\dot{w}}_{\lambda}. \]
    The action of monodromy on nearby cycles factors through the tame fundamental group, and in the truncated convolution category, the tame generator gives a monodromy automorphism that represents an element in $\mathfrak{o}_{\textbf{c}}$ in $\widehat{G}_{\lambda}$.  That is, monodromy acts with semisimple part $\lambda^{-1}$ and unipotent part $N_d \bullet 1$ from Proposition~\ref{prop:conjugacy}.
\end{corollary}

\begin{proof}
    The truncated convolution $(Z^{\lambda}_V * \mathbb{L}^{\dot{d}}_{\lambda}) \bullet \mathbb{L}^{\dot{w}}_{\lambda}$ is itself a nearby cycle
    \[ R^{2a(\textbf{c})+\dim T} m_! f^* \Psi (\mathcal{S}_V \boxtimes \mathbb{L}^{\dot{d}}_{\lambda} \boxtimes \mathbb{L}^{\dot{w}}_{\lambda}), \]
    mod the Serre subcategory $< \textbf{c}$, where $m_! f^*$ give convolution.  The result follows as $\Psi$ commutes with $m_! f^*$ by the properness of $m$ and smoothness of $f$ in the convolution diagram~\eqref{eqn:convolution diagram}.  This base change of $\Psi$ against $m_! f^*$ also preserves the monodromy automorphism, which allows us to characterize the monodromy using the monodromy of $(Z^{\lambda}_V * \mathbb{L}^{\dot{d}}_{\lambda})$ computed in Proposition~\ref{prop:conjugacy}.
\end{proof}

\begin{corollary}\label{corr:conjugacy Ao}
    Let $w \cdot \lambda \in \textbf{c}$.  In the truncated convolution category $\mathcal{A}^{\mathfrak{o}}_{\textbf{c}}$, the functor
    \begin{equation}
    \begin{aligned}
    \Rep(\widehat{G}) &\to \mathcal{A}^{\mathfrak{o}}_{\textbf{c}},\\
    V &\mapsto \Psi(S_V \boxtimes \mathbb{L}^{\dot{w}}_{\lambda}) \cong Z^{w(\lambda)}_{V} * \mathbb{L}^{\dot{w}}_{\lambda}
    \end{aligned}
    \end{equation}
    admits a monodromy action that factors through the tame fundamental group, and the monodromy automorphism represents an element in $\mathfrak{o}_{\textbf{c}}$.
\end{corollary}

\begin{proof}
    Write $w = w' m$ with $m = w(w(\lambda),\lambda)$.  Convolution with minimal IC sheaves give an isomorphism, and on monodromy, we have
    \[ Z^{w(\lambda)}_{V} * \mathbb{L}^{\dot{w}}_{\lambda} \cong Z^{w(\lambda)}_{V} * \mathbb{L}^{\dot{w'}}_{w(\lambda)} * \mathbb{L}^{\dot{m}}_{\lambda}, \]
    where monodromy acts on the left by $\mathfrak{M}_{V} * 1$ and on the right by $\mathfrak{M}_{V,w'} * 1$, where $\mathfrak{M}_{V,w'}$ is the monodromy of nearby cycles on $\Psi(S_V \boxtimes \mathbb{L}^{\dot{w'}}_{\lambda})$ in the category $\mathcal{A}^{\lambda}_{\textbf{c}}$.  So the result follows from Corollary~\ref{corr:conjugacy Alambda} together with the fact that convolution with a minimal IC sheaf is an equivalence of categories.
\end{proof}

\begin{theorem}\label{thm: monodromy chtr}
    Let $\zeta \in D(G /_{\sigma} G)$ correspond to an irreducible representation and let $\textbf{c}$ be its associated two-sided cell.  Consider the exact functor $\Rep(\widehat{G}) \to D(\Cht\mathcal{R}_0)_{\textbf{c}}$ given by $V \mapsto \Psi(\mathcal{S}_V \boxtimes \zeta)$.  Then there is a subgroup $\widehat{H} \subset \widehat{G}$ such that this functor extends to an exact fiber functor $\Rep(\widehat{H}) \to D(\Cht\mathcal{R}_0)_{\textbf{c}}$ and the monodromy automorphism functor corresponds to an element $m \in \mathfrak{o}_{\textbf{c}}$.
\end{theorem}

\begin{proof}
    The sheaf $\zeta$ is a summand of $\underline{\chi}(\mathbb{L}^{\dot{w}}_{\lambda})$ for some $w \in \textbf{c}$ by assumption.  The monodromy of $\Psi(\mathcal{S}_V \boxtimes \zeta)$ is a summand of the monodromy $\underline{\chi}(\mathfrak{M})$ of $\underline{\chi} \Psi(\mathcal{S}_V \boxtimes \mathbb{L}^{\dot{w}}_{\lambda})$ by the commutative diagram~\eqref{eqn: truncated induction functors} combined with the fact that nearby cycles commutes with pushforward along a proper map and pullback along a smooth map.  So the result follows from Corollary~\ref{corr:conjugacy Ao}.
\end{proof}

We also recall the following proposition from \cite[Section~11.3]{bezrukavnikov2016two}.  We introduce notation, setting $\mathrm{Coh}(X)$ as the category of coherent sheaves on $X$ and let $D^b(\mathrm{Coh}(X))$ be the bounded derived category of coherent sheaves.  For the base field $\qlbar$, we will write $\tilde{\mathcal{N}}$ for the Springer resolution of the nilpotent cone, $\tilde{\mathfrak{g}}$ for the Grothendieck resolution of the Lie algebra with map $\tilde{\mathfrak{g}} \to \mathfrak{g}$, and use $\mathrm{St}'$ to denote a one-sided Steinberg variety $\tilde{\mathfrak{g}} \times_{\mathfrak{g}} \tilde{\mathcal{N}}$.  Bezrukavnikov introduces an equivalence $\Phi_{I^{\circ} I} \colon D(I^{\circ} \backslash LG / I, \qlbar) \cong D^b(\mathrm{Coh}^{\widehat{G}}(\mathrm{St}'))$ and notes that both sides are stratified by two-sided cells.  The left side is stratified by $\wtilde$ in the same way that the usual affine Hecke algebra is, and we form a stratification by two-sided cells, closed under convolution, by considering subcategory $D(I^{\circ} \backslash LG / I, \qlbar)_{\le \textbf{c}}$ generated by simple objects $\mathbb{L}_w$ with $w \in \bigcup_{\textbf{c}' \le \textbf{c}} \textbf{c}' \subset \wtilde$.  Letting $\mathfrak{o}_{\textbf{c}}$ be the corresponding nilpotent orbit as an orbit on the Lie algebra, let $D^b(\mathrm{Coh}^{\widehat{G}}(\mathrm{St}'))_{\le \textbf{c}}$ be the subcategory consisting of coherent sheaves supported over the closure of $\mathfrak{o}_{\textbf{c}}$.

\begin{proposition}[Bezrukavnikov]\label{prop:bezcoherent}
    $\Phi_{I^{\circ} I}$ restricts to an equivalence
    \begin{equation}
        \Phi_{I^{\circ} I,\le \textbf{c}} \colon D(I^{\circ} \backslash LG / I)_{\le \textbf{c}} \simeq D^b(\mathrm{Coh}^{\widehat{G}}(\mathrm{St}'))_{\le \textbf{c}}.
    \end{equation}
\end{proposition}

\section{Application to the depth 0 Langlands correspondence}\label{sec:depth 0}

\subsection{The regular representation of $\widehat{G}$}

In this section, if $V$ denotes a representation of $\widehat{G}$ then $\underline{V}$ denotes the underlying vector space with trivial $\widehat{G}$-action.  A particularly important representation of $\widehat{G}$ is the left regular representation whereby $\widehat{G}$ acts on the ring of regular functions $\mathcal{O}_{\widehat{G}}$.  Here, we will take $\widehat{G}$ to live over the field $E$ which is a finite extension of $\mathbb{Q}_{\ell}$ or $\qlbar$.

Elements of the left regular representation $\Reg$ of $\widehat{G}$ in $E$-vector spaces can be viewed as functions $f \colon \widehat{G} \to E$, where the $\widehat{G}$-action is given by $h \cdot f(g) = f(h^{-1} g)$.  Given $x \in V$ and $\xi \in V^*$, we can construct such a function by $f(g) = \langle x, g\cdot \xi \rangle$.  The map $V \otimes \underline{V^*} \to \Reg$ is $\widehat{G}$-equivariant because $\langle hx, g\xi \rangle = \langle x, h^{-1}g\xi \rangle = f(h^{-1}g)$.

Given the representation $\Reg$, we can identify $\Reg \otimes V$ with functions $\widehat{G} \to V$.  The diagonal $\widehat{G}$ action is by $h \cdot (f \otimes v)(g) = f(h^{-1} g) h v$.  We define a map
\[ \theta \colon \Reg \otimes \underline{V} \to \Reg \otimes V \]
by $\theta(f \otimes v)(g) = f(g) gv$.  The equivariance of $\theta$ is checked as
\begin{align*}
    \theta(h \cdot (f \otimes v))(g) &= \theta(hf \otimes v)(g) = f(h^{-1}g) gv, \\
    (h \cdot \theta(f \otimes v))(g) &= h \theta(f \otimes v)(h^{-1} g) = f(h^{-1} g) hh^{-1} gv.
\end{align*}

Given an element $\gamma \in \widehat{G}$, we will consider another map $F_{\gamma} \colon \Reg \otimes \underline{V} \to \Reg \otimes \underline{V}$ of vector spaces (not of $\widehat{G}$-representations) defined by
\[ F_\gamma (f \otimes v)(g) = f(g) g^{-1} \gamma g v. \]
It is easy to check that if $(1 \otimes \gamma)$ acts on $\Reg \otimes V$ via the action $\widehat{G} \times \widehat{G}$, then $F_{\gamma}= \theta^{-1} \circ (1 \otimes \gamma) \circ \theta$.

The map $F_\gamma$ is not a map in the category $\Rep(\widehat{G})$ but it is a map in the category $\Rep(Z(\gamma))$, the centralizer of $\gamma$ in $\widehat{G}$.  Therefore, if we are given a central functor from $\Rep(Z(\gamma)) \to \mathcal{C}$ for some suitable category $\mathcal{C}$ of sheaves, $F_\gamma$ will give well-defined maps between sheaves.

Let $f'$ be another regular function on $\widehat{G}$.  Let $\langle \cdot, \cdot \rangle$ be the trace pairing $V \otimes V^* \to 1$, which is $\widehat{G}$-invariant under the diagonal $\widehat{G}$-action.  We know that $\phi(x \otimes \xi)(g) = f'(g) = \langle x, g\xi \rangle$ for some finite-dimensional representation $V$ and elements $x \in V$ and $\xi \in V^*$, and the map $\phi \colon V \otimes \underline{V^*} \to \Reg$ is $\widehat{G}$-equivariant, as
\[ hf'(g) = \langle x, h^{-1} g\xi \rangle = \langle hx, g\xi \rangle = \phi(hx \otimes \xi). \]

Let $\mathfrak{o}_{\gamma}$ be the conjugacy class of $\gamma$.  The closure of the conjugacy class $\overline{\mathfrak{o}_{\gamma}}$ then defines an ideal in the ring of regular functions, which we denote $\mathscr{J}_{\mathfrak{o}_{\gamma}}$.  Let $x \otimes \xi \in V \otimes V^*$ be chosen so that $f' = \phi(x \otimes \xi)$ is nonvanishing everywhere on $\mathfrak{o}_{\gamma}$.  Then the composition
\begin{equation}\label{eqn:pre-excursion}
    \begin{tikzcd}
    \Reg \arrow{r}{x} & \Reg \otimes \underline{V} \arrow{r}{F_{\gamma}} & \Reg \otimes \underline{V} \arrow{r}{\xi} & \Reg,
    \end{tikzcd}
\end{equation}
sends $f(g) \in \Reg$ to $f(g) f'(g^{-1} \gamma^{-1} g)$ and therefore must be an injection.  In light of this, we will use the following lemma to show that a certain excursion operator is injective.

For any conjugacy class $\mathfrak{o}_{\gamma'} \subset \overline{\mathfrak{o}_{\gamma}}$ strictly contained in the boundary so that $\mathfrak{o}_{\gamma'} \ne \mathfrak{o}_{\gamma}$, we can find a function $f'$ in the ideal $\mathscr{J}_{\mathfrak{o}_{\gamma'}} \setminus \mathscr{J}_{\mathfrak{o}_{\gamma}}$.  By exponentiating, we can even assume that $f' \in \mathscr{J}_{\mathfrak{o}_{\gamma'}}^n$.

\subsection{Truncated convolution category and excursion operator}\label{sec:truncatedexcursion}

The composition in Equation~\eqref{eqn:pre-excursion} defines a map in $\Rep(\widehat{G}_{\gamma})$.  On the other hand, if $d$ is a Duflo involution and $\gamma = s^{-1} N_d$, then we have a functor $\Rep(\widehat{G}_{\gamma}) \to \mathcal{A}^{\lambda}_d$ into a truncated convolution category by Proposition~\ref{prop:conjugacy}.

This gives a map
\begin{equation}\label{eqn:hecke pre-excursion}
    \begin{tikzcd}
    Z^{\lambda}_{\Reg} * \mathbb{L}^{\dot{d}}_{\lambda} \arrow{r}{x} & Z^{\lambda}_{\Reg \otimes \underline{V}} * \mathbb{L}^{\dot{d}}_{\lambda} \arrow{r}{F_{\gamma}} & Z^{\lambda}_{\Reg \otimes \underline{V}} * \mathbb{L}^{\dot{d}}_{\lambda} \arrow{r}{\xi} & Z^{\lambda}_{\Reg} * \mathbb{L}^{\dot{d}}_{\lambda},
    \end{tikzcd}
\end{equation}
in the truncated convolution category.  Consider two cases
\begin{enumerate}
    \item Suppose $\phi(x,\xi) \in \mathscr{J}_{\mathfrak{o}_{\gamma}}$.  Then the composition in $\Rep(\widehat{G}_{\gamma})$ is $0$ by definition.
    \item Suppose $\phi(x,\xi) \not \in \mathscr{J}_{\mathfrak{o}_{\gamma}}$.  Then the composition in $\Rep(\widehat{G}_{\gamma})$ is an injection.
\end{enumerate}

Moreover, if $d \le w_0$ the long Weyl group element in the finite Hecke algebra, we can identify $Z(V) * \mathbb{L}^{\dot{d}}_{\lambda}$ as a nearby cycles sheaf for $\mathcal{S}_V \boxtimes \mathbb{L}^{\dot{d}}_{\lambda}$.  Under this identification, the map $F_{\gamma}$ becomes the action of the monodromy of nearby cycles in $\mathcal{A}^{\lambda}_d$ under the identification of $\mathfrak{M}_V$ with the action of $\lambda^{-1} N_d$.
\begin{equation}
    \begin{tikzcd}
    Z^{\lambda}_{\Reg \otimes \underline{V}} * \mathbb{L}^{\dot{d}}_{\lambda} \arrow{r}{\theta} & Z^{\lambda}_{\Reg} * (Z^{\lambda}_{V} * \mathbb{L}^{\dot{d}}_{\lambda}) \arrow{r}{1 * \mathfrak{M}_V} & Z^{\lambda}_{\Reg} * (Z^{\lambda}_{V} * \mathbb{L}^{\dot{d}}_{\lambda}) \arrow{r}{\theta^{-1}} & Z^{\lambda}_{\Reg \otimes \underline{V}} * \mathbb{L}^{\dot{d}}_{\lambda},
    \end{tikzcd}
\end{equation}
which makes sense in both the affine Hecke algebra and the truncated convolution category.

Now consider truncated convolution
\[Z^{\lambda}_{\Reg \otimes \underline{V}} * \mathbb{L}^{\dot{d}}_{\lambda} \bullet \mathbb{L}^w_{\lambda}\]
for fixed $w \cdot \lambda \in \textbf{c}$ we may write $w = w' m$ with $m = w(w(\lambda),\lambda)$ and $w' \in W_{w(\lambda)}$.  Let $d$ be the Duflo involution in the left cell containing $w'$, and consider the map
\[1 * \mathfrak{M}_{V,d} \bullet 1 * 1 \colon Z^{w(\lambda)}_{\Reg} * (Z^{w(\lambda)}_{V} * \mathbb{L}^{\dot{d}}_{w(\lambda)}) \bullet \mathbb{L}^{\dot{w}}_{w(\lambda)} * \mathbb{L}^{\dot{m}}_{\lambda} \to Z^{w(\lambda)}_{\Reg} * (Z^{w(\lambda)}_{V} * \mathbb{L}^{\dot{d}}_{w(\lambda)}) \bullet \mathbb{L}^{\dot{w}}_{w(\lambda)} * \mathbb{L}^{\dot{m}}_{\lambda},\]
or the arrow in $\mathcal{H}^{\mathfrak{o}}_{\le \textbf{c}}$,
\[ 1 * \mathfrak{M}_{V,w} \colon Z^{w(\lambda)}_{\Reg} * (Z^{w(\lambda)}_{V} * \mathbb{L}^{\dot{w}}_{\lambda}) \to Z^{w(\lambda)}_{\Reg} * (Z^{w(\lambda)}_{V} * \mathbb{L}^{\dot{w}}_{\lambda}). \]
We observe that in the composition
\begin{equation}
    \begin{tikzcd}
    Z^{w(\lambda)}_{\Reg \otimes \underline{V}} * \mathbb{L}^{\dot{w}}_{\lambda} \arrow{r}{\theta} & Z^{w(\lambda)}_{\Reg} * (Z^{w(\lambda)}_{V} * \mathbb{L}^{\dot{w}}_{\lambda}) \arrow{d}{1 * \mathfrak{M}_{V,w}} \\
    Z^{w(\lambda)}_{\Reg \otimes \underline{V}} * \mathbb{L}^{\dot{w}}_{\lambda} & Z^{w(\lambda)}_{\Reg} * (Z^{w(\lambda)}_{V} * \mathbb{L}^{\dot{w}}_{\lambda}) \arrow{l}{\theta^{-1}},
    \end{tikzcd}
\end{equation}
in $\mathcal{H}_{\le \textbf{c}}^{\mathfrak{o}}$, resp. in $\mathcal{A}^{\mathfrak{o}}_{\textbf{c}}$, the arrow $1 * \mathfrak{M}_{V,w}$ is the map for sheaves on a two-dimensional nearby cycle with the first leg containing the $\Reg$ representation and the second containing the $V$ representation.  Namely, this map is monodromy of nearby cycles for the second leg and identity on the first leg starting with the sheaf $\mathcal{S}_{\Reg} \boxtimes \mathcal{S}_V \boxtimes \mathbb{L}^{\dot{w}}_{\lambda}$ on $\Gr\btilde_{G,N,I,y}$ with $I = \{ 0,1 \}$.  To make this identification, we are implicitly using \ref{prop:Grbtilde psi factorizable}.

Now we apply induction, resp. truncated induction, using the diagram in Equation~\eqref{eqn: truncated induction functors} and Theorem~\ref{thm: monodromy chtr}.  Under truncated induction, our composition becomes
\begin{equation}
    \begin{tikzcd}
    \Psi(\mathcal{S}_{\Reg} \boxtimes \underline{\chi}(\mathbb{L}^{\dot{w}}_{\lambda})) \arrow{r}{x} & \Psi(\mathcal{S}_{\Reg \otimes \underline{V}} \boxtimes \underline{\chi}(\mathbb{L}^{\dot{w}}_{\lambda})) \arrow{d}{F_{\gamma}} \\
    \Psi(\mathcal{S}_{\Reg} \boxtimes \underline{\chi}(\mathbb{L}^{\dot{w}}_{\lambda})) & \Psi(\mathcal{S}_{\Reg \otimes \underline{V}} \boxtimes \underline{\chi}(\mathbb{L}^{\dot{w}}_{\lambda})) \arrow{l}{\xi}.
    \end{tikzcd}
\end{equation}
So over $\Cht\mathcal{R}^{ad}$ and over $\Cht_{G,I}$, we get a composition
\begin{equation}\label{eqn:F-excursion}
    \begin{tikzcd}
    \Psi \mathscr{F}_{\Reg,\underline{\chi}(\mathbb{L}^{\dot{w}}_{\lambda})} \arrow{r}{x} & \Psi \mathscr{F}_{\Reg \otimes \underline{V},\underline{\chi}(\mathbb{L}^{\dot{w}}_{\lambda})} \arrow{r}{F_{\gamma}} & \Psi \mathscr{F}_{\Reg \otimes \underline{V},\underline{\chi}(\mathbb{L}^{\dot{w}}_{\lambda})} \arrow{r}{\xi} & \Psi \mathscr{F}_{\Reg,\underline{\chi}(\mathbb{L}^{\dot{w}}_{\lambda})}.
    \end{tikzcd}
\end{equation}

Finally pushing forward to the curve, under the functor $\mathfrak{p}_!$, the composition becomes
\begin{equation}\label{eqn:H-excursion}
    \begin{tikzcd}
    \mathfrak{p}_!\Psi \mathscr{F}_{\Reg,\underline{\chi}(\mathbb{L}^{\dot{w}}_{\lambda})} \arrow{r}{x} & \mathfrak{p}_!\Psi \mathscr{F}_{\Reg \otimes \underline{V},\underline{\chi}(\mathbb{L}^{\dot{w}}_{\lambda})} \arrow{d}{F_{\gamma}} \\
    \mathfrak{p}_!\Psi \mathscr{F}_{\Reg,\underline{\chi}(\mathbb{L}^{\dot{w}}_{\lambda})} & \mathfrak{p}_!\Psi \mathscr{F}_{\Reg \otimes \underline{V},\underline{\chi}(\mathbb{L}^{\dot{w}}_{\lambda})} \arrow{l}{\xi}.
    \end{tikzcd}
\end{equation}
We call this operator $F_{\phi(x,\xi),\mathfrak{M}}$ for reasons that will be apparent after we identify this with an element of the framed excursion algebra.  Letting $1 \subset \Reg$ be the trivial representation inside the regular representation, we can consider the restriction of $F_{\phi(x,\xi),\mathfrak{M}}$ along $\mathfrak{p}_! \mathscr{F}_{\{0\},1,\underline{\chi}(\mathbb{L}^{\dot{w}}_{\lambda})} \to \mathfrak{p}_! \mathscr{F}_{\{0\},\Reg,\underline{\chi}(\mathbb{L}^{\dot{w}}_{\lambda})}$, which we call $F_{1,\phi(x,\xi),\mathfrak{M}}$ to distinguish it.

\begin{theorem}\label{thm:upper bound}
Suppose $\phi(x,\xi)$ is in $\mathscr{J}_{\mathfrak{o}_{\textbf{c}}}$.  Then if $w \cdot \lambda \in \textbf{c} \cap W \times \mathfrak{s}_{\infty}$ the composition $F_{1,\phi(x,\xi),\mathfrak{M}}$ is $0$.

If in addition $\lambda = 1$ (and conjecturally for $\lambda$ general), $F_{\phi(x,\xi),\mathfrak{M}} = 0$ in Equation~\eqref{eqn:H-excursion}.
\end{theorem}

\begin{proof}
    We prove the case $\lambda = 1$ first.  Under Bezrukavnikov's equivalence Proposition~\ref{prop:bezcoherent}, for any sheaf $M$ in the affine Hecke category, we may consider objects $Z_V * M$ as living in $D_{I^{\circ} I}$ under averaging, which is compatible with the action of nearby cycles.  The functor $\Phi_{I^{\circ} I}$ sends objects of $Z_V * M$ to $V \otimes \mathscr{O}_{\tilde{\widehat{\mathfrak{g}}}} \otimes \Phi_{I^{\circ} I}(M)$, where $V \otimes \mathscr{O}_{\tilde{\widehat{\mathfrak{g}}}}$ is considered as a coherent sheaf with a tautological $\widehat{G}$ action on $V$ as a representation.  Monodromy acts fiberwise by a tautological endomorphism so that the fiber over $g$ acts by $g$.  We note that in this case, everything is actually supported over the closure of the relevant two-sided cell in the nilpotent cone, so we observe that the composition $Z_{\Reg} * M \to Z_{\Reg} * Z_V * M \to Z_{\Reg} * Z_V * M \to Z_{\Reg} * M$ as in \eqref{eqn:hecke pre-excursion} is $0$.  This proves the second statement.

    Now we will prove the first statement, that $F_{1,\phi(x,\xi),\mathfrak{M}} = 0$.  We note that $F_{1,\phi(x,\xi),\mathfrak{M}}$ is the functor $\mathfrak{p}_! \epsilon^*$ applied to a map
    \begin{equation}\label{eqn:F1 map}
    \begin{tikzcd}
    F_1 \colon \Psi(\mathcal{S}_{1} \boxtimes \underline{\chi}(\mathbb{L}^{\dot{w}}_{\lambda})) \arrow{r}{\subset}& \Psi(\mathcal{S}_{\Reg} \boxtimes \underline{\chi}(\mathbb{L}^{\dot{w}}_{\lambda})) \arrow{r}{F}& \Psi(\mathcal{S}_{\Reg} \boxtimes \underline{\chi}(\mathbb{L}^{\dot{w}}_{\lambda})).
    \end{tikzcd}
    \end{equation}
    We will adopt the abbreviation $\Psi = \Psi(\mathcal{S}_{\Reg} \boxtimes \underline{\chi}(\mathbb{L}^{\dot{w}}_{\lambda}))$.  If we consider this map under the Serre quotient functor $Q$ to $D(\Cht\mathcal{R}_0)_{\textbf{c}}$, we see that the map $F$, and therefore $F_1$, becomes $0$ in the quotient category.  We argue that this implies that the composition above is $0$ \emph{before} passing to the quotient category.  Let $i_G$ be the inclusion $G /_{\sigma} G \to \Cht\mathcal{R}_0$, and let $j$ be the inclusion of the complement.  Composing $F_1$ with $\Psi \to j_* j^* \Psi$ gives by adjunction $j^* \underline{\chi}(\mathbb{L}^{\dot{w}}_{\lambda}) \to j^* \Psi$, which is $0$ by noting that $\underline{\chi}(\mathbb{L}^{\dot{w}}_{\lambda})$ is supported on $G /_{\sigma} G$.  Therefore, $F_1$ factors through $i_! i^! \Psi \to \Psi$.  By adjunction, this gives a map $i^* \underline{\chi}(\mathbb{L}^{\dot{w}}_{\lambda}) \to i^! \Psi$, which becomes $0$ upon applying the quotient functor $Q$.  But $Q$ is fully faithful when restricted to $G /_{\sigma} G$, so $F_1$ must be $0$.
\end{proof}

\begin{theorem}\label{thm:injectivity}
Suppose $\phi(x,\xi)$ is not in $\mathscr{J}_{\mathfrak{o}_{\textbf{c}}}$.  If $w \in \textbf{c} \cap W \times \mathfrak{s}_{\infty}$ and $\pi \in H_{1,\underline{\chi}(\mathbb{L}^{\dot{w}})} \subset H_{\Reg,\underline{\chi}(\mathbb{L}^{\dot{w}}_{\lambda})}$ is nonzero, then $F_{1,\phi(x,\xi),\mathfrak{M}} \pi \ne 0$.
\end{theorem}

\begin{proof}
    Recall from the previous theorem that $F_{1,\phi(x,\xi),\mathfrak{M}}$ is the functor $\mathfrak{p}_! \epsilon^*$ applied to the map $F_1$ from Equation~\eqref{eqn:F1 map}.  After applying $Q$, we now see that $F_1$ is nonzero, and in fact injective as it is nonzero map on each simple subobject.  The object $\Psi(\mathcal{S}_{\Reg} \boxtimes \underline{\chi}(\mathbb{L}^{\dot{w}}_{\lambda}))$ also has an ascending weight filtration, which means we have a subobject $\Psi(\mathcal{S}_{\Reg} \boxtimes \underline{\chi}(\mathbb{L}^{\dot{w}}_{\lambda}))^{\le 0}$ which is mixed of weight $\le 0$ and a quotient $\Psi(\mathcal{S}_{\Reg} \boxtimes \underline{\chi}(\mathbb{L}^{\dot{w}}_{\lambda}))^{\ge 1}$ which is mixed of weight $\ge 1$.  To summarize, we may consider the composition
    \[
    \begin{tikzcd}
    \Psi(\mathcal{S}_{1} \boxtimes \underline{\chi}(\mathbb{L}^{\dot{w}}_{\lambda})) \arrow{r}{F_{1,\phi(x,\xi),\mathfrak{M}}}& \Psi(\mathcal{S}_{\Reg} \boxtimes \underline{\chi}(\mathbb{L}^{\dot{w}}_{\lambda})) \arrow{r}& \Psi(\mathcal{S}_{\Reg} \boxtimes \underline{\chi}(\mathbb{L}^{\dot{w}}_{\lambda}))^{\ge 1},
    \end{tikzcd}
    \]
    which is a map from a sheaf pure of weight $0$ to a sheaf mixed of weight $\ge 1$, and therefore must be $0$.  This means that the map $F_{1,\phi(x,\xi),\mathfrak{M}}$ lifts to a map
    \[ F_{1,\phi(x,\xi),\mathfrak{M}}^{\le 0} \colon \Psi(\mathcal{S}_{1} \boxtimes \underline{\chi}(\mathbb{L}^{\dot{w}}_{\lambda})) \to \Psi(\mathcal{S}_{\Reg} \boxtimes \underline{\chi}(\mathbb{L}^{\dot{w}}_{\lambda}))^{\le 0}, \]
    which is injective.  Let $C^{\le 0}$ be the cokernel of this map, which is mixed of weight $\le 0$.
    
    To show that $F_{1,\phi(x,\xi),\mathfrak{M}}$ is injective after pulling back by $\epsilon$ and pushing forward by $\mathfrak{p}_!$, consider the 5-term exact sequence
    \begin{equation}
    \begin{tikzcd}
    0 \arrow{r} & R^{-1} \mathfrak{p}_! \mathscr{F}_{\Reg,\underline{\chi}(\mathbb{L}^{\dot{w}}_{\lambda})}^{\le 0} \arrow{r}& R^{-1} \mathfrak{p}_! \epsilon^* C^{\le 0} \arrow{dl}{\delta}& \\
    & H^0_{\{0\},1,\underline{\chi}(\mathbb{L}^{\dot{w}}_{\lambda})} \arrow{r}{\mathfrak{p}_! F_1}& R^{0} \mathfrak{p}_! \mathscr{F}_{\Reg,\underline{\chi}(\mathbb{L}^{\dot{w}}_{\lambda})}^{\le 0} \arrow{r}& R^0 \mathfrak{p}_! \epsilon^* C^{\le 0} \arrow{r} & 0.
    \end{tikzcd}
    \end{equation}
    This sequence shows that $F_{1,\phi(x,\xi),\mathfrak{M}} \pi = 0$ if and only if $\pi$ lives in the image of the boundary map $\delta$.  On the other hand, by applying Deligne's theorem on weights \cite{deligne1980weilii} \cite{bbd}, $R^{-1} \mathfrak{p}_! C^{\le 0}$ is (ind-)mixed of weight $\le -1$.  But $H^0_{\{0\},\Reg,\underline{\chi}(\mathbb{L}^{\dot{w}}_{\lambda})}$ is (ind-)pure of weight $0$ because it is the direct image of a perverse sheaf (in fact a local system), pure of weight $0$, along a smooth, proper $0$-dimensional stack $\Bun_G(\mathbb{F}_q)$.  Therefore, the boundary map is $0$.
\end{proof}

\subsection{Action of the framed excursion algebra}

For the remainder of this section we fix $G$ a split reductive group over $\mathbb{F}_q$, $K$ a function field of a curve $C$ over $\mathbb{F}_q$, and $Q = \prod_{x \in |C|} Q_x$ a compact open subgroup of $G(\mathbb{A}_K)$.  We furthermore fix a closed $\mathbb{F}_q$-point $y \in C(\mathbb{F}_q)$ and $Q_y = L^{\circ} G$, the unipotent radical of $G(\mathcal{O}_y)$.  Let $G(\mathbb{F}_{q})$ be the Levi quotient $L^+_y G / L^{\circ}_y G$.

For $\zeta \in \Rep(G(\mathbb{F}_q))$, we define $\mathcal{H}_{I,V,\zeta} = \mathfrak{p}_! \mathscr{F}_{I,V,\zeta}$ over $C^I$.  We also define $H_{I,V,\zeta}$ as the stalk at the geometric generic point $\mathcal{H}_{I,V,\zeta} |_{\overline{\eta_I}}$.  There exist constant term maps along each proper parabolic subgroup of $G$.

Following \cite{xue2020cuspidal}, we can define $H^{cusp}_{I,V,\zeta}$ as the intersection of the kernels of all the constant term maps, which Xue proves is the same as the Hecke-finite part of cohomology.  Under the decomposition
\[ H^{cusp}_{\{0\},\Reg,\zeta} = \bigoplus_V H^{cusp}_{\{0\},V,\zeta} \otimes \underline{V^*}, \]
the vector space $H^{cusp}_{\{0\},\Reg,\zeta}$ admits a right $\widehat{G}$-module structure.  It also admits an additional action by excursion operators.  For $f$ a regular function on $\widehat{G}^{I}$ defined by $f(g_1, \dots, g_{|I|}) = \langle \xi, g x \rangle$ and $(\gamma_1, \dots, \gamma_{|I|})$ an $I$-tuple of elements of $\Gal(\overline{K} / K)$, we define an excursion operator $F_{f,\gamma}$ by the composition
\begin{equation}
\begin{tikzcd}
    H^{cusp}_{\{0\},\Reg,\zeta} \arrow{r}{x} & H^{cusp}_{\{0\} \cup I, \Reg \boxtimes V,\zeta} \arrow{r}{1,\gamma} & H^{cusp}_{\{0\} \cup I, \Reg \boxtimes V,\zeta} \arrow{r}{\xi} & H^{cusp}_{\{0\},\Reg,\zeta},
\end{tikzcd}
\end{equation}
following the format of the preceding sections.

We are particularly interested in the space of cuspidal automorphic forms $H^{cusp}_{0,1} = H^{cusp}_{0,1,\Reg(G(\mathbb{F}_q))}$, where the representation in this case is the regular representation of $G(\mathbb{F}_q)$.  This space decomposes under the action of $G(\mathbb{F}_q)$ into $\zeta$-isotypic components.
\[ H^{cusp}_{0,1,\Reg(G(\mathbb{F}_q))} = \bigoplus_{\zeta} H^{cusp}_{0,1,\zeta} \otimes \underline{\zeta^*}. \]
Our strategy will be to characterize the Langlands parameters attached to a nonzero vector in a $\zeta$-isotypic component by computing the action of framed excursion operators on $H_{0,1,\zeta}$.

We define the framed excursion algebra $\mathscr{B}_{\zeta}$ as the algebra of endomorphisms of $H^{cusp}_{\{0\},\Reg,\zeta}$ generated by operators $F_{f,\gamma}$ for $f$ a regular function on $\widehat{G}^I$ and $\gamma$ an $I$-tuple of elements of the Weil group $\Weil(\overline{K} / K)$.  It can be viewed as living over space of functions on the affine space of representations of the Weil group \cite{lafforgue2018d}.  Moreover, in the case of finite central character, the Frobenius acts continuously, and the Weil action extends to an action of the absolute Galois group.

We note that there are some immediate enhancements of this setup.  First, we can consider a quotient of $\Weil(\overline{K} / K)$ that is unramified outside of $N \cup R$, which we denote $\Weil((C \setminus (N \cup R))_{\mathbb{F}_q}, \overline{\eta})$.  Furthermore, at the special place $y \in N$, we can consider the contribution of wild inertia, which is a normal subgroup $P_y \subset \Weil(C \setminus (N \cup R), \overline{\eta})$.  This gives us a quotient, which we will denote $\Weil^y(C \setminus (N \cup R), \overline{\eta})$.  Excursion operators defined for the group $\Weil(\overline{K} / K)$ will thus factor over the quotient $\Weil^y(C \setminus (N \cup R), \overline{\eta})$.

We will now consider a much smaller algebra.  Consider the local Galois group at the special place $y$, which canonically embeds in the global Galois group.  The local Galois group has an inertia subgroup $I_y$, which in turn has a Sylow pro-p subgroup $P_y$ (wild inertia).  The quotient $I_y / P_y$ has a topological generator, which we shall denote $\gamma$.  We call $\gamma$ the tame generator, and given this choice, the monodromy of nearby cycles $\mathfrak{M}$ in Section~\ref{sec:truncatedexcursion} is an automorphism, not just the action of the local Galois group, and this automorphism represents the action of the element $\gamma$ of the tame Galois group.

We also have a conjugacy class distinguished by $\zeta$.  Recall that there is a unique $W$-orbit of $\lambda$ and two sided cell $\textbf{c}$ attached to $\zeta$ for which there is a $w \in \textbf{c} \cap W_{\lambda}$ such that $\zeta \subset \underline{\chi}(\mathbb{L}^{\dot{w}}_{\lambda})$ is a summand of $G$-equivariant sheaves on $G$, i.e., a summand as $G(\mathbb{F}_q)$-representations.  The pair $(\lambda, \textbf{c})$ determines a semisimple element of $\widehat{T}$ and unipotent conjugacy class in the endoscopic group $\widehat{H}$, and the product of these gives a conjugacy class in $\widehat{G}$ which is independent of the representative $\lambda$ of the $W$-orbit.  We denote this conjugacy class by $\mathfrak{o}_{\zeta}$ and the ideal sheaf in $\mathscr{O}(\widehat{G})$ defined by its Zariski closure in $\widehat{G}$ as $\mathscr{J}_{\mathfrak{o}_{\zeta}}$.

\begin{theorem}\label{thm: annihilator is Jozeta}
    Let $\pi \ne 0$ be an automorphic form in $H^{cusp}_{\{0\},1,\zeta} \subset H^{cusp}_{\{0\},\Reg,\zeta}$.  Then $F_{\cdot,\gamma}$ gives a map $\mathscr{O}_{\widehat{G}} \to \mathscr{B}_{\zeta}$.  The annihilator of $\pi$ in $\mathscr{O}_{\widehat{G}}$ is $\mathscr{J}_{\mathfrak{o}_{\zeta}}$.
\end{theorem}

\begin{proof}
    Let us set up notation by writing $\zeta$ as a summand of $\underline{\chi}(\mathbb{L}^{\dot{w}}_\lambda)$.  In this case, we may identify excursion operators on $H^{cusp}_{\{0\},1,\zeta}$ with the summand of excursion operators on $H^{cusp}_{\{0\},1,\underline{\chi}(\mathbb{L}^{\dot{w}}_{\lambda})}$.  We work with the latter excursion operators for the remainder of the proof.

    We first prove that every element $f \in \mathscr{J}_{\mathfrak{o}_{\gamma}}$ satisfies $F_{f,\gamma} \pi = 0$.  Writing $f(g) = \phi(x,\xi)$, the operator $F_{f,\gamma}$ is the same as $F_{\phi(x,\xi),\mathfrak{M}}$ in Equation~\eqref{eqn:H-excursion} since the isomorphism $\mathfrak{p}_! \Psi \simeq \Psi \mathfrak{p}_!$ in Theorem~\ref{thm:xuezorro} is Galois equivariant.  But we know that $F_{\phi(x,\xi),\mathfrak{M}} \pi = 0$, so the annihilator of $\pi$ contains $\mathscr{J}_{\mathfrak{o}_{\gamma}}$.

    To prove the reverse inclusion, consider a function $f'$ not in $\mathscr{J}_{\mathfrak{o}_{\gamma}}$.  By Theorem~\ref{thm:injectivity}, $F_{f,\gamma} \pi \ne 0$ and the annihilator cannot contain $f$.

\end{proof}

\subsection{Compatibility of the excursion algebra action on $H_{\{0\},1}$ and $H_{\{0\},1,\zeta}$}

If $\zeta \in \Rep(G(\mathbb{F}_q))$ is irreducible and $v \in \underline{\zeta^*}$ is a vector in the underlying vector space of the contragredient representation, we may produce a map $v \colon \zeta \to \Reg(G(\mathbb{F}_q))$ to the regular representation of $G(\mathbb{F}_q)$.  On the level of cohomology sheaves of shtukas this produces a map
\[v \colon \mathcal{H}_{I,V,\zeta} \to \mathcal{H}_{I,V,\Reg(G(\mathbb{F}_q))}.\]
This map is compatible with passing to cuspidal subspaces, which can be seen as the action of $G(\mathbb{F}_q)$ is an action by elements of the Hecke algebra, and there is an alternate characterization of the cuspidal subspace as the space of Hecke-finite vectors.  The goal of this section is to show that this map $v$ does not affect Langlands parameters, which is to say that the map $v$ commutes with the operations of creation, annihilation, and coalescence and is Galois equivariant.

\begin{proposition}
    The functor $(I,V,\zeta) \mapsto \mathcal{S}_{I,V} \boxtimes \zeta$ defines a functor
    \begin{equation}
    \begin{tikzcd}
        \Rep(\widehat{G}^{I}) \otimes \Rep(H(\mathbb{F}_q)) \arrow{r}& D(\Cht\mathcal{R}^{ad}_{G,I,N}|_{(C \setminus (N \cup R))^{I}}, \qlbar).
    \end{tikzcd}
    \end{equation}
\end{proposition}

\begin{proof}
    This is immediate from the product structure of $\Gr_{G,I}^{\triv(N)}$ over $(C \setminus (N \cup R))^I$ and geometric Satake.
\end{proof}

\begin{proposition}
    Let $v \colon \zeta \to \pi$ be a map in $\Rep(H(\mathbb{F}_q))$.  Let $\overline{x}$ and $\overline{y}$ be two geometric points with $\mathfrak{sp} \colon \overline{y} \to \overline{x}$ a specialization in the \'etale topos of $(C \setminus (N \cup R))^I$.  Then the following diagram commutes:
    \begin{equation}
        \begin{tikzcd}
            \mathcal{H}_{I,V,\zeta}|_{\overline{x}} \arrow{r}{\mathfrak{sp}^*} \arrow{d}{v} & \mathcal{H}_{I,V,\zeta}|_{\overline{y}} \arrow{d}{v} \\
            \mathcal{H}_{I,V,\pi}|_{\overline{x}} \arrow{r}{\mathfrak{sp}^*} & \mathcal{H}_{I,V,\pi}|_{\overline{y}}
        \end{tikzcd}
    \end{equation}
\end{proposition}

\begin{proof}
    The maps $v$ come from a morphism of sheaves
    \[\mathfrak{p}_! (\epsilon^{ad})^*(\mathcal{S}_{I,V} \boxtimes \zeta) \to \mathfrak{p}_! (\epsilon^{ad})^*(\mathcal{S}_{I,V} \boxtimes \pi),\]
    so this commutative diagram is just saying that taking stalks is functorial.
\end{proof}

\begin{remark}
    The above proposition when combined with the previous one shows compatibility with coalescence, since coalescence for a map $I \to J$ of finite sets involves specializations from the geometric generic point $\overline{\eta_I}$ to the geometric generic point of the generalized diagonal $\Delta_{I \to J}(\overline{\eta_J})$.

    Moreover, the above proposition also shows Galois equivariance, since Galois actions on stalks arises from automorphisms of the corresponding geometric points.
\end{remark}

\begin{proposition}
Excursion operators commute with morphisms $v \colon \zeta \to \pi$.
\end{proposition}

\begin{proof}
    We will prove this for unframed excursion operators.  The verification for framed excursion operators is similar.  Let $I$ and $J$ be finite sets and $V \in \Rep(\widehat{G}^I)$ and $W \in \Rep(\widehat{G}^J)$.  Let $J_1$ and $J_2$ be two identical copies of $J$.  For any finite set $K$, let $\overline{\eta_K}$ be the geometric generic point of $C^K$ and fix a specialization $\mathfrak{sp} \colon \overline{\eta_{I \cup J_1 \cup J_2}} \to \Delta(\overline{\eta_{I \cup J}})$ under the map $I \cup J_1 \cup J_2 \to I \cup J$.  The morphism $v$ produces a commutative diagram from one excursion operator to the other, depicted in Figure~\ref{fig: excursion compatibility}.
    \begin{figure}
    \begin{center}
        \begin{tikzcd}
            \mathcal{H}_{V \boxtimes 1,\zeta}|_{\overline{\eta_{I\cup J}}} \arrow{d}{\mathscr{C}^{\sharp}} \arrow{r}{v} & \mathcal{H}_{V \boxtimes 1,\pi}|_{\overline{\eta_{I\cup J}}} \arrow{d}{\mathscr{C}^{\sharp}} \\
            \mathcal{H}_{V \boxtimes W \boxtimes W^*,\zeta}|_{\Delta(\overline{\eta_{I\cup J}})} \arrow{d}{\mathfrak{sp}^*} \arrow{r}{v} & \mathcal{H}_{V \boxtimes W \boxtimes W^*,\pi}|_{\Delta(\overline{\eta_{I\cup J}})} \arrow{d}{\mathfrak{sp}^*} \\
            \mathcal{H}_{V \boxtimes W \boxtimes W^*,\zeta}|_{\overline{\eta_{I\cup J_1 \cup J_2}}} \arrow{d}{(\gamma_j)_{j \in J_1}} \arrow{r}{v} & \mathcal{H}_{V \boxtimes W \boxtimes W^*,\pi}|_{\overline{\eta_{I\cup J_1 \cup J_2}}} \arrow{d}{(\gamma_j)_{j \in J_1}} \\
            \mathcal{H}_{V \boxtimes W \boxtimes W^*,\zeta}|_{\overline{\eta_{I\cup J_1 \cup J_2}}} \arrow{d}{(\mathfrak{sp}^*)^{-1}} \arrow{r}{v} & \mathcal{H}_{V \boxtimes W \boxtimes W^*,\pi}|_{\overline{\eta_{I\cup J_1 \cup J_2}}} \arrow{d}{(\mathfrak{sp}^*)^{-1}} \\
            \mathcal{H}_{V \boxtimes W \boxtimes W^*,\zeta}|_{\Delta(\overline{\eta_{I\cup J}})} \arrow{d}{\mathscr{C}^{\flat}} \arrow{r}{v} & \mathcal{H}_{V \boxtimes W \boxtimes W^*,\pi}|_{\Delta(\overline{\eta_{I\cup J}})} \arrow{d}{\mathscr{C}^{\flat}} \\
            \mathcal{H}_{V \boxtimes 1,\zeta}|_{\overline{\eta_{I\cup J}}} \arrow{r}{v} & \mathcal{H}_{V \boxtimes 1,\pi}|_{\overline{\eta_{I\cup J}}}
        \end{tikzcd}
    \caption{Naturality of excursion operators with respect to maps in $\Rep(H(\mathbb{F}_q))$.}\label{fig: excursion compatibility}
    \end{center}
    \end{figure}
\end{proof}

\subsection{General annihilators in the framed excursion algebra}

We now want to explain the significance of annihilator ideals in the framed or unframed excursion algebra for Vincent Lafforgue's automorphic-to-Galois direction of the global Langlands correspondence.  Recall that for any fixed level, the space of cuspidal automorphic forms is finite-dimensional and decomposes as a direct sum
\begin{equation}\label{eqn: isotypic decomposition}
H^{cusp}_{\{0\},1} \cong \bigoplus_{\rho} \mathfrak{H}_{\rho},
\end{equation}
where $\rho$ runs over semisimple Langlands parameters.  If $\pi \in \mathfrak{H}_{\rho}$ is a nonzero cuspidal automorphic form, we say that it is attached to the semisimple parameter $\rho$.  Such semisimple parameters arise from a smaller algebra of excursion operators than we have considered in previous sections, namely, the operators $\mathscr{B}^1 \subset \mathscr{B}$ such that $f \in \mathscr{O}(\widehat{G} \backslash \widehat{G}^I / \widehat{G})$.  It is a fact that such operators $\mathscr{B}^1$ preserve $H_{\{0\},1} \subset H_{\{0\},\Reg}$ and their characters can be put in correspondence with semisimple Langlands parameters.  A priori, these give semisimple representations of $\Weil(\overline{K} / K)$, but it can be shown that the Frobenius acts continuously and thus we get a representation of $\Gal(\overline{K} / K)$.  To aid our development, we introduce some notation for ``framed'' and ``unframed'' annihilators in excursion algebras.

Let us recast V. Lafforgue's correspondence in terms of annihilators in the excursion algebra.  Consider the action of $\widehat{G} \times \widehat{G}$ on $\widehat{G}^I$ where the first copy acts by multiplication on the left and the second by multiplication on the right.  The framed excursion operator $f \in \mathscr{O}(\widehat{G}^{n})^{\widehat{G} \times \widehat{G}}$ preserves each $H_{I,V}$ and in particular preserves sends $H^{cusp}_{\{0\},1}$ to itself.  This produces excursion operators $F_{f, (\gamma_i)_{i \in I}}$ in the sense of \cite{lafforgue2018chtoucas}.  Running over all possible choices of $f \in \mathscr{O}(\widehat{G}^{n})^{\widehat{G} \times \widehat{G}}$, varying $n$ and $(\gamma_1, \dots, \gamma_n)$, produces the commutative algebra $\mathscr{B}^1$, which we call the excursion algebra, and the characters of the excursion algebra are in 1-1 correspondence with semisimple Langlands parameters.

\begin{definition}
    For a tuple $(\gamma_1, \dots, \gamma_n) \in \Gal(\overline{K} / K)^n$, let $\mathscr{B}_{(\gamma_1, \dots, \gamma_n)}\pi$ denote the module generated by $F_{f, (\gamma_1, \dots, \gamma_n)} \pi$ for all $f \in \mathscr{O}(\widehat{G}^n)$.  Let $\Ann^{\Box}_{(\gamma_1, \dots, \gamma_n)}(\pi)$ denote the annihilator of the module $\mathscr{B}_{(\gamma_1, \dots, \gamma_n)}\pi$ under the action of $\mathscr{O}(\widehat{G}^n)$.  Let $\xi^{\Box}_{\pi}(\gamma_1, \dots, \gamma_n)$ be the support of $\Ann^{\Box}_{(\gamma_1, \dots, \gamma_n)}(\pi)$ as a subset of $\widehat{G}^n$.

    Let $\mathscr{B}^1_{(\gamma_1, \dots, \gamma_n)} \pi$ denote the module generated by $F_{f, (1, \gamma_1, \dots, \gamma_n)} \pi$ for all $f \in \mathscr{O}(\widehat{G}^{n+1})^{\widehat{G} \times \widehat{G}}$.  Let $\Ann_{(\gamma_1, \dots, \gamma_n)}(\pi)$ denote the annihilator of the module $\mathscr{B}^1_{(\gamma_1, \dots, \gamma_n)} \pi$ under the action of $\mathscr{O}(\widehat{G}^{n+1})^{\widehat{G} \times \widehat{G}}$.  Let $\xi_{\pi}(\gamma_1, \dots, \gamma_n)$ be the support of $\Ann_{(\gamma_1, \dots, \gamma_n)}(\pi)$ as a subset of $\widehat{G}^{n+1} // (\widehat{G} \times \widehat{G})$.
\end{definition}

Suppose $\pi \in \mathfrak{H}_{\rho}$ corresponding to the character $\nu$ of the excursion algebra and let $\rho$ be a semisimple $\widehat{G}(\qlbar)$-valued Galois representation corresponding to $\nu$.  We briefly recall how $\rho$ can be constructed from $\nu$.  Following Lafforgue \cite[Proof of Proposition~5.7]{lafforgue2014introduction} \cite[Proof of Proposition~11.7]{lafforgue2018chtoucas}, recall the isomorphism $\widehat{G}^{n+1} // (\widehat{G} \times \widehat{G}) \cong \widehat{G}^n // \widehat{G}$, where the latter quotient is the action by conjugation.  Let $\xi_{n}(\gamma_1, \dots, \gamma_n) = \xi_{\pi}(\gamma_1, \dots, \gamma_n)$ denote the $\qlbar$-point in $\widehat{G}^n // \widehat{G}$ corresponding to the character $f \mapsto \nu(F_{f, (1, \gamma_1, \dots, \gamma_n)}) \colon \mathscr{O}(\widehat{G}^{n+1})^{\widehat{G} \times \widehat{G}} \to \qlbar$.  By Richardson's theorem, points in $\widehat{G}^n // \widehat{G}$ correspond to semisimple conjugacy classes, so and let $\overline{\xi} \subset \widehat{G}^n$ be the semisimple conjugacy class attached to a point $\xi \in \widehat{G}^n // \widehat{G}$.

Now assume that $(\gamma_1, \dots, \gamma_n)$ are chosen so that $\rho(\gamma_1), \dots \rho(\gamma_n) \in \widehat{G}(\qlbar)$ generate a Zariski dense subset of the Zariski closure of the image $\rho(\Gal(\overline{K} / K))$.  Then for any $\gamma \in \Gal(\overline{K} / K)$, there is a unique $g$ such that $g$ lives in the Zariski closure of the group generated by $\rho(\gamma_1), \dots \rho(\gamma_n)$ and $(\rho(\gamma_1), \dots, \rho(\gamma_n), g) \in \overline{\xi_{n+1}(\gamma_1, \dots, \gamma_n, \gamma)} \subset \widehat{G}^{n+1}$, and this unique $g$ is equal to $\rho(\gamma)$.

The character $\nu$ of the excursion algebra can be considered as giving maximal ideals $\nu_{(\gamma_1, \dots, \gamma_n)} \subset \mathscr{O}(\widehat{G}^{n+1})^{\widehat{G} \times \widehat{G}}$ for any $n$-tuple $(\gamma_1, \dots, \gamma_n)$.  This maximal ideal $\nu_{(\gamma_1, \dots, \gamma_n)}$ is characterized by the property that $\nu_{(\gamma_1, \dots, \gamma_n)}^n \subset \Ann_{(1, \gamma_1, \dots, \gamma_n)}(\pi) \subset \nu_{(\gamma_1, \dots, \gamma_n)}$.

\begin{proposition}
\begin{enumerate}
    \item The framed and unframed annihilators are related by
    \[ \Ann_{(\gamma_1, \dots, \gamma_n)}(\pi) = \Ann^{\Box}_{(\gamma_1, \dots, \gamma_n)}(\pi) \cap \mathscr{O}(\widehat{G}^{n+1})^{\widehat{G} \times \widehat{G}}. \]
    \item Let $I \to J$ be a map of finite sets, and consider the corresponding map $\widehat{G}^J \to \widehat{G}^{I}$ and ring map $\mathscr{O}(\widehat{G}^I) \to \mathscr{O}(\widehat{G}^J)$.  If $I \to J$ is an inclusion the ring map is also an injection and
    \[ \Ann^{\Box}_{(\gamma_i)_{i \in I}}(\pi) = \Ann^{\Box}_{(\gamma_j)_{j \in J}}(\pi) \cap \mathscr{O}(\widehat{G}^I) \]
    \item The ideal $\Ann^{\Box}_{(\gamma_1, \dots, \gamma_n)}(\pi)$ is invariant under $\widehat{G}$-conjugation.  Therefore, the variety $\xi^{\Box}_{(\gamma_1, \dots, \gamma_n)}$ is invariant under $\widehat{G}$-conjugation.
\end{enumerate}
\end{proposition}

\begin{proof}
\begin{enumerate}
    \item This is a tautology because the framed excursion algebra acts in the same way as the excursion algebra, and is simply a larger ring.  So elements of the smaller ring that annihilate $\pi$ are precisely those that annihilate $\pi$ in the larger ring that live in the smaller ring.
    \item This is identical to the above but relating the two rings by $f \mapsto f \otimes 1$ where $1 \in \mathscr{O}(\widehat{G}^{J \setminus I})$.
    \item The second statement follows from the first.  So it suffices to check that the annihilator of $\pi$ must be conjugation invariant in the sense that if $F_{f, (\gamma_1, \dots, \gamma_n)} \pi = 0$, then $f^h(g) = f(h^{-1}gh)$ also satisfies $F_{f^h, (\gamma_1, \dots, \gamma_n)} \pi = 0$.  This results from the general fact that $F_{f^h,(\gamma_1, \dots, \gamma_n)} = hF_{f,(\gamma_1, \dots, \gamma_n)} h^{-1}$ where the action of $F_{f,\gamma}$ is by excursion operators and the action of $h,h^{-1} \in \widehat{G}$ is by the right action on $H^{cusp}_{\Reg}$ \cite[Section~6]{lafforgue2018d}.  If $F_{f,\gamma} \pi = 0$ and $\pi \in H^{cusp}_{1} \subset H^{cusp}_{\Reg}$, then
    \[F_{f^h,(\gamma_1, \dots, \gamma_n)} \pi = h F_{f,(\gamma_1, \dots, \gamma_n)} h^{-1} \pi = h F_{f,(\gamma_1, \dots, \gamma_n)} \pi = 0.\]
\end{enumerate}
\end{proof}

On supports, the above proposition produces the commutative diagram
\begin{equation}\label{eqn: excursion annihilator diagram}
    \begin{tikzcd}
        \xi_{\pi}((\gamma_i)_{i \in I}) \arrow{d} & \xi_{\pi}^{\Box}((\gamma_i)_{i \in I}) \arrow{r} \arrow{l} \arrow{d} & \xi_{\pi}^{\Box}((\gamma_j)_{j \in J}) \arrow{d} \\
        \widehat{G}^I // \widehat{G} & \widehat{G}^I \arrow{r} \arrow{l} & \widehat{G}^J
    \end{tikzcd}
\end{equation}
for finite sets $J \subset I$, where the framed supports $\xi^{\Box}$ are a union of orbits under $\widehat{G}$-conjugation.

\subsection{Proof of the first part of Theorem~\ref{thm:depth0}}

\begin{proposition}
    For any $\gamma \in \Gal(\overline{K} / K)$ and $\pi \in \mathfrak{H}_{\rho}$ nonzero corresponding to some semisimple Galois representation $\rho$, and if $\xi^{\Box}_\pi(\gamma)$ is the closure of a conjugacy class $\mathfrak{o}_{\pi}$, then $\rho(\gamma)$ lives in the closure of $\mathfrak{o}_{\pi}$.
\end{proposition}

\begin{proof}
    Consider the diagram~\eqref{eqn: excursion annihilator diagram} with $J = \{ 0 \}$ and $\gamma_0 = \gamma$.  By the commutativity of the diagram, $\xi^{\Box}(\gamma_1, \dots, \gamma_n, \gamma)$ and its Zariski closure are contained in the preimage of the closure of this conjugacy class.  Choosing $(\gamma_1, \dots, \gamma_n)$ as in the construction of the Langlands parameter, $\overline{\xi_{\pi}(\gamma_1, \dots, \gamma_n, \gamma)}$ is the unique closed $\widehat{G}$ orbit in the preimage of $\xi_{\pi}(\gamma_1, \dots, \gamma_n, \gamma)$ along the GIT quotient map and so must be contained in the Zariski closure of $\xi^{\Box}(\gamma_1, \dots, \gamma_n, \gamma)$.  Since $(\rho(\gamma_1), \dots, \rho(\gamma_n),\rho(\gamma))$ lives in this preimage, which in turn lives in the preimage of the closure, $\widehat{G}^n \times \overline{\mathfrak{o}_{\pi}}$, we conclude that $\rho(\gamma)$ lives in the closure of $\mathfrak{o}_{\pi}$.
\end{proof}

By combining this proposition with Theorem~\ref{thm: annihilator is Jozeta}, we arrive at the corollary, which is the first part of Theorem~\ref{thm:depth0}.

\begin{corollary}\label{cor: isotypic component closure zeta}
    Let $\rho$ be a semisimple Langlands parameter $\rho$ attached to a nonzero cuspidal automorphic form $\pi \in H^{cusp}_{\{0\},1}$.  If $\zeta$ is an irreducible representation of $G(\mathbb{F}_q)$ and $\pi$ lives in the $\zeta$-isotypic component of cuspidal automorphic forms under the $G(\mathbb{F}_q)$-action, then $\rho$ is trivial on wild inertia at $y$.  Furthermore, if $\gamma$ is the tame generator at the point $y$, $\rho(\gamma) \in \overline{\mathfrak{o}_{\zeta}}$.
\end{corollary}

\subsection{The elliptic case}

A special case is when $\rho$ is elliptic, which has been studied by Lafforgue and Zhu \cite{lafforgue2018d}.  We can consider the image $\rho(\Gal(\overline{K} / K))$ in $\widehat{G}(\qlbar)$ and denote the centralizer of the image by $\mathfrak{S}_{\rho}$.  For a split group, if $\mathfrak{S}_{\rho} / Z_{\widehat{G}}$ is finite, then we say that $\rho$ is elliptic.

We recall some general facts about the map $q \colon \widehat{G}^n \to \widehat{G}^n // \widehat{G}$ by conjugation.  For any $\qlbar$-point $x \in \widehat{G}^n // \widehat{G}$, we consider the preimage $\overline{x} \subset \widehat{G}^n$, which contains a unique closed orbit given by semisimple $n$-tuples, and moreover, any $\widehat{G}$-orbit in the preimage contains this closed orbit in its closure \cite[Section~1.3.2, 3.6]{richardson1988conjugacy}.  Specializing to the case $\xi_{\pi}(\gamma_1, \dots, \gamma_n)$, from the diagram \eqref{eqn: excursion annihilator diagram} and using the fact that $\xi^{\Box}_{\pi}(\gamma_1, \dots, \gamma_n)$ must contain an orbit $\widehat{G} \cdot y$ for some $y \in \widehat{G}^I$, we find that $\xi^{\Box}_{\pi}(\gamma_1, \dots, \gamma_n)$ must contain the semisimple conjugacy class attached to $\xi_{\pi}(\gamma_1, \dots, \gamma_n)$ in its closure.

\begin{lemma}
Suppose $\rho$ is an elliptic Langlands parameter attached to $\pi$ and let the Zariski closure of the group generated by $\gamma_1, \dots, \gamma_n$ contain the image of $\Gal(\overline{K} / K)$.  Then $\xi^{\Box}_{\pi}(\gamma_1, \dots, \gamma_n) = q^{-1}(\xi_{\pi}(\gamma_1, \dots, \gamma_n))$ where $q$ is the quotient map $\widehat{G}^n \to \widehat{G}^n // \widehat{G}$ by conjugation.
\end{lemma}

\begin{proof}
    The diagram simplifies since there can only be one such orbit in the fiber over $\xi_{\pi}(\gamma_1, \dots, \gamma_n)$.  To see this, note that any orbit $\widehat{G} \cdot y$ is of the form $\widehat{G} / \widehat{G}_y$ for the stabilizer $\widehat{G}_y$ of $y$.  This stabilizer must be contained in the centralizer $C(\gamma_1, \dots, \gamma_n)$ of a point in the semisimple conjugacy class and contains the center $Z_{\widehat{G}}$.  Moreover, if $y$ does not itself live in the semisimple conjugacy class then this stabilizer must have strictly smaller dimension than $C(\gamma_1, \dots, \gamma_n)$.  This is impossible if $C(\gamma_1, \dots, \gamma_n) / Z_{\widehat{G}}$ is finite.  By construction, if $(\gamma_1, \dots, \gamma_n)$ are such that the Zariski closure of the group generated by $\rho(\gamma_1), \dots, \rho(\gamma_n)$ contains the image $\rho(\Gal(\overline{K} / K))$, then we must have $C(\gamma_1, \dots, \gamma_n) = \mathfrak{S}_{\rho}$ taking the stabilizer at the point $(\rho(\gamma_1), \dots, \rho(\gamma_n)) \in \widehat{G}^n$.  For any such choice of $(\gamma_1, \dots, \gamma_n)$, we find that $\xi^{\Box}_{\pi}(\gamma_1, \dots, \gamma_n)$ coincides with the preimage of $\xi_{\pi}(\gamma_1, \dots, \gamma_n)$.
\end{proof}

With this preparation, we can prove the second part of Theorem~\ref{thm:depth0}.

\begin{corollary}\label{cor:Langlands param}
    Let $\rho$ be a semisimple global Langlands parameter attached to a cuspidal automorphic form $\pi \ne 0$ in the $\zeta$-isotypic component of the space of automorphic forms as in Corollary~\ref{cor: isotypic component closure zeta}.  Let $\gamma$ be a choice of tame generator at $y$ as above.  If $\rho$ is an elliptic Langlands parameter, then the image $\rho(\gamma)$ is contained in the conjugacy class $\mathfrak{o}_{\gamma}$ (not just in the closure).
\end{corollary}

\begin{proof}
    By the previous lemma, the $(n+1)$-tuple $(\rho(\gamma_1), \dots, \rho(\gamma_n), \rho(\gamma))$ lives in $\xi^{\Box}_{\pi}(\gamma_1, \dots, \gamma_n, \gamma)$, which consists of a single closed orbit.  Note that
    \[\mathscr{O}(\widehat{G}) / \Ann^{\Box}_{\pi}(\gamma) \to \mathscr{O}(\widehat{G}^{n+1}) / \Ann^{\Box}_{\pi}(\gamma_1, \dots, \gamma_n, \gamma)\]
    is injective, so the generic point of the closed orbit $\xi^{\Box}_{\pi}(\gamma_1, \dots, \gamma_n, \gamma)$ maps to the generic point of $\mathfrak{o}_{\zeta}$ along the projection $\widehat{G}^{n+1} \to \widehat{G}$.  If $\rho(\gamma) \not \in \mathfrak{o}_{\zeta}$, then the $\widehat{G}$-orbit of $(\rho(\gamma_1), \dots, \rho(\gamma_n), \rho(\gamma))$ would be contained in a conjugacy class on the boundary $\mathfrak{o}_{\zeta}$, which contradicts the observation that the generic point must map to the generic point.
\end{proof}

For general depth, we expect a correspondence extending the relationship between depth in the Moy-Prasad filtration and depth in the upper numbering filtration.  For example, we expect the following consequence:

\begin{conjecture}
Let $N = \sum a_y y$ for nonnegative rational numbers $a_y$.  The local Galois action given by monodromy of nearby cycles of sheaves $S_V \boxtimes \zeta$ on restricted shtukas $\Cht\mathcal{R}_{G,N,I}$ along the fiber over $y$ is trivial on the $a_y$th ramification subgroup in the upper numbering filtration.
\end{conjecture}

Such conjectures, together with corresponding statements about $\Psi$-factorizability, would imply additional local-global compatibilities in the global Langlands correspondence for function fields.

\printbibliography

\end{document}